\documentclass[12pt,oneside,reqno]{amsart}

 \usepackage{txfonts}
\usepackage{bbm}
\usepackage{cases}
\usepackage{amsmath}
 \usepackage{mathrsfs}
\usepackage{stmaryrd}
\usepackage{amsfonts}
 \usepackage{enumerate,amsmath,amssymb,amsthm}

\usepackage{hyperref}

\numberwithin{equation}{section}

\newcommand{\be}{\begin{eqnarray}}
\newcommand{\ee}{\end{eqnarray}}
\newcommand{\ce}{\begin{eqnarray*}}
\newcommand{\de}{\end{eqnarray*}}
\newtheorem{theorem}{Theorem}[section]
\newtheorem{lemma}[theorem]{Lemma}
\newtheorem{remark}[theorem]{Remark}
\newtheorem{definition}[theorem]{Definition}
\newtheorem{proposition}[theorem]{Proposition}
\newtheorem{Examples}[theorem]{Example}
\newtheorem{corollary}[theorem]{Corollary}

\DeclareMathOperator*{\essinf}{essinf}

\def\eps{\varepsilon}

\def\p{\partial}

\def\[{{\Big[}}
\def\]{{\Big]}}
\def\<{{\langle}}
\def\>{{\rangle}}
\def\({{\Big(}}
\def\){{\Big)}}

\def\bx{{\mathbf{x}}}

\def\e{{\rm e}}

\def\dif{{\mathord{{\rm d}}}}

\def\Vol{\mathord{{\rm Vol}}}

\def\no{\nonumber}
\def\={&\!\!=\!\!&}
\def\bt{\begin{theorem}}
\def\et{\end{theorem}}
\def\bl{\begin{lemma}}
\def\el{\end{lemma}}
\def\br{\begin{remark}}
\def\er{\end{remark}}

\def\bd{\begin{definition}}
\def\ed{\end{definition}}
\def\bp{\begin{proposition}}
\def\ep{\end{proposition}}
\def\bc{\begin{corollary}}
\def\ec{\end{corollary}}
\def\bx{\begin{Examples}}
\def\ex{\end{Examples}}

\def\cB{{\mathcal B}}

\def\cJ{{\mathcal J}}

\def\m1{{\mathbbold{1}}}

\def\mD{{\mathbb D}}
\def\mE{{\mathbb E}}

\def\mI{{\mathbb I}}

\def\mK{{\mathbb K}}

\def\mN{{\mathbb N}}

\def\mP{{\mathbb P}}

\def\mR{{\mathbb R}}

\def\mZ{{\mathbb Z}}

\def\sF{{\mathscr F}}

\def\sI{{\mathscr I}}

\def\sL{{\mathscr L}}

\def\geq{\geqslant}
\def\leq{\leqslant}

\allowdisplaybreaks

\begin{document}
\title{Heat kernels for non-symmetric diffusion operators with jumps  }

\date{}

\author{{Zhen-Qing Chen},\ \ {Eryan Hu},\ \ {Longjie Xie}
and  {Xicheng Zhang}
 }

\address{Zhen-Qing Chen:
Department of Mathematics, University of Washington, Seattle, WA 98195, USA\\
Email: zqchen@uw.edu
 }

\address{Eryan Hu:
Department of Mathematics, Beijing Institute of Technology,
Beijing, 100081, P.R.China\\
Email: eryanhu@gmail.com
 }
\address{Longjie Xie:
School of Mathematics and Statistics, Jiangsu Normal University,
Xuzhou, Jiangsu 221000, P.R.China\\
Email: xlj.98@whu.edu.cn
 }
\address{Xicheng Zhang:
School of Mathematics and Statistics, Wuhan University,
Wuhan, Hubei 430072, P.R.China\\
Email: XichengZhang@gmail.com
 }

\thanks{Research of XZ is partially supported by NNSFC grant of China (Nos. 11271294, 11325105).}

\begin{abstract}
  For $d\geq 2$, we prove the existence and uniqueness of heat kernels to the following time-dependent second order diffusion operator with jumps:
  $$
    \sL_t:=\frac{1}{2}\sum_{i,j=1}^d a_{ij}(t,x)\p ^2_{ij}+\sum_{i=1}^{d}b_i(t,x)\p_i+\sL^\kappa_t,
  $$
  where $a=(a_{ij})$ is a uniformly bounded, elliptic, and H\"older continuous matrix-valued function, $b$ belongs to some
 suitable Kato's class,
  and $\sL^\kappa_t$ is a non-local $\alpha$-stable-type operator with bounded kernel $\kappa$. Moreover, we establish sharp two-sided estimates, gradient estimate and fractional derivative estimate for the heat kernel under some mild conditions.

  \bigskip



\end{abstract}

\maketitle

\section{Introduction}

Let $C_0(\mR^d)$ be the Banach space of all continuous functions on $\mR^d$  vanishing at infinity
equipped with uniform norm,
and $C_c(\mR^d)$ the space of all continuous functions on $\mR^d$ with compact support.
 Let $\sL$ be a linear operator on $C_0(\mR^d)$
with domain Dom$(\sL)$. Suppose that $C^\infty_c(\mR^d)\subset\text{Dom}(\sL)$.
We say  $\sL$ satisfies a positive maximum principle if for all
$f\in C^\infty_c(\mR^d)$ reaching a positive maximum at point $x_0\in\mR^d$, then $\sL f(x_0)\leq 0$. The well-known Courr\`ege theorem states that
$\sL$ satisfies the positive maximum principle if and only if $\sL$ takes the following form
\begin{align}\label{e:1.1}
\begin{split}
\sL f(x)&=\frac{1}{2} \sum_{i,j=1}^d a_{ij}(x)\p ^2_{ij} f(x)+\sum_{i=1}^{d}b_i(x)\p_if(x)+c(x)f(x)\\
& \quad+\int_{\mR^d}\left(f(x+z)-f(x)-1_{\{|z|\leq 1\}}z\cdot\nabla f(x)\right)\mu_x(\dif z),
\end{split}
\end{align}
where $a=(a_{ij}(x))_{1\leq i,j\leq d}$ is a $d\times d$-symmetric positive definite matrix-valued measurable function on $\mR^d$,
$b(x): \mR^d\to\mR^d$, $c:\mR^d\to(-\infty,0]$ are measurable functions and $\mu_x(\dif z)$
is a family of L\'evy measures, with that $a,b,c,\mu$ enjoy some continuity with respect to $x$ (see \cite{Jacob}).
On the other hand, from the probabilistic viewpoint, consider the following SDE with jumps:
\begin{align} \label{e:1.2}
\begin{split}
\dif X_t &=\sigma(X_t)\dif W_t+b(X_t)\dif t+\int_{|z|\leq 1}g(X_{t-}, z)\tilde N(\dif t,\dif z)\\
&\quad+\int_{|z|>1}g(X_{t-}, z)N(\dif t,\dif z),\ \ X_0=x,
\end{split}
\end{align}
where $\sigma(x)=\sqrt{a(x)}$, $g(x,z):\mR^d\times\mR^d\to\mR^d$,   $W$ is a $d$-dimensional standard Brownian motion,
while $N$ is a Poisson random measure with intensity measure $\nu$, and $\tilde N$ is the associated compensated Poisson random measure.
Under some Lipschitz assumptions in $x$-variable on $\sigma (x)$, $b (x)$ and
$g(x, z)$,
 it is well knownn that the above SDE admits a unique strong solution, which defines a strong Markov process
whose infinitesimal generator $\sL$ is of the form \eqref{e:1.1} with
 $\mu_x(\dif z)=\nu\circ g^{-1}(x,\cdot)(\dif z)$ (see \cite{Ikeda-Watanabe}).
A natural question is whether SDE \eqref{e:1.2} has a (weak) solution without Lipschitz assumption on $\sigma (x)$, $b (x)$ and  $g(x, z)$, and how about its density.
\medskip

In this work we are concerned with the existence, uniqueness,
 and estimates of fundamental solutions of time-dependent version
of the operator $\sL$ in \eqref{e:1.1},
with minimal regularity assumptions on $a(t, x)$, $b(t, x)$ and
$\kappa (t, x, z)$, where $\kappa(t,x,z):=|z|^{d+\alpha}\mu_{t,x}(\dif z)/\dif z$.
More precisely,
we shall consider the following time-inhomogeneous and non-symmetric non-local operators:
\begin{align}
  \sL_t f(x)&:=\sL^{a}_tf(x)+b_t\cdot \nabla f(x)+\sL^{\kappa}_tf(x),\label{eqL}
\end{align}
where
\begin{align*}
  &\sL^{a}_tf(x):=\frac{1}{2}\sum_{i,j=1}^d a_{ij}(t,x)\p ^2_{ij} f(x),\ b_t\cdot \nabla f(x):=\sum_{i=1}^{d}b_i(t,x)\p_if(x),\\
  &\sL^{\kappa}_tf(x):=\int_{\mR^d}\left(f(x+z)-f(x)-1_{\{|z|\leq 1\}}z\cdot\nabla f(x)\right)\frac{\kappa(t,x,z)}{|z|^{d+\alpha}}\dif z.
\end{align*}
Here $a(t,x):=(a_{ij}(t,x))_{1\leq i,j\leq d}$ is a $d\times d$-symmetric matrix-valued measurable function on $[0,\infty)\times\mR^d$, $b(t,x): [0,\infty)\times\mR^d\to\mR^d$ and $\kappa(t,x,z): [0,\infty)\times\mR^d\times\mR^d\to\mR$ are measurable functions, and $\alpha\in(0,2)$.

\medskip

With different choices of $a, b$ and $\kappa$, we get different types of operators $\mathscr{L}_t$.
For example, when $a =\mathbb{I}_{d\times d}$, $b=0$ and
$\kappa(t,x,z) = \mathcal{A}(d,-\alpha) \kappa$ for some $\kappa>0$,
$\mathscr{L}_t = \frac{1}{2}\Delta+\kappa\Delta^{\alpha/2}$ is the generator of independent sum of Brownian motion and rotational $\alpha$-stable process with weight $\kappa$. Here $\mathcal{A}(d,-\alpha)$ is a positive constant:
$\mathcal{A}(d,-\alpha) := \alpha2^{\alpha-1}\pi^{-d/2}\Gamma((d+\alpha)/2)\Gamma(1-\alpha/2)^{-1}$ and
$\Gamma$ is the Gamma function defined by
$\Gamma(\lambda):= \int_0^\infty t^{\lambda-1}e^{-t}\dif t, \lambda > 0$.
Moreover, the heat kernel of $\frac{1}{2}\Delta+\kappa\Delta^{\alpha/2}$ exists, denoted by
$p^\kappa(t,x,y)=p^\kappa(t,|y-x|)$. It is shown in \cite[Theorem 1.4]{ChenKumagai.2010.RMI551}
(see also \cite[Theorem 2.13]{SongVondravcek.2007.TMJ21} and \cite[Corollary 1.2]{ChenKimSong.2011.JLMS258}) that ,
there are constants $C,\lambda\geq 1$ depending only on $d,\alpha$ such that for all $t > 0$ and $x\in \mathbb{R}^d$,
\begin{equation}
\begin{split}
& C^{-1}\left(t^{-d/2}\wedge(\kappa t)^{-d/\alpha}\right) \wedge \left(t^{-d/2}\e^{-\lambda|x|^2/t}+\tfrac{\kappa t}{|x|^{d+\alpha}}\right) \leq p^\kappa(t,x) \\
  &\qquad\leq C\left(t^{-d/2}\wedge(\kappa t)^{-d/\alpha}\right) \wedge \left(t^{-d/2}\e^{-\lambda^{-1}|x|^2/t}+\tfrac{\kappa t}{|x|^{d+\alpha}}\right).
 \end{split}
 \end{equation}
The above in particular implies that for each $T, M>0$, all $\kappa \in [0, M]$, $t\in (0, T]$ and $x\in \mR^d$,
\begin{equation}\label{eq:paEst}
\begin{split}
& \widetilde C^{-1}\left(t^{-d/2}\e^{-\lambda|x|^2/t}+t^{-d/2}\wedge\tfrac{\kappa t}{|x|^{d+\alpha}}\right) \leq p^\kappa(t,x)
 \leq \widetilde C\left(t^{-d/2}\e^{-\lambda^{-1}|x|^2/t}+t^{-d/2}\wedge\tfrac{\kappa t}{|x|^{d+\alpha}}\right),
 \end{split}
 \end{equation}
where $\widetilde C \geq 1$ depends on $T,M,d,\alpha$.
For notational convenience, define for $\gamma,\lambda\in\mR$, $t>0$ and $x\in\mR^d$,
\begin{align}\label{ETA}
 \xi_{\lambda,\gamma}(t,x):=t^{(\gamma-d)/2}\e^{-\lambda |x|^2/t}
\quad \hbox{and}\quad
  \eta_{\alpha,\gamma}(t,x):=t^{\gamma/2}\big(|x|+t^{1/2}\big)^{-d-\alpha}.
\end{align}
It is easy to check that we can rewrite \eqref{eq:paEst} as
\begin{equation}\label{e:1.8}
 \widehat C^{-1}   \left( \xi_{\lambda, 0}(t,x)+ \kappa \eta_{\alpha,2}(t,x)   \right)
 \leq p^\kappa(t,x) \leq \widehat C  \left( \xi_{\lambda^{-1}, 0}(t,x)+ \kappa \eta_{\alpha,2}(t,x)   \right)
\end{equation}
for some $\widehat C\geq 1$ depending on $T,M,d,\alpha$.

\medskip

When $\kappa(t,x, z) = \mathcal{A}(d,-\alpha)1_{|z|\leq 1}$, $\mathscr{L}^\kappa$ is just the truncated fractional Laplacian
operator $\bar{\Delta}^{\alpha/2}$:
\begin{equation*}
  \bar{\Delta}^{\alpha/2} f(x) = \int_{\{|z|\leq 1\}} \left(f(x+z) -f(x) - z \cdot \nabla f(z)\right) \frac{\mathcal{A}(d,-\alpha)}{|z|^{d+\alpha}} \dif z.
\end{equation*}
It follows from \cite{ChenKimKumagai.2008.MA833} that the heat kernel of $\bar{\Delta}^{\alpha/2}$, denoted
by $\bar{p}_\alpha(t,x,y)=\bar{p}_\alpha(t,x-y)$, exists and it is jointly continuous and has the following estimates: there are constants $C_i = C_i(d,\alpha)>1,~i=1,2$ such that
\begin{equation}\label{eq:pbarIII}
  \begin{split}
    & C_1^{-1}\left(\left(\tfrac{t}{|x|}\right)^{C_2|x|}1_{|x|>1}+\left(t^{-\frac{d}{\alpha}}\wedge\tfrac{t}{|x|^{d+\alpha}}\right)1_{|x|\leq1}\right) \leq \bar{p}_\alpha(t,x) \\
    &\qquad\leq C_1  \left(\left(\tfrac{t}{|x|}\right)^{C^{-1}_2|x|}1_{|x|>1}+\left(t^{-\frac{d}{\alpha}}\wedge\tfrac{t}{|x|^{d+\alpha}}\right)1_{|x|\leq1}\right),\quad t \in (0,1], x \in \mathbb{R}^d.
  \end{split}
\end{equation}

\medskip

Throughout this paper, we assume $d\geq 2$ and make the following assumptions on $a$ and $\kappa$:
\begin{enumerate}
  \item[{(\bf H$^a$)}] There are $c_1>0$ and $\beta\in(0,1)$ such that for all $t>0$ and $x,y\in\mR^d$,
  \begin{align}
    |a(t,y)-a(t,x)|\leq c_1|y-x|^{\beta},\label{eqa2}
  \end{align}
  and for some $c_2\geq 1$,
  \begin{align}
    c_2^{-1}\mI_{d\times d} \leq a(t,x)\leq c_2 \mI_{d\times d}.\label{eqa1}
  \end{align}
  Here $\mI_{d\times d}$ denotes the $d\times d$ identity matrix.
\end{enumerate}
\begin{enumerate}
  \item[{(\bf H$^\kappa$)}] $\kappa(t,x,z)$ is a bounded measurable function and if $\alpha=1$, we require for any $0< r<R<\infty$,
  \begin{align}\label{Sym}
    \int_{r<|z|\leq R}z\kappa(t,x,z)|z|^{-d-1}\dif z=0.
  \end{align}
\end{enumerate}

Let $Z(t, x; s, y)$ be the fundamental solution of $\{\sL^a_t; t\geq 0\}$; see
Theorem \ref{T23} below for details.
Since $\sL_t$ can be viewed as a perturbation of $\sL^a_t$ by $\sL_t^{b, \kappa} :=b\cdot \nabla + \sL^\kappa_t$, heuristically the fundamental solution
(or heat kernel)
$p (t, x; s, y)$
of $\sL_t $ should satisfy the following
Duhamel's formula: for all $0\leq t<s<\infty$ and $x, y\in \mR^d$,
\begin{equation}\label{eqdu}
p(t, x; s, y)=Z(t, x; s, y)+\int_t^s\!\!\! \int_{\mR^d} p(t, x; r, z) \sL^{b, \kappa}_r
Z(r, \cdot ;  s, y) (z) \dif z \dif r,
\end{equation}
or
\begin{equation}\label{eqdu0}
p(t, x; s, y)=Z(t, x; s, y)+\int_t^s\!\!\! \int_{\mR^d} Z(t, x; r, z) \sL^{b, \kappa}_r
p(r, \cdot ;  s, y) (z) \dif z \dif r.
\end{equation}
For any $T\in(0,\infty]$ and $\eps\in[0,T)$, we write
$$
  \mD^T_{\eps}:=\Big\{(t,x;s,y): x,y\in\mR^d \mbox{ and } s,t\geq 0 \mbox{ with } \eps<s-t<T\Big\}.
$$

\medskip

The following  are  the main results of this paper.
See \eqref{Def3} below  for the definition of space-time Kato class $\mK_2$ of functions on $\mR\times \mR^d$.
We will see from Proposition \ref{in}  below that  $\mK_2$ contains
$L^q(\mR;L^p(\mR^d))$ for any
$p, q\in[1,\infty]$
with $\frac{d}p+\frac2{q}<1$.

\bt\label{main1}
Let $\alpha\in(0,2)$. Under {\bf (H$^a$)}, {\bf (H$^\kappa$)} and
$b \in\mK_2$,
there is a unique  continuous function $p(t,x; s, y)$ on $\mD^\infty_0$  that   satisfies \eqref{eqdu},
and
\begin{enumerate}[\rm (1)]
  \item (Upper-bound estimate) For any $T>0$, there exist constants $C_0,\lambda_0>0$ such that on $\mD^T_0$,
      \begin{align}
        |p(t,x;s,y)|\leq C_0(\xi_{\lambda_0,0}+\|\kappa\|_\infty\eta_{\alpha,2})(s-t,y-x).    \label{eqlpe}
      \end{align}
Moreover, the following hold.
  \item (C-K equation) For all $0\leq t<r<s<\infty$ and $x,y\in\mR^d$,
\begin{equation}\label{e:1.13}
\int_{\mR^d} p(t, x; r,  z) p(r, z; s, y) \dif z =p(t, x; s, y).
\end{equation}

  \item (Gradient estimate) For any $T>0$, there exist constants $C_1,\lambda_1>0$ such that on $\mD^T_0$,
      \begin{align}
        |\nabla_x p(t,x;s,y)|\leq C_1(\xi_{\lambda_1,-1}+\|\kappa\|_\infty\eta_{\alpha,1})(s-t,y-x).   \label{eqpge}
      \end{align}
   \item (Fractional derivative estimate) If in addition for $\alpha \in (0,1]$, $b\in \mK_1$ and for $\alpha\in(1,2)$, $b\in\bar\mK_{\alpha}$
   (see \eqref{Def4} below for a definition), then  for any $T>0$,  there exists a constant $C_2>0$ such that  on $\mD^T_0$,
      \begin{align}
        |\Delta^{ {\alpha}/{2}}p(t,\cdot;s,y)(x)|\leq C_2\eta_{\alpha,0}(s-t,y-x).\label{ET311}
      \end{align}
  Meanwhile, equation \eqref{eqdu0} holds on $\mathbb{D}_0^\infty$.
 \item (Conservativeness) For any $0\leq t<s<\infty$ and $x \in \mathbb{R}^d$,
      \begin{equation}\label{eq:conser}
        \int_{\mR^d}p(t,x;s,y)\dif y=1.
      \end{equation}
  \item (Generator) For any $f\in C_b^2(\mR^d)$, we have
      \begin{align}
        P_{t,s}f(x)-f(x)=\int^s_t\!P_{t,r}\sL_rf(x)\dif r,\label{eqge}
      \end{align}
    where  $P_{t,s}f(x):=\int_{\mR^d}p(t,x;s,y)f(y)\dif y.$
  \item (Continuity) For any bounded and uniformly continuous function $f(x)$, we have
      \begin{align}
        \lim_{|t-s|\to 0}\|P_{t,s}f-f\|_\infty=0.\label{cz1}
      \end{align}
\end{enumerate}
\et
\br
Estimate \eqref{ET311} is new even for $\kappa\equiv0$.
\er

Note that in Theorem \ref{main1}, we do not assume $\kappa (t, x, z)\geq 0$ and so the fundamental solution
$p(t, x; s, y)$
can take negative values; see Remark \ref{R:1.4} below.
The following theorem gives the lower bound estimate.

\begin{theorem}\label{thm:posiLow}
  Under the same assumptions of Theorem \ref{main1}, if for each $t>0$ and $x \in \mathbb{R}^d$,
  \begin{equation}\label{eq:kappPosi}
    \kappa(t,x,z) \geq 0,\quad a.e.~ z\in \mathbb{R}^d,
  \end{equation}
  then $p(t,x;s,y) \geq 0$ on $\mD^\infty_0$. Moreover, for any $T > 0$, there are constants $C_3, \lambda_3 > 0$ such that
  \begin{equation}\label{Low}
    p(t,x;s,y)\geq C_3(\xi_{\lambda_3,0}+m_\kappa\eta_{\alpha,2})(s-t,y-x)\ \mbox{on $\mD^T_0$},
  \end{equation}
  where $m_\kappa := \inf_{(t,x)}\essinf_{z\in \mathbb{R}^d} \kappa(t,x,z)$.
\end{theorem}

\begin{remark}\label{R:1.4}
  Under the hypothesis of Theorem \ref{thm:posiLow}, if in addition, $\kappa$ satisfies that for each $t\geq 0$,
  \begin{equation*}
    x \mapsto \kappa(t,x,z) \text{ is continuous }\quad a.e.~ z\in \mathbb{R}^d,
  \end{equation*}
  then, we can prove that (\ref{eq:kappPosi}) is also a necessary condition to the positivity of $p(t,x;s,y)$. For example, see the proof of
  \cite[Theorem 1.2]{CW}, or \cite[Lemma 4.5]{ChenHu.2015.SPA2603} or \cite{Wang.2015.MZ521}.
\end{remark}

The following corollary follows immediately from Theorems \ref{main1} and   \ref{thm:posiLow}.
\begin{corollary}\label{cor:2sided}
  Let $\alpha\in(0,2)$. Under {\bf (H$^a$)}, {\bf (H$^\kappa$)}, $b \in\mK_2$ and \eqref{eq:kappPosi},  for every $T>0$,
  there are positive constants $C,\lambda\geq 1$ such that on $\mD^T_0$,
  \begin{align*}
    C^{-1} \Big(\xi_{\lambda,0}+m_\kappa\eta_{\alpha,2}\Big)(s-t,y-x)\leq p(t,x;s,y) \leq C\Big(\xi_{\lambda^{-1},0}+\|\kappa\|_\infty\eta_{\alpha,2}\Big)(s-t,y-x).
  \end{align*}
\end{corollary}

\vspace{2mm}

In the truncated case, we consider the following two conditions on $\kappa$:
\begin{enumerate}
  \item[{\bf (HU$^\kappa$)}]
$ 0\leq \kappa(t,x,z) \leq \kappa_0 1_{|z|\leq 1}(z)$ for some $\kappa_0>0$.

\medskip

  \item[{\bf (HL$^\kappa$)}] $\kappa(t,x,z) \geq \kappa_0 1_{|z|\leq 1}(z)$ for some $\kappa_0>0$.
\end{enumerate}

\begin{theorem}\label{thm:trunc}
  Suppose that {\bf (H$^a$)}, {\bf (H$^\kappa$)} and $b\in \mathbb{K}_2$ hold. Let $T>0$, and for $\lambda>0$, set
  \begin{align}\label{Eta}
  \bar\eta_{\alpha,\lambda}(t,x):=t\big(|x|+t^{1/2}\big)^{-(d+\alpha)}1_{|x|\leq 1/2}+(t/|x|)^{\lambda |x|}1_{|x|>1/2}.
  \end{align}
  \begin{enumerate}[\rm (i)]
    \item If in addition $\kappa$ satisfies {\bf (HU$^\kappa$)}, then there are constants $C_1,\lambda_1> 0$ such that
      \begin{equation*}
        p(t,x;s,y) \leq C_1 \left(\xi_{\lambda_1,0} + \bar{\eta}_{\alpha,1/8}\right)(s-t,y-x)
				\quad \text{on } \mathbb{D}^T_0.
      \end{equation*}
    \item If in addition $\kappa$ satisfies {\bf (HL$^\kappa$)}, then there are constants $C_2,\lambda_2 > 0$ such that
      \begin{equation*}
        p(t,x;s,y) \geq C_2 \left(\xi_{\lambda_2,0} + \bar{\eta}_{\alpha,8}\right)(s-t,y-x)
				\quad \hbox{on } \mathbb{D}^T_0.
      \end{equation*}
  \end{enumerate}
\end{theorem}

Heat kernel analysis takes an important place in PDE and in probability theory,
as heat kernel encodes all the information about the corresponding generator and
the corresponding Markov processes.
Since explicit formula can only be derived in some very special and limited  cases,
the main focus of the heat kernel analysis is on its sharp estimates. While it is relatively
easy to get some crude bounds, obtaining sharp two-sided bounds on the heat kernel
is typically quite  delicate and challenging. It requires deep understanding of
the corresponding generator.
For second order elliptic operators and  diffusion process, a lot is known and
there are many beautiful results .
For instance, the celebrated
Aronson's estimate \cite{Aronson1968}
asserts that
the heat kernel
for uniformly  elliptic operators of divergence form
with measurable coefficients  has
two-sided Gaussian-type bounds.
Aronson's estimate also holds for non-divergence form elliptic operators
with H\"older continuous coefficients; see Theorem \ref{T23} below.

\medskip

 The study of heat kernel for non-local operators is relatively recent,
 propelled by interest  in discontinuous Markov processes,
 as many physical, engineering and social phenomena  can be successfully
 modeled by using discontinuous Markov processes including L\'evy processes.
 The infinitesimal generators of discontinuous Markov processes are non-local
 operators.
During the past
  several years  there
  is also many interest
  from the theory of PDE (such as singular obstacle problems) to study
non-local operators; see, for example,
 \cite{CSS2008, Silvestre2007} and the references
therein. Quite many progress has been made in the last fifteen years on
the  development of the
   DeGiorgi-Nash-Moser-Aronson type  theory for   symmetric non-local operators.
   For example,    Kolokoltsov  \cite{Kolokoltsov2000} obtained two-sided heat kernel
estimates for certain stable-like processes in $\mR^d$, whose
infinitesimal generators are   a class of pseudo-differential
operators having smooth symbols.
 Bass and Levin
\cite{BL2002.TAMS} used a completely different approach to obtain similar
estimates for discrete time Markov chain on $\mZ^d$, where the
conductance between $x$ and $y$ is comparable to $|x-y|^{-n-\alpha}$
for $\alpha \in (0,2)$.
In Chen and Kumagai \cite{CK03},
two-sided heat kernel
estimates and a scale-invariant parabolic Harnack inequality (PHI in abbreviation) for
symmetric $\alpha$-stable-like processes on $d$-sets are obtained.
 Recently in \cite{CK08},
two-sided heat kernel estimates and PHI are
established for symmetric non-local operators of variable order.
The DeGiorgi-Nash-Moser-Aronson type  theory is studied very recently in
Chen and Kumagai \cite{ChenKumagai.2010.RMI551} for symmetric diffusions with jumps.
  We refer the reader to the survey articles \cite{Chen2009, GHL2014} and the references therein on the study of
  heat kernels for symmetric non-local operators.
   However,   for non-symmetric non-local operators,  much less is known.
In \cite{BogdanJakubowski.2007.CMP179}, Bogdan and Jakubowski considered a fundamental solution to
the non-local operator $\Delta^{\alpha/2}+b(x)\cdot\nabla$ with $\alpha\in(1,2)$ and
$b$ belonging to some Kato's class, and obtained its sharp two-sided estimates.
 The uniqueness of fundamental solution to $\Delta^{\alpha/2}+b(x)\cdot\nabla$
 and its connection to stable processes with drifts are settled in Chen and Wang \cite{ChenLWang}.
 In \cite{Xie-Zhang}, Xie and Zhang studied
the critical case
$a(t, x)\Delta^{1/2}+b(t, x)\cdot\nabla$.
Heat kernels for subordinate Brownian motions with drifts have been studied in  \cite{ChenHu.2015.SPA2603}
and \cite{ChenDou}.
Chen and Wang \cite{CW} studied heat kernel estimates for $\Delta^{\alpha/2}$ under non-local perturbation,
while Wang \cite{Wang.2015.MZ521} investigated heat kernel for $\Delta$ perturbed by non-local operators.
Recently,  Chen and Zhang \cite{Chen-Zhang} obtained sharp two-sided estimates, gradient estimate
and fractional derivative estimate of the heat kernel
for general  non-local and non-symmetric operator $\sL^\kappa$ with $\kappa(t,x,z)=\kappa(x,z)$ by using Levi's parametrix method.

\medskip

In this paper, we concentrate on the study of heat kernel for non-symmetric operators $\sL$ of
type \eqref{eqL},
which have both diffusive and non-local parts . When $\kappa (t, x, y) \geq 0$,
its fundamental solution $p(t, x; s,y)$ becomes a family of transition density and so it determines a Feller process $X$
having strong Feller property. Clearly, the law of $X$ is  a solution to the martingale problem for $(\sL, C^2_c (\mR^d))$.
Is the solution to the martingale problem for $(\sL, C^2_c (\mR^d))$ unique?
It is also tempting to ask that when
$ \kappa (t, x, z)/|z|^{d+\alpha}$ is of the form $\nu \circ g^{-1} (t, x, \cdot) (\dif z)$
for some $g(t, x,z):\mR^d\times\mR^d\to\mR^d$ and a $\sigma$-finite measure $\nu$ on
$\mR^d\setminus \{ 0\}$,  whether this Feller process $X$ satisfies the following
 SDE:
\begin{align}
\dif X_t &=\sigma(t, X_t)\dif W_t+b(t, X_t)\dif t+\int_{|z|\leq 1}g(t, X_{t-}, z)\tilde N(\dif t,\dif z) \nonumber \\
&\quad+\int_{|z|>1}g(t, X_{t-}, z)N(\dif t,\dif z),\ \ X_0=x,  \label{e:1.24}
\end{align}
where
$\sigma(t, x)=\sqrt{a(t, x)}$,
$W$ is a $d$-dimensional standard Brownian motion,
$N$ is a Poisson random measure with intensity measure $\nu$, and $\tilde N$ is the associated compensated Poisson random measure?
We plan to address these questions in a separate work.

\medskip

The rest of the  paper is organized as follows.
 In Section 2, we present some key estimates that will be used later. In Section 3,
we prove our main result Theorem \ref{main1}.  The main  crux of work
is on various gradient and fractional derivative estimates, which is crucial for the iteration
procedure and rigorously establishing the Duhamel's formula.
In Section 4, we first show the positivity of $p(t,x;s,y)$ by the maximum principle under the non-negativeness of $\kappa$.
We then derive the lower bound estimate by a
probabilistic approach after obtaining the on-diagonal estimate of $p(t,x;s,y)$.
In Section 5, we consider the truncated case. In the Appendix, we show a maximum principle and
derive two-sided Aronson-type Gaussian estimates for heat kernels of  time-dependent second-order
elliptic differential operators.

\medskip

We conclude this introduction by mentioning some conventions that will be used throughout this paper. The letter $C$ or $c$ with or without subscripts will denote an unimportant constant.
For two quantities $f$ and $g$, $f\asymp g$ means that $C^{-1} f\leq f\leq C g$ for some $C\geq 1$, and
$f\preceq g$ means that $f\leq C g$  for some $C\geq 1$. The letter $\mathbb{N}$ will denote the collection of positive integers, and $\mathbb{N}_0 := \mathbb{N} \cup \{0\}$.

\section{Preliminaries}

\subsection{Basic estimates}
We first prove the following elementary but important estimates (which can also be called $3P$-inequalities)
for later use.
 Recall that the functions $\xi_{\lambda, \gamma}$ and $\eta_{\alpha, \gamma}$ are defined
in \eqref{ETA}.

\begin{lemma}\label{Le21}  \begin{enumerate}[\rm (i)]
    \item  For any $\alpha\in(0,+\infty)$ and $\lambda>0$, there exist positive constants $C_1=C_1(d,\alpha,\lambda)$ and $C_2=C_2(d,\alpha)$ such that for all $t>0$ and $x\in\mR^d$,
        \begin{align}
          \xi_{\lambda,0}(t,x)\leq C_1 \eta_{\alpha,\alpha}(t,x),\label{eqcom}
        \end{align}
        and for all $\gamma\geq 0$ and $t>0$,
        \begin{align}
          \int_{\mR^d}\eta_{\alpha,\gamma}(t,x)\dif x\leq C_2t^{(\gamma-\alpha)/2},\label{ineq}
        \end{align}
        and for some $C_3=C_3(d,\alpha)>0$
        and all $t,s>0$, $x\in\mR^d$ and $\gamma\in[0,\alpha]$,
         \begin{equation}\label{eq:etaSmpgLower}
          \int_{\mathbb{R}^d} \eta_{\alpha,\gamma}(t,x-z)\eta_{\alpha,\alpha}(s,z) \dif z \geq C_3 \eta_{\alpha,\gamma}(t+s,x).
        \end{equation}
    \item For any $0<\alpha \leq\beta$ and for all $t,s>0$, $x,y\in\mR^d$, we have
        \begin{align}
          \eta_{\alpha,0}(t,x)\eta_{\beta,0}(s,y)\leq 2^{d+\alpha} \left(\eta_{\beta,0}(t,x)+\eta_{\beta,0}(s,y)\right)\eta_{\alpha,0}(t+s,x+y). \label{eq3p}
        \end{align}
        Moreover,  there is a constant $C_4=C_4(d,\alpha,\beta)>0$ such that for all $\gamma_1,\gamma_2>\beta-2$,
        \begin{align}
          \begin{split}
            &\int^s_t\!\!\!\int_{\mR^d}\eta_{\alpha,\gamma_1}(s-r,y-z)\eta_{\beta,\gamma_2}(r-t,z-x)\dif z\dif r\\
            &\qquad\leq C_4\cB\big(\tfrac{\gamma_1-\beta}{2}+1,\tfrac{\gamma_2-\beta}{2}+1\big) \eta_{\alpha, 2+\gamma_1+\gamma_2-\beta}(s-t,y-x) ,\label{3P}
          \end{split}
        \end{align}
        where $\cB(\beta,\gamma):=\int^1_0(1-s)^{\beta-1}s^{\gamma-1}\dif s$ is the usual Beta function.

    \item For any $\alpha\in(0,2)$, there exists a constant $C_5=C_5(d,\alpha,\lambda)>0$ such that for all $\gamma_1>-2$ and
    $\gamma_2>\alpha-2$,
        \begin{align}
          &\int^s_t\!\!\!\int_{\mR^d}\xi_{\lambda,\gamma_1}(r-t,z-x)\eta_{\alpha,\gamma_2}(s-r,y-z)\dif z\dif r\no\\
          &\qquad\leq C_5\cB\big(\tfrac{\gamma_1}{2}+1,\tfrac{\gamma_2-\alpha}{2}+1\big) \eta_{\alpha,2+\gamma_1+\gamma_2}(s-t,y-x).\label{UY3}
        \end{align}
    \item For any $\lambda>0$, we have
        \begin{align}\label{CKE}
          \int_{\mR^d}\xi_{\lambda,0}(t,x-y)\xi_{\lambda,0}(s,y)\dif y=(\pi\lambda^{-1})^{{d/2}}\xi_{\lambda,0}(t+s,x),
        \end{align}
        and for all $\gamma_1,\gamma_2>-2$,
        \begin{align}
          &\int^s_t\!\!\!\int_{\mR^d}\xi_{\lambda,\gamma_1}(r-t,z-x)\xi_{\lambda,\gamma_2}(s-r,y-z)\dif z\dif r\no\\
          &\qquad=(\pi\lambda^{-1})^{{d/2}} \cB\big(\tfrac{\gamma_1}{2}+1,\tfrac{\gamma_2}{2}+1\big) \xi_{\lambda,2+\gamma_1+\gamma_2}(s-t,y-x).\label{UY30}
        \end{align}
  \end{enumerate}
\end{lemma}

\begin{proof}
  (i) If $|x|\leq t^{{1 /2}}$, then
  $$
    \xi_{\lambda,0}(t,x)\leq t^{-{d/2}}\leq 2^{d+\alpha}\eta_{\alpha,\alpha}(t,x).
  $$
  If $|x|>t^{{1 /2}}$, then
  $$
    \xi_{\lambda,0}(t,x)=t^{{\alpha /2}}|x|^{-(d+\alpha)}\left(|x|^2/t\right)^{(d+\alpha)/2} \e^{-\lambda|x|^2/t} \preceq t^{{\alpha /2}}|x|^{-(d+\alpha)}\preceq \eta_{\alpha,\alpha}(t,x).
  $$
  Moreover, we have
  \begin{align*}
    \int_{\mR^d}\eta_{\alpha,\gamma}(t,x)\dif x\leq t^{\gamma/2}\left(\int_{|x|\leq t^{1/2}}
t^{-\frac{d+\alpha}{2}}\dif x + \int_{|x|> t^{1/2}}|x|^{-d-\alpha}\dif x\right)\preceq t^{(\gamma-\alpha)/2}.
  \end{align*}
  To prove \eqref{eq:etaSmpgLower},  it suffices to show it for $\gamma=\alpha$. Thus, by symmetry we may assume $s\leq t$.
  Noticing that for $|z| \leq s^{1/2}$,
  \begin{equation*}
    |x-z|+t^{1/2} \leq |x|+|z|+t^{1/2} \leq |x|+2(t+s)^{1/2},
  \end{equation*}
  we have
  \begin{align*}
    &\int_{\mathbb{R}^d} \eta_{\alpha,\alpha}(t,x-z)\eta_{\alpha,\alpha}(s,z) \dif z\geq  \int_{|z|\leq s^{1/2}} \frac{t^{\alpha/2}}{(|x-z|+t^{1/2})^{d+\alpha}}
    \eta_{\alpha,\alpha}(s,z) \dif z\\
    &\geq ~ \frac{2^{-\alpha/2}(t+s)^{\alpha/2}}{(|x|+2(t+s)^{1/2})^{d+\alpha}}\int_{|z|\leq s^{1/2}} \eta_{\alpha,\alpha}(s,z) \dif z
    \geq ~ \frac{2^{-\alpha/2}}{2^{d+\alpha}}\eta_{\alpha,\alpha}(t+s,x)\int_{|z|\leq 1} \eta_{\alpha,\alpha}(1,z) \dif z.
  \end{align*}
  (ii) Estimate \eqref{eq3p} follows by the following easy inequality:
  \begin{align*}
    &\big(|x+y|+(t+s)^{1/2}\big)^{d+\alpha}\leq 2^{d+\alpha-1}\Big\{\big(|x|+t^{1/2}\big)^{d+\alpha}+\big(|y|+s^{1/2}\big)^{d+\alpha}\Big\}\no\\
    &\qquad\leq2^{d+\alpha}\Big\{\big(|x|+t^{1/2}\big)^{d+\alpha}+\big(|x|+t^{1/2}\big)^{\alpha-\beta}\big(|y|+s^{1/2}\big)^{d+\beta}\Big\},
  \end{align*}
  where the second inequality is due to $b^{\alpha}\leq a^{\alpha}+a^{\alpha-\beta} b^{\beta}$ for $0\leq \alpha\leq\beta$ and $a,b\geq 0$.
  Moreover, by (\ref{eq3p}) and (\ref{ineq}), we have
  \begin{eqnarray*}
    &&\int^s_t\!\!\!\int_{\mR^d}\eta_{\alpha,\gamma_1}(s-r,y-z)\eta_{\beta,\gamma_2}(r-t,z-x)\dif z\dif r \leq 2^{d+\alpha}\eta_{\alpha,0}(s-t,y-x)\\
    &&\times \int^s_t(s-r)^{\gamma_1/2}(r-t)^{\gamma_2/2}\!\!\!\int_{\mR^d}\Big(\eta_{\beta,0}(s-r,y-z)+\eta_{\beta,0}(r-t,z-x)\Big)\dif z\dif r\\
    &\preceq &\eta_{\alpha,0}(s-t,y-x)\int^s_t\left((s-r)^{(\gamma_1-\beta)/2}(r-t)^{\gamma_2/2} +(s-r)^{\gamma_1/2}(r-t)^{(\gamma_2-\beta)/2}\right)\dif r\\
    &\preceq &\eta_{\alpha,2+\gamma_1+\gamma_2-\beta}(s-t,y-x) \left(\cB(\tfrac{\gamma_1-\beta}{2}+1,\tfrac{\gamma_2}{2}+1) + \cB(\tfrac{\gamma_1}{2}+1,\tfrac{\gamma_2-\beta}{2}+1)\right)\\
    &\preceq & \cB(\tfrac{\gamma_1-\beta}{2}+1,\tfrac{\gamma_2-\beta}{2}+1) \eta_{\alpha,2+\gamma_1+\gamma_2-\beta}(s-t,y-x) .
  \end{eqnarray*}
  (iii) It follows by \eqref{eqcom} with $\xi_{\lambda,\gamma_1}(t,x)
	\leq C_1 \eta_{\alpha,\alpha+\gamma_1}(t,x)$ and \eqref{3P} with $\beta=\alpha$.
	\medskip\\
  (iv) It follows by Chapman-Kolmogorov's equation for Brownian transition density function.
\end{proof}

\subsection{Fractional derivative estimates of Gaussian kernel}
For $\alpha\in(0,2)$, set
$$
z^{(\alpha)}:=z1_{\alpha\in(1,2)}+z1_{|z|\leq 1}1_{\alpha=1}.
$$
Let  $J:\mR^d\to\mR$ be a bounded measurable function. For a function $f(x)$ on $\mR^d$, define
\begin{align}
  \widetilde\sL^{J}f(x):=\int_{\mR^d}\delta^{(\alpha)}_f(x;z) J(z)|z|^{-d-\alpha}\dif z,\label{j}
\end{align}
where
\begin{align}
  \delta^{(\alpha)}_f(x;z):=f(x+z)-f(x)-z^{(\alpha)}\cdot\nabla f(x).\label{de}
\end{align}
The following lemma will play an important role in the sequel.
\bl\label{Le22}
  Given $\alpha\in(0,2)$, let $J:\mR^d\to\mR$ be a bounded measurable function with
  \begin{align}\label{Sym1}
    \int_{r<|z|\leq R}z\cdot J(z)|z|^{-d-1}\dif z=0,\quad 0<r<R<\infty.
  \end{align}
  Let $T > 0$ and  $G_t(x):(0,T)\times\mR^d\to\mR$ be a $C^2$ function in $x$. Suppose that for each $j=0,1,2$, there are $C_j>0$ and $\beta_j\geq0$ such that for $t \in (0,T)$ and $x \in \mathbb{R}^d$,
  \begin{align}\label{TR1}
    |\nabla^jG_t(x)|\leq C_j\eta_{\alpha,\alpha-\beta_j-j}(t,x).
  \end{align}
  Then for any $\gamma\in[0,(2-\alpha)\wedge 1)$, there exists a constant $C=C(\gamma, d,\alpha)>0$ such that
  \begin{align}
    \big|\widetilde\sL^{J}G_t(x)\big|\leq C\|J\|_\infty\Big(C_0t^{-\frac{\beta_0}{2}}+ C_1t^{-\frac{\beta_1}{2}}+C^\gamma_1C^{1-\gamma}_2 t^{-\frac{\gamma\beta_1+(1-\gamma)\beta_2}{2}}\Big) \eta_{\alpha,0}(t,x).\label{eqfrG0}
  \end{align}
\el
\begin{proof}
 By definition \eqref{j}-\eqref{de},
\eqref{Sym1} and \eqref{Sym}, we have
  \begin{align*}
    \big|\widetilde\sL^{J}G_t(x)\big|
    &\leq\|J\|_\infty\Bigg[\int_{|z|\leq t^{1/2}}|\delta^{(\alpha)}_{G_t}(x;z)|\cdot|z|^{-d-\alpha}\dif z
    +\int_{|z|>t^{1/2}}|G_t(x+z)|\cdot|z|^{-d-\alpha}\dif z\\
    &\quad+|G_t(x)|\int_{|z|>t^{1/2}}|z|^{-d-\alpha}\dif z +1_{\alpha\in(1,2)}|\nabla G_t(x)|\int_{|z|>t^{1/2}}|z|\cdot|z|^{-d-\alpha}\dif z\Bigg]\\
    &=:\|J\|_\infty\Big[I_1+I_2+I_3+I_4\Big].
  \end{align*}
  Notice that for $\alpha\in(0,1)$,
  $$
    \delta^{(\alpha)}_{ G_t}(x;z)=\int_0^1\<z, \nabla  G_t(x+\theta z)\>\dif \theta,
  $$
  and for $\alpha\in[1,2)$,
  $$
    \delta^{(\alpha)}_{ G_t}(x;z)=\int_0^1\!\!\!\int_0^1\theta\<z\otimes z, \nabla^2 G_t(x+\theta'\theta z)\>\dif \theta'\dif \theta.
  $$
  By \eqref{TR1}, we have for all $|z|\leq t^{1/2}$,
  \begin{align*}
    |\delta^{(\alpha)}_{ G_t}(x;z)|&\leq|z|\int_0^1|\nabla G_t(x+\theta z)|\dif\theta+1_{\alpha\in[1,2)}|z|\cdot|\nabla G_t(x)|\\
    &\leq C_1\left(\int_0^1\eta_{\alpha,\alpha-\beta_1-1}(t,x+\theta z)\dif \theta+1_{\alpha\in[1,2)}\eta_{\alpha,\alpha-\beta_1-1}(t,x)\right)|z|\\
    &\leq 2^{d+\alpha}C_1\, \eta_{\alpha,\alpha-\beta_1-1}(t,x)\,|z|,
  \end{align*}
  and if $\alpha\in[1,2)$, we alternatively have
  \begin{align*}
    |\delta^{(\alpha)}_{ G_t}(x;z)|\leq 2^{d+\alpha}C_2\, \eta_{\alpha,\alpha-\beta_2-2}(t,x)\,|z|^2.
  \end{align*}
  Thus  for $I_1$, if $\alpha\in(0,1)$, then
  $$
   I_1\leq 2^{d+\alpha}C_1\, \eta_{\alpha,\alpha-\beta_1-1}(t,x)\int_{|z|\leq t^{1/2}}|z|^{1-d-\alpha}\dif z\preceq C_1 t^{-\beta_1/2}\eta_{\alpha,0}(t,x);
  $$
  if $\alpha\in[1,2)$, then by interpolation, we have for all $\gamma\in[0,2-\alpha)$,
  \begin{align*}
    I_1&\leq 2^{d+\alpha}\int_{|z|\leq t^{1/2}}\Big(C_1\, \eta_{\alpha,\alpha-\beta_1-1}(t,x)\,|z|\Big)^\gamma
    \Big(C_2\, \eta_{\alpha,\alpha-\beta_2-2}(t,x)\,|z|^2\Big)^{1-\gamma}|z|^{-d-\alpha}\dif z\\
    &\preceq C^\gamma_1 C^{1-\gamma}_2 t^{-\gamma\beta_1/2-(1-\gamma)\beta_2/2}\eta_{\alpha,0}(t,x).
  \end{align*}
  For $I_2$, by \eqref{eq3p} and \eqref{ineq}, we have
  \begin{align*}
  I_2&\leq C_0t^{\alpha-\beta_0/2}\int_{|z|>t^{1/2}}\eta_{\alpha,0}(t,x+z)|z|^{-d-\alpha}\dif z\\
  &\leq 2^{d+\alpha}C_0t^{\alpha-\beta_0/2}\int_{|z|>t^{1/2}}\eta_{\alpha,0}(t,x+z)\eta_\alpha(t,-z)\dif z\\
  &\leq 4^{d+\alpha}C_0t^{\alpha-\beta_0/2}\eta_{\alpha,0}(2t,x)\int_{\mR^d}[\eta_{\alpha,0}(t,x+z)+\eta_\alpha(t,-z)]\dif z\\
  &\preceq C_0t^{-\beta_0/2}\eta_{\alpha,0}(t,x).
  \end{align*}
  For $I_3$ and $I_4$, it is easy to see that
  $$
    I_3+I_4\preceq \Big(C_0t^{-\beta_0/2}+C_1t^{-\beta_1/2}\Big)\eta_{\alpha,0}(t,x).
  $$
  Combing the above calculations, we get (\ref{eqfrG0}).
\end{proof}

Under {\bf (H$^a$)}, it is more or less well known that there exists a fundamental solution to the operator $\p_t-\sL^{a}_t$ (cf. \cite{Friedman.1964.346}). However, to the best of our knowledge, most of the proofs also require the H\"older continuity of $a$ with respect to the time variable $t$. For the readers' convenience, a proof of the following result is provided in Appendix 6.2.
\bt\label{T23}
  Under {\bf (H$^a$)},
    there is a unique continuous function $Z$ on $\mD^\infty_0$ such that for a.e. $t\in (0, s)$,
  and every $x, y\in \mR^d$,
  \begin{equation}\label{e:2.15}
\partial_t Z(t, x; s, y) + \sL^a_t Z(t, \cdot ; s, y)(x)=0,
  \end{equation}
  and
  \begin{enumerate}[\rm (1)]
    \item (Upper and gradient estimate) For $j=0, 1,2$ and $T>0$, there exist constants $C,\lambda>0$ such that on $\mD^T_0$,
        \begin{align}
          |\nabla^j_x Z(t,x;s,y)|\leq C\xi_{\lambda,-j}(s-t,y-x).\label{eq21}
        \end{align}
          \noindent
    \!\!\!\!\!\!\!\!\!\! Moreover, $Z(t, x; s, y)$ enjoys the following properties.
    \medskip

    \item (H\"older estimate in $y$) For $j=0,1$, $\beta'\in(0,\beta)$ and $T>0$, there exist constants $C,\lambda>0$ such that on $\mD^T_0$,
        \begin{align}\label{eq202}
          |\nabla^j_x Z(t,x;s,y_1)-\nabla^j_x Z(t,x;s,y_2)|\leq C|y_1-y_2|^{\beta'}\sum_{i=1,2}\xi_{\lambda,-\beta'-j}(s-t,y_i-x).
        \end{align}
    \noindent

   \medskip

    \item (Continuity) For any bounded and uniformly continuous function $f(x)$,
        \begin{align}
          \lim_{|t-s|\to 0}\|P^{(Z)}_{t,s}f-f\|_\infty=0,\label{cz2}
        \end{align}
        where $P^{(Z)}_{t,s}f(x):=\int_{\mR^d}Z(t,x;s,y)f(y)\dif y$.
    \item (C-K equation) For all $0\leq t<r<s<\infty$, we have
        \begin{align}
          P^{(Z)}_{t,r}P^{(Z)}_{r,s}f=P^{(Z)}_{t,s}f.\label{CK1}
        \end{align}
    \item (Conservativeness) For all $0\leq t<s<\infty$, we have
        \begin{align}
          P^{(Z)}_{t,s}1=1.\label{z11}
        \end{align}
    \item (Generator) For any $f\in C^2_b(\mR^d)$, we have
        \begin{align}
          P^{(Z)}_{t,s}f(x)-f(x)=\int^s_tP^{(Z)}_{t,r}\sL^{a}_rf(x)\dif r=\int^s_t\sL^{a}_rP^{(Z)}_{r,s}f(x)\dif r.\label{ET01}
        \end{align}
     \item (Two-sided estimates) For any $T>0$, there exist constants $C,\lambda\geq 1$ such that on $\mD^T_0$,
        \begin{align}
          C^{-1}\xi_{\lambda,0}(s-t,y-x)\leq Z(t,x;s,y)\leq C\xi_{\lambda^{-1},0}(s-t,y-x).\label{ET41}
        \end{align}
\end{enumerate}
\et

We call $Z(t, x; s, y)$ the fundamental solution or heat kernel of $\sL^a$.
The following corollary gives  fractional derivative estimates of Gaussian kernels.

\bc
  Let $\alpha\in(0,2)$ and $J:\mR^d\to\mR$ be a bounded measurable function satisfying \eqref{Sym1}. Let $Z(t,x;s,y)$ be as in Theorem \ref{T23},  $\beta\in(0,1)$ be as in \eqref{eqa2}, and $T>0$.
  \begin{enumerate}[\rm (i)]
    \item There is a constant $C>0$ such that on $\mD^T_0$,
        \begin{align}
          \big|\widetilde\sL^{J}Z(t,\cdot;s,y)(x)\big|\leq C\|J\|_\infty\eta_{\alpha,0}(s-t,y-x).\label{ET31}
        \end{align}
    \item For any $\gamma\in[0,(2-\alpha)\wedge1)$, there is a constant $C>0$ such that on $\mD^T_0$,
        \begin{align}
          \begin{split}\label{ET32}
            &\big|\widetilde\sL^{J}Z(t,\cdot;s,y)(x_1)-\widetilde\sL^{J}Z(t,\cdot;s,y)(x_2)\big|
            \leq C \|J\|_\infty|x_1-x_2|^\gamma\sum_{i=1,2}\eta_{\alpha,-\gamma}(s-t,y-x_i).
          \end{split}
        \end{align}
    \item For any $\gamma\in[0,(2-\alpha)\wedge1)$ and $\beta'\in(0,\beta)$, there is a constant $C>0$ such that on $\mD^T_0$,
        \begin{align}
          \begin{split}\label{ET33}
            &\big| \widetilde\sL^{J}Z(t,\cdot;s,y_1)(x) - \widetilde\sL^{J}Z(t,\cdot;s,y_2)(x)\big|
            \leq C\|J\|_\infty |y_1-y_2|^{\beta'\gamma}\sum_{i=1,2}\eta_{\alpha,-\beta'\gamma}(s-t,y_i-x).
          \end{split}
        \end{align}
    \item For any $\gamma\in(0,1]$, there is a constant $C>0$ such that on $\mD^T_0$,
        \begin{align}\label{225}
         \begin{split}
          &\big|\nabla_x Z(t,x_1;s,y)-\nabla_x Z(t,x_2;s,y)\big|\leq C |x_1-x_2|^\gamma\sum_{i=1,2}\xi_{\lambda/2,-\gamma-1}(s-t,y-x_i).
          \end{split}
        \end{align}
\end{enumerate}
\ec
\begin{proof}
  (i) By  \eqref{eq21} and \eqref{eqcom}, estimate \eqref{ET31} follows by applying Lemma \ref{Le22} to function
  $$
    (r,x)\mapsto Z(t,x+y; t+r,y)
  $$
  with $\beta_j=0,~ j=0,1,2$, and letting $r=s-t$.
    \medskip\\
  (ii) For fixed $t<s$ and $x_1,x_2,y\in\mR^d$, let us define
  $$
    G_r(z):=Z(t,z+y;t+r,y)-Z(t,x_2-x_1+z+y;t+r,y).
  $$
  Clearly,
  $$
    I:=\widetilde\sL^{J}Z(t,\cdot;s,y)(x_1)-\widetilde\sL^{J}Z(t,\cdot;s,y)(x_2)=\widetilde\sL^{J}G_{s-t}(x_1-y).
  $$
  If $|x_1-x_2|>\sqrt{s-t}$, then by (i), we have
  \begin{align}
    |I|&\leq |\widetilde\sL^{J}Z(t,\cdot;s,y)(x_1)|+|\widetilde\sL^{J}Z(t,\cdot;s,y)(x_2)|\no\\
    &\preceq C\|J\|_\infty\Big(\eta_{\alpha,0}(s-t,y-x_1)+\eta_{\alpha,0}(s-t,y-x_2)\Big)\no\\
    &\leq C\|J\|_\infty|x_1-x_2|^\gamma\Big(\eta_{\alpha,-\gamma}(s-t,y-x_1) +\eta_{\alpha,-\gamma}(s-t,y-x_2)\Big).\label{JH1}
  \end{align}
  If $|x_1-x_2|\leq\sqrt{s-t}$, then by \eqref{eq21}, we have for $j=0,1$,
  \begin{align*}
    |\nabla^jG_{s-t}(z)|&\leq|x_1-x_2|\int^1_0|\nabla^{j+1}_x Z(t,z+y+\theta(x_2-x_1);s,y)|\dif\theta\\
    &\preceq|x_1-x_2|\int^1_0\xi_{\lambda,-j-1}(s-t,-z-\theta(x_2-x_1))\dif\theta\\
    &\preceq |x_1-x_2|\cdot\xi_{\lambda/2,-j-1}(s-t,z),
  \end{align*}
  and
  \begin{align*}
  \begin{split}
    |\nabla^2G_{s-t}(z)|&\leq|\nabla^2_xZ(t,z+y;s,y)|+|\nabla^2_xZ(t,x_2-x_1+z+y;s,y)|\preceq \xi_{\lambda,-2}(s-t,z).
   \end{split}
  \end{align*}
  Hence, by \eqref{eqcom} and Lemma \ref{Le22} with $\beta_0=\beta_1=1$ and $\beta_2=0$, we obtain that for $|x_1-x_2|\leq\sqrt{s-t}$,
  \begin{align*}
    |I| &\leq C\|J\|_\infty\Big(|x_1-x_2|(s-t)^{-{1 /2}}+ |x_1-x_2|^\gamma (s-t)^{-\frac{\gamma}{2}}\Big)\eta_{\alpha,0}(s-t,x_1-y)\\
    &\leq C\|J\|_\infty|x_1-x_2|^\gamma\eta_{\alpha,-\gamma}(s-t,x_1-y).
  \end{align*}
  Combining this with \eqref{JH1}, we obtain \eqref{ET32}.
    \medskip\\
  (iii) As above, for fixed $t<s$ and $x,y_1, y_2\in\mR^d$, let us define
  $$
    G_r(z):=Z(t,z+y_1;t+r,y_1)-Z(t,z+y_1;t+r,y_2).
  $$
  Clearly,
  $$
    I:=\widetilde\sL^{J}Z(t,\cdot;s,y_1)(x)-\widetilde\sL^{J}Z(t,\cdot;s,y_2)(x)=\widetilde\sL^{J}G_{s-t}(x-y_1).
  $$
  If $|y_1-y_2|>\sqrt{s-t}$, then by (i), we have
  \begin{align}
    |I|&\leq |\widetilde\sL^{J}Z(t,\cdot;s,y_1)(x)|+|\widetilde\sL^{J}Z(t,\cdot;s,y_2)(x)|\no\\
    &\preceq \|J\|_\infty\Big(\eta_{\alpha,0}(s-t,y_1-x)+\eta_{\alpha,0}(s-t,y_2-x)\Big)\no\\
    &\preceq \|J\|_\infty|y_1-y_2|^{\beta'\gamma}\Big(\eta_{\alpha,-\beta'\gamma}(s-t,y_1-x) +\eta_{\alpha,-\beta'\gamma}(s-t,y_2-x)\Big).\label{JH2}
  \end{align}
  If $|y_1-y_2|\leq\sqrt{s-t}$, then by \eqref{eq202}, we have for $j=0,1$,
  \begin{align*}
    |\nabla^jG_{s-t}(z)|\preceq&\, |y_1-y_2|^{\beta'}\Big(\xi_{\lambda,-\beta'-j}(s-t,-z) +\xi_{\lambda,-\beta'-j}(s-t,y_2-y_1-z)\Big)\\
    \preceq&\, |y_1-y_2|^{\beta'}\xi_{\lambda/2,-\beta'-j}(s-t,z),
  \end{align*}
  and by \eqref{eq21},
  \begin{align*}
    |\nabla^2G_{s-t}(z)|\leq|\nabla^2_xZ(t,z+y_1;s,y_1)|+|\nabla^2_xZ(t,z+y_1;s,y_2)| \preceq \xi_{\lambda,-2}(s-t,z).
  \end{align*}
  Hence, by \eqref{eqcom} and Lemma \ref{Le22} with $\beta_0=\beta_1=\beta'$ and $\beta_2=0$, we obtain that for $|y_1-y_2|\leq\sqrt{s-t}$,
  \begin{align*}
    |I|&\leq C\|J\|_\infty\Big(|y_1-y_2|^{\beta'}(s-t)^{-\frac{\beta'}{2}} + |y_1-y_2|^{\beta'\gamma} (s-t)^{-\frac{\beta'\gamma}{2}}\Big)\eta_{\alpha,0}(s-t,x-y_1)\\
    &\leq C\|J\|_\infty|y_1-y_2|^{\beta'\gamma}\eta_{\alpha,-\beta'\gamma}(s-t,x-y_1).
  \end{align*}
  Combining this with \eqref{JH2}, we obtain \eqref{ET33}.
    \medskip\\
  (iv) It follows from \eqref{eq21} and the same argument as above.
\end{proof}

\subsection{Kato's class}
We introduce the following Kato's class of space-time functions.
Recall that functions $\xi_{\lambda, \gamma}(t, x)$ and $\eta_{\alpha, \gamma}(t, x)$ are defined in \eqref{ETA}.
For a function $g(t, x)$, we will use $g(t\pm s,x\pm y) $ as an abbreviation for
  $\sum_{j, k=0}^1  g (t+(-1)^j s,x + (-1)^ky)$.

\bd\label{Def2}
  (Generalized Kato's class) Let $\eta_{\alpha,\gamma}$ be given by \eqref{ETA}. For $\alpha\in[1,+\infty)$, define
  \begin{align}
    &\mK_{\alpha}:=\Big\{f: \mathbb{R}\times \mathbb{R}^d\to\mR \text{ satisfies } \lim_{\delta\to 0} K^{f}_{\alpha}(\delta)=0\Big\},\label{Def3}\\
    &\bar\mK_{\alpha}:=\Big\{f: \mathbb{R}\times \mathbb{R}^d \to\mR \text{ satisfies } \bar K^{f}_{\alpha}(1)<\infty\Big\},\label{Def4}
  \end{align}
  where
  \begin{align}
    K^{f}_{\alpha}(\delta)&:=\sup_{(t,x)}\int^{\delta}_0\!\!\!\int_{\mR^d} |f|(t\pm s,x\pm y) \eta_{\alpha,\alpha-1}(s,y)\dif y\dif s , \label{Def5} \\
    \bar K^{f}_{\alpha}(\delta)& :=\sup_{(t,x)}\int^{\delta}_0\!\!\!\int_{\mR^d} |f|(t\pm s,x\pm y) \eta_{\alpha,0}(s,y)\dif y\dif s \label{Def6} .
  \end{align}
  For $\alpha=+\infty$, we define
  $$
    \mK_\infty:=\Big\{f: \mathbb{R}\times \mathbb{R}^d\to\mR \text{ satisfies } \lim_{\delta\rightarrow 0}N^{f}_{\lambda}(\delta)=0 \ \hbox{ for every } \lambda>0\Big\},
  $$
  where
  \begin{align*}
    N^f_\lambda(\delta):=\sup_{(t,x)}\int^{\delta}_0\!\!\!\int_{\mR^d}|f|(t\pm s,x\pm y) \xi_{\lambda,-1}(s,y)\dif y\dif s   .
  \end{align*}
\ed

\br
  $\mK_\infty$ is the same as the Kato class defined in \cite{Zhang.1997.MM381}. For any $\lambda>0$ and $\alpha\in[1,\infty)$, by  (\ref{eqcom}), it is easy to see that there exists a constant $C=C(d,\alpha,\lambda)>0$ such that for all $\delta\in(0,1)$,
  \begin{align}
    N^f_\lambda(\delta)\leq C K^f_\alpha(\delta).\label{e219}
  \end{align}
  Moreover, for any time-independent function $f(t,x)=f(x)$, we have $f\in\mK_\alpha$ if and only if
  $$
    M^\alpha_f(\delta):=\sup_{x\in\mR^d}\int_{\mR^d}|f(x+y)|\cdot \frac{1}{|y|^{d-1}}\left(1\wedge \frac{\delta}{|y|^2}\right)^{(1+\alpha)/{2}}\dif y\to 0\quad\mbox{  as $\delta\to 0$}.
  $$
  Indeed, it follows by noticing that
  $$
    \int^{\delta}_0\eta_{\alpha,\alpha-1}(s,y)\dif s=\int^{\delta}_0\frac{s^{(\alpha-1)/{2}} \dif s}{(|y|+s^{1/2})^{d+\alpha}}\asymp \frac{(|y|^2\wedge\delta)^{(1+\alpha)/{2}}}{|y|^{d+\alpha}}=\frac{1}{|y|^{d-1}}\left(1\wedge \frac{\delta}{|y|^2}\right)^{(1+\alpha)/{2}}.
  $$
\er

We have the following results about the above Kato classes.
\bp\label{in}
  Let $\alpha\in[1,\infty)$, $p,q\in[1,\infty]$ and $f:\mathbb{R}\times\mathbb{R}^d\to \mathbb{R}$.
  \begin{enumerate}[\rm (i)]
  \item There is a constant $C=C(d,\alpha)>0$ such that for all $\delta>0$ and $j\in \mN$,
  \begin{equation}\label{DR2}
    K^f_\alpha(j\delta) \leq CjK^f_\alpha(\delta),\ \ \bar K^f_\alpha(j\delta) \leq Cj\bar K^f_\alpha(\delta)
  \end{equation}
  and
  \begin{equation}\label{eq:KatoLocalNorm}
    \sup_{(t,x)}\int_{0}^{\delta}\!\!\!\int_{\mathbb{R}^d} |f|(t\pm s,x\pm y) \dif y \dif s
    \leq C \left((\delta^{-(d+1)/2}K^f_\alpha(\delta))\wedge(\delta^{-(d+\alpha)/2}\bar K^f_\alpha(\delta))\right).
  \end{equation}
  \item If $\frac{d}{p}+\frac{2}{q}<1$, then
  $$
    L^q(\mR;L^p(\mR^d))\subset\mK_1\subset\mK_{\alpha}\subset\mK_\infty;
  $$
  and if $\alpha \in [1,2)$ and $\frac{d}{p}+\frac{2}{q}<2-\alpha$, then
  $$
    L^q(\mR;L^p(\mR^d))\subset\bar\mK_{\alpha}\subset\mK_\alpha.
    $$
    \end{enumerate}
\ep
\begin{proof}
 (i)  By definition \eqref{Def5}, we have
  \begin{align}\label{DR1}
  \begin{split}
    K^f_\alpha(j\delta)&= \sup_{(t,x)} \sum_{k=0}^{j-1} \int_{k\delta}^{(k+1)\delta}\!\! \int_{\mathbb{R}^d}
    |f|(t\pm s,x\pm y) \eta_{\alpha,\alpha-1}(s,y)\dif y\dif s\\
    &\leq K^f_\alpha(\delta)+ \sum_{k=1}^{j-1} \sup_{(t,x)}\int_0^{\delta}\!\!\! \int_{\mathbb{R}^d}
    |f|\big((t\pm k\delta)\pm s,x\pm y\big) \eta_{\alpha,\alpha-1}(s+k\delta,y)\dif y\dif s.
    \end{split}
    \end{align}
    Denoting the term in the above sum by $I_k$, by \eqref{eq:etaSmpgLower} and \eqref{ineq}, we have for each $k=1,\cdots,j-1$,
    \begin{align*}
    I_k&\preceq \sup_{(t,x)}\int_0^{\delta} \!\!\!\int_{\mathbb{R}^d} |f|\big((t\pm k\delta)\pm s,x\pm y\big)
    \int_{\mathbb{R}^d} \eta_{\alpha,\alpha-1}(s,y-z)\eta_{\alpha,\alpha}(k\delta,z) \dif z\dif y\dif s\\
    &=\sup_{(t,x)}\int_{\mathbb{R}^d} \left(\int_0^{\delta}\!\!\! \int_{\mathbb{R}^d}
    |f|\big((t\pm k\delta)\pm s,(x\pm z)\pm y\big) \eta_{\alpha,\alpha-1}(s,y) \dif y\dif s\right) \eta_{\alpha,\alpha}(k\delta,z) \dif z\\
    &\leq K^f_\alpha(\delta)  \int_{\mathbb{R}^d} \eta_{\alpha,\alpha}(k\delta,z) \dif z
    = K^f_\alpha(\delta)\left(\int_{\mathbb{R}^d} \eta_{\alpha,\alpha}(1,z) \dif z\right).
  \end{align*}
  Substituting this into \eqref{DR1}, we get the first inequality in \eqref{DR2}. Similarly, we can prove the second inequality in \eqref{DR2}.
On the other hand, for any $(t,x) \in \mathbb{R}\times \mathbb{R}^d$,
  \begin{align*}
   \int_0^{\delta}\!\!\!\int_{\mathbb{R}^d} |f|(t\pm s,x\pm y) \dif y \dif s
   & =\int_{\delta}^{2\delta}\!\!\!\int_{\mathbb{R}^d} |f|((t\pm\delta)\pm s,x\pm y) \frac{\eta_{\alpha,\alpha-1}(s,y)}{\eta_{\alpha,\alpha-1}(s,y)} \dif y \dif s\\
    &\leq \frac{(2\delta)^{(\alpha-1)/2}}{\delta^{(d+\alpha)/2}} \sup_{(t,x)} \int_{\delta}^{2\delta}\!\!\!\int_{\mathbb{R}^d} |f|((t\pm \delta)\pm s,x\pm y) \eta_{\alpha,\alpha-1}(s,y) \dif y \dif s\\
    &\leq 2^{(\alpha-1)/2}\delta^{-(d+1)/2}K^f_\alpha(2\delta)\preceq \delta^{-(d+1)/2}K^f_\alpha(\delta).
  \end{align*}
  Similar, we have
  $$
  \int_0^{\delta}\!\!\!\int_{\mathbb{R}^d} |f|(t\pm s,x\pm y) \dif y \dif s\preceq\delta^{-(d+\alpha)/2}\bar K^f_\alpha(\delta).
  $$
 (ii) The inclusions of $\mK_1\subset\mK_\alpha\subset\mK_\infty$ and $\bar\mK_\alpha\subset\mK_\alpha$ follow by definitions and \eqref{e219}. Let us prove
 $L^q(\mR;L^p(\mR^d))\subset\cap_{\alpha\geq 1}\mK_\alpha$.
 Let $f\in L^q(\mR;L^p(\mR^d))$. By H\"older's inequality, we have
  $$
    K^f_\alpha(\delta)\leq\Bigg(\int_{\mR}\left(\int_{\mR^d}|f(s,y)|^p\dif y\right)^{\frac{q}{p}}\dif s\Bigg)^{\frac{1}{q}}
    I_\alpha(\delta),
  $$
  where
  $$
    I_\alpha(\delta):=\Bigg(\int^\delta_0\!\Bigg(\!\int_{\mR^d}\frac{s^{(\alpha-1)p^*/2}\dif y}
    {\big(|y|+s^{1/2}\big)^{(d+\alpha)p^{*}}}\Bigg)^{\frac{q^*}{p^*}}\dif s\Bigg)^{\frac{1}{q^{*}}},
  $$
  with $q^*:=\frac{q}{q-1}$ and $p^*:=\frac{p}{p-1}$. Noticing that
  \begin{align*}
     \int_{\mR^d}\frac{s^{(\alpha-1)p^*/2}\dif y}{\big(|y|+s^{1/2}\big)^{(d+\alpha)p^{*}}}
     &\leq s^{(\alpha-1)p^*/2}\left(\int_{|y|\leq s^{1/2}}s^{-\frac{(d+\alpha)p^{*}}{2}}\dif y
     + \int_{|y|> s^{1/2}}\frac{\dif y}{|y|^{(d+\alpha)p^{*}}}\right)\preceq s^{\frac{d-(d+1)p^{*}}{2}},
  \end{align*}
  we have
  $$
    I_\alpha(\delta)
    \preceq\Bigg(\int^\delta_0s^{\frac{dq^*}{2 p^*}-\frac{(d+1)q^*}{2}}\dif s\Bigg)^{\frac{1}{q^{*}}}.
  $$
  Thus  $I_\alpha(\delta)$ converges to zero as $\delta\to 0$ provided that
  $$
    \frac{dq^*}{2 p^*}-\frac{d+1}{2}q^*+1>0\Leftrightarrow\frac{d}{p}+\frac{2}{q}<1.
  $$
  Similarly, we can show that $L^q(\mR;L^p(\mR^d))\subset\bar\mK_{\alpha}$ provided $\frac{d}{p}+\frac{2}{q}<2-\alpha$.
\end{proof}

Next we study the mollifying approximation of $f\in\mK_\alpha$. Let $\rho(t,x): \mR^{d+1}\to [0,1]$ be a smooth function with support in the unit ball and $\int\rho=1$. For $\eps\in(0,1)$, define a family of mollifiers $\rho_\eps$ as follows:
$$
  \rho_\eps(t,x):=\eps^{-d-1}\rho(\eps^{-1} t, \eps^{-1} x).
$$
For $f\in \mK_\alpha$, we define
\begin{align}
  f_\eps(t,x):=f*\rho_\eps(t,x)=\int_{\mR^{d+1}}f(s,y)\rho_\eps(t-s,x-y)\dif y\dif s.   \label{n}
\end{align}
By Fubini's theorem, it is easy to see that
\begin{align}
  K^{f_\eps}_\alpha(\delta)\leq K^{f}_\alpha(\delta).\label{ER1}
\end{align}

\bl\label{Le28}
  For $\alpha \in [1,\infty)$ and $f\in\mK_\alpha$, we have
  \begin{align*}
    \lim_{\eps\to 0}\int^1_0\!\!\!\int_{\mR^d} |f_\eps-f|(t\pm s,x\pm y)\eta_{\alpha,\alpha-1}(s,y)\dif y \dif s=0.
  \end{align*}
\el
\begin{proof}
  First of all, notice that
  \begin{align*}
    \lim_{\delta\to 0}\sup_{\eps\in(0,1)}\int^\delta_0\!\!\!\int_{\mR^d} |f_\eps-f|(t\pm s,x\pm y) \eta_{\alpha,\alpha-1}(s,y)\dif y \dif s \stackrel{\eqref{ER1}}{\leq} 2\lim_{\delta\to 0}K^{f}_\alpha(\delta)=0.
  \end{align*}
  So, it suffices to prove that for fixed $\delta\in(0,1)$,
  \begin{align}
    \lim_{\eps\to 0}\int^1_\delta\!\!\!\int_{\mR^d} \frac{|f_\eps-f|(t\pm s,x\pm y) s^{(\alpha-1)/2}}{(|y|+s^{1/2})^{d+\alpha}} \dif y\dif s=0.\label{RT6}
  \end{align}
  Let $f^n(t,x):=(-n)\vee f(t,x)\wedge n$ and $f^n_\eps:=f^n*\rho_\eps$. Since $\rho_\eps$ has support in $\{(t,x): |(t,x)|\leq\eps\}$, by the definition of convolution, we have
  \begin{align}
\sup_{\eps\in(0,\delta/4)}\int^1_{\delta}\!\!\!&\int_{\mR^d} \frac{|f^n_\eps-f_\eps|(t\pm s,x\pm y) s^{(\alpha-1)/2}}{(|y|+s^{1/2})^{d+\alpha}}\dif y\dif s\notag\\
    \leq&\, \int^{1+\delta/4}_{3\delta/4}\!\!\!\int_{\mR^d} \frac{|f^n-f|(t\pm s,x\pm y) (s+\delta/4)^{(\alpha-1)/2}}{(|y|-\delta/4+(s-\delta/4)^{1/2})^{d+\alpha}}\dif y\dif s\notag\\
    \leq&\, \int^{1+\delta/4}_{3\delta/4}\!\!\!\int_{\mR^d} \frac{|f^n-f|(t\pm s,x\pm y) (2s)^{(\alpha-1)/2}}{(|y|+s^{1/2}/4)^{d+\alpha}}\dif y\dif s,\label{RT7}
  \end{align}
  which converges to zero as $n\to\infty$ by the dominated convergence theorem. On the other hand, for fixed $n\in\mN$, since
  $\lim_{\eps\to 0} f^n_\eps = f^n$ a.e.,
   by the bounded convergence theorem, we have
  $$
    \lim_{\eps\to 0}\int^1_{\delta}\!\!\!\int_{\mR^d} \frac{|f^n_\eps-f^n|(t\pm s,x\pm y) s^{(\alpha-1)/2}}{(|y|+s^{1/2})^{d+\alpha}}\dif y\dif s=0.
  $$
  Combining this with \eqref{RT7}, we obtain \eqref{RT6}.
\end{proof}

\section{Proof of Theorem \ref{main1}}

In the remaining part of this paper, we shall fix $\alpha\in(0,2)$ and assume {\bf (H$^a$)}, {\bf (H$^\kappa$)} and $b\in\mK_2$.
Below, a function $f(t,x)$ on $[0,\infty)\times\mR^d$ will be automatically extended to $\mR\times\mR^d$ by letting $f(t,\cdot)=0$ for $t<0$.
Notice that
\begin{align}
  \sL^{b,\kappa}_t:=b_t\cdot \nabla+\sL^\kappa_t= {\tilde b}_t\cdot \nabla +{\widetilde\sL}^{\kappa(t,\cdot)}, \label{LLL}
\end{align}
where ${\widetilde\sL}^{\kappa(t,\cdot)}$ is defined by \eqref{j}, and
\begin{equation}\label{eq:tildebDef}
  \tilde b(t,x):=b(t,x)+1_{\alpha\in(1,2)}\int_{|z|>1}z\kappa(t,x,z)|z|^{-d-\alpha}\dif z
  -1_{\alpha\in(0,1)} \int_{|z|\leq 1}z\kappa(t,x,z)|z|^{-d-\alpha}\dif z.
\end{equation}
By definition, it is easy to see that for some $c=c(d,\alpha)>0$,
  \begin{equation}\label{kbkb}
    K^{\tilde{b}}_2(r) \leq K^b_2(r) + c\|\kappa\|_\infty r^{{1 /2}},\quad r> 0.
  \end{equation}
Let $Z(t,x;s,y)$ be the heat kernel of $\sL^a_t$ constructed in Theorem \ref{T23}.
We will construct the fundamental solution $p(t, x; s, y)$ of $\sL_t$ by using
Duhamel's formula \eqref{eqdu}.
To solve that integral equation, let $p_0(t,x;s,y):=Z(t,x;s,y)$, and for $n\in\mN$, define
\begin{align}
  p_n(t,x;s,y):=\int^s_t\!\!\!\int_{\mR^d}p_{n-1}(t,x;r,z)\sL^{b,\kappa}_r Z(r,\cdot;s,y)(z)\dif z\dif r.\label{ET1}
\end{align}

We first prepare the following lemma for later use.
\bl\label{lem:3p}
  For  any $\lambda>0$ and $j=0,1$, there exists a constant $C_j=C_j(d,\alpha,\lambda)>0$ such that for all $(t,x;s,y)\in\mD^\infty_0$,
  \begin{align}
  \begin{split}\label{UY1}
    &\int^s_t\!\!\!\int_{\mR^d}\eta_{\alpha,2-j}(r-t,z-x)|b(r,z)|\xi_{\lambda,-1}(s-r,y-z)\dif z\dif r\\
     &\qquad\qquad\leq C_jK^b_2(s-t)\eta_{\alpha,2-j}(s-t,y-x),
    \end{split}
    \end{align}
    and
    \begin{align}
  \begin{split}\label{UY2}
    &\int^s_t\!\!\!\int_{\mR^d}\xi_{\lambda,-j}(r-t,z-x)|b(r,z)| \xi_{2\lambda,-1}(s-r,y-z)\dif z\dif r\\
    &\qquad\qquad\leq C_jK^b_2(s-t)\xi_{\lambda,-j}(s-t,y-x),
    \end{split}
  \end{align}
  where $K^b_2(s-t)$ is defined by \eqref{Def5}.
\el
\begin{proof}
  Notice that by \eqref{eqcom},
  \begin{align*}
    \xi_{\lambda,-j}(s-t,y-x)\preceq \eta_{2,2-j}(s-t,y-x),\quad j=0,1.
  \end{align*}
  Hence, by \eqref{eq3p}, we have
  \begin{align*}
    &\int^s_t\!\!\!\int_{\mR^d}\eta_{\alpha,2-j}(r-t,z-x) ~|b(r,z)|~ \xi_{\lambda,-1}(s-r,y-z)\dif z\dif r\\
    &\preceq\eta_{\alpha,0}(s-t,y-x)\int^s_t\!\!\!\int_{\mR^d}(r-t)^{(2-j)/2}(s-r)^{1/2} |b(r,z)|\\
    &\qquad\times\Big(\eta_{2,0}(r-t,z-x)+\eta_{2,0}(s-r,y-z)\Big)\dif z\dif r\\
    &\preceq\eta_{\alpha,2-j}(s-t,y-x)\int^s_t\!\!\!\int_{\mR^d}|b(r,z)|\\
    &\qquad\times\Big(\eta_{2,1}(r-t,z-x)+\eta_{2,1}(s-r,y-z)\Big)\dif z\dif r\\
    &\leq C_jK^b_2(s-t)\eta_{\alpha,2-j}(s-t,y-x),
  \end{align*}
  where the last step is due to the change of variables and the definition of $K^b_2$. Thus  \eqref{UY1} is proved. Estimate \eqref{UY2} follows from \cite[Lemma 3.1]{Zhang.1997.MM381} and \eqref{e219}.
\end{proof}

\bl\label{Le32}
  For each $n\in\mN$ and $j=0,1$, $\nabla^j_xp_n(t,x;s,y)$ is a jointly continuous function on $\mD^\infty_0$, and for any $T>0$, there exist constants $c,\lambda>0$ such that for all $n\in\mN$ and $(t,x;s,y)\in\mD^T_0$,
  \begin{align}
  \begin{split}
    |\nabla^j_x p_n(t,x;s,y)|&\leq c (c\ell_{b,\kappa}(s-t))^{n-1} \|\kappa\|_\infty\eta_{\alpha,2-j}(s-t,y-x)\\
    &\quad+(c\ell_{b,\kappa}(s-t))^{n}\xi_{\lambda,-j}(s-t,y-x),
    \end{split}\label{ep2}
  \end{align}
  where $\ell_{b,\kappa}(r):=\|\kappa\|_\infty \left(r^{1-\frac{\alpha}{2}} + r^{{1 /2}}\right)+K^b_2(r)$.
\el
\begin{proof}
  (1) First of all, by definition, {\bf (H$^\kappa$)}, (\ref{ET31}) and \eqref{eq21}, there is a $\lambda>0$ such that
  \begin{align}
    |\widetilde\sL^{\kappa(t,\cdot)} Z(t,\cdot;s,y)(x)|\preceq \|\kappa\|_\infty \eta_{\alpha,0}(s-t,y-x),\quad |\nabla_xZ(t,x;s,y)|\preceq \xi_{2\lambda,-1}(s-t,y-x).\label{ep4}
  \end{align}
 For $r>0$, let
  \begin{equation*}
    \widetilde{\ell}_{b,\kappa}(r) := \|\kappa\|_\infty r^{1-\frac{\alpha}{2}} + K^{\tilde{b}}_2(r),
  \end{equation*}
where $\tilde b$ is defined by (\ref{eq:tildebDef}). In view of (\ref{LLL}) and (\ref{kbkb}), it is enough to prove \eqref{ep2} with $\widetilde{\ell}_{b,\kappa}$ instead of $\ell_{b,\kappa}$. For $n=1$, by \eqref{UY3} and \eqref{UY2} we have
  \begin{align*}
    |\nabla^j_xp_1(t,x;s,y)|&\preceq \!\!\!\int^s_t\!\!\!\int_{\mR^d}\xi_{\lambda,-j}(r-t,z-x) \Big(|\tilde{b}(r,z)|\cdot\xi_{2\lambda,-1}(s-r,y-z)\\
    &\qquad\qquad+\|\kappa\|_\infty\eta_{\alpha,0}(s-r,y-z)\Big)\dif z\dif r\\
    &\leq c_0\widetilde{\ell}_{b,\kappa}(s-t)\cdot\xi_{\lambda,-j}(s-t,y-x)+c_1\|\kappa\|_\infty\eta_{\alpha,2-j}(s-t,y-x).
  \end{align*}
  Suppose that  \eqref{ep2} holds for $\widetilde{\ell}_{b,\kappa}$ and for some $n\in\mN$. By \eqref{ep4} and the induction hypothesis, we have
  \begin{align*}
   &\quad  |\nabla^j_xp_{n+1}(t,x;s,y)|\\
   &\preceq  \int^s_t\!\!\!\int_{\mR^d} \!\! \Big(c\big(c\widetilde{\ell}_{b,\kappa}(r-t)\big)^{n-1}\|\kappa\|_\infty\eta_{\alpha,2-j}(r-t,z-x)
    +(c \widetilde{\ell}_{b,\kappa}(r-t))^n\xi_{\lambda,-j}(r-t,z-x)\Big)\\
    &\qquad \times \left(|\tilde{b}(r,z)| \cdot\xi_{2\lambda,-1}(s-r,y-z) + \|\kappa\|_\infty\eta_{\alpha,0}(s-r,y-z)\right) \dif z\dif r\\
    &\preceq  c(c \widetilde{\ell}_{b,\kappa}(s-t))^{n-1}\|\kappa\|_\infty(I_1+I_2)+(c \widetilde{\ell}_{b,\kappa}(s-t))^{n}(I_3+I_4),
  \end{align*}
  where
  \begin{align*}
    I_1&:=\int^s_t\!\!\!\int_{\mR^d}\eta_{\alpha,2-j}(r-t,z-x)\cdot |\tilde{b}(r,z)|\cdot \xi_{2\lambda,-1}(s-r,y-z)\dif z\dif r,\\ I_2&:=\|\kappa\|_\infty\int^s_t\!\!\!\int_{\mR^d}\eta_{\alpha,2-j}(r-t,z-x)\cdot \eta_{\alpha,0}(r-t,z-x)\dif z\dif r,\\
    I_3&:=\int^s_t\!\!\!\int_{\mR^d}\xi_{\lambda,-j}(r-t,z-x)\cdot |\tilde{b}(r,z)|\cdot \xi_{2\lambda,-1}(s-r,y-z)\dif z\dif r,\\
    I_4&:=\|\kappa\|_\infty\int^s_t\!\!\!\int_{\mR^d}\xi_{\lambda,-j}(r-t,z-x)\cdot \eta_{\alpha,0}(r-t,z-x)\dif z\dif r.
  \end{align*}
  By \eqref{UY1} and \eqref{3P}, one sees that
  \begin{align*}
    I_1+I_2\preceq \widetilde{\ell}_{b,\kappa}(s-t)\cdot\eta_{\alpha,2-j}(s-t,y-x),
  \end{align*}
  and by \eqref{UY2} and \eqref{UY3},
  \begin{align*}
    I_3\preceq \widetilde{\ell}_{b,\kappa}(s-t)\cdot\xi_{\lambda,-j}(s-t,y-x),\ \ I_4&\preceq  \|\kappa\|_\infty\eta_{\alpha,2-j}(s-t,y-x).
  \end{align*}
  Therefore,
  \begin{align*}
     |\nabla^j_xp_{n+1}(t,x;s,y)|
     &\leq  c_2c(c \widetilde{\ell}_{b,\kappa}(s-t))^{n-1}\widetilde{\ell}_{b,\kappa}(s-t) \cdot \|\kappa\|_\infty
       \eta_{\alpha,2-j}(s-t,y-x)\\
    & +(c \widetilde{\ell}_{b,\kappa}(s-t))^{n}\Big(c_3\widetilde{\ell}_{b,\kappa}(s-t)\cdot\xi_{\lambda,-j}(s-t,y-x) + c_4\|\kappa\|_\infty\eta_{\alpha,2-j}(s-t,y-x)\Big)\\
    &\leq  (c_2+c_4)(c \widetilde{\ell}_{b,\kappa}(s-t))^{n}\cdot\|\kappa\|_\infty\eta_{\alpha,2-j}(s-t,y-x)\\
    &\quad +c_3(c \widetilde{\ell}_{b,\kappa}(s-t))^{n}\widetilde{\ell}_{b,\kappa}(s-t)\cdot\xi_{\lambda,-j}(s-t,y-x).
  \end{align*}
  Finally, by choosing $c=c_0\vee c_1\vee (c_2+c_4)\vee c_3$, we obtain the desired result.
  \\
  \vbox{~}
  (2) We use induction to show the joint continuity of $\nabla^j_xp_n$. First of all, by Theorem \ref{T23}, $\nabla^j_xp_0$ is jointly continuous. Suppose now that $\nabla^j_xp_{n-1}$ is jointly continuous for some $n\in\mN$. For fixed $\eps>0$ and $\delta\in(0,\eps/2)$, we write
  \begin{align*}
    \nabla^j_xp_n(t,x;s,y)&=\left(\int^s_{s-\delta}+\int^{t+\delta}_t\right)\! \int_{\mR^d}\nabla^j_xp_{n-1}(t,x;r,z) \sL^{b,\kappa}_r Z(r,\cdot;s,y)(z)\dif z\dif r\\
    &\quad+\int^{s-\delta}_{t+\delta}\!\!\!\int_{\mR^d}\nabla^j_xp_{n-1}(t,x;r,z)\sL^{b,\kappa}_r Z(r,\cdot;s,y)(z)\dif z\dif r\\
    &=:I_1(\delta,t,x;s,y)+I_2(\delta,t,x;s,y).
  \end{align*}
  For $I_1$, as in step (1), there is a $C_\eps>0$ such that for all $(t,x;s,y)\in\mD^T_\eps$ and $\delta\in(0,\eps/2)$,
  \begin{align}\label{RT8}
    |I_1(\delta,t,x;s,y)|\leq C_\eps\ell_{b,\kappa}(\delta).
  \end{align}
  For $I_2$, by the dominated convergence theorem and induction hypothesis, one sees that for fixed $\delta\in(0,\eps/2)$,
  $$
    (t,x;s,y)\mapsto I_2(\delta,t,x;s,y)\mbox{ is continuous on $\mD^T_\eps$.}
  $$
  Combining this with \eqref{RT8}, we obtain the continuity of $\nabla^j_xp_n$ on $\mD^T_\eps$.
  Since $\eps, T>0$ are arbitrary, $\nabla^j_xp_n$ is jointly continuous on $\mD^\infty_0$. The proof is complete.
\end{proof}

\bl\label{Th2}
  Let $J:\mR^d\to\mR$ be a bounded measurable function satisfying \eqref{Sym1}.
  \begin{enumerate}[\rm (i)]
    \item If $\alpha\in(0,1]$ and $b\in\mK_1$, then for any $T>0$, there is a constant $c>0$ such that for all $n\in\mN_0$ and $(t,x;s,y)\in\mD^T_0$,
        \begin{align}
          |\widetilde\sL^Jp_n(t,\cdot;s,y)(x)|\leq c(c\ell_{b,\kappa}(s-t))^{n}\|J\|_\infty\eta_{\alpha,0}(s-t,y-x),\label{UT01}
        \end{align}
        where $\ell_{b,\kappa}(r)=\|\kappa\|_\infty \left(r^{1-\frac{\alpha}{2}}+r^{{1 /2}}\right)+K^b_1(r)$.
    \item If $\alpha\in(1,2)$ and $b\in\bar\mK_{\alpha}$, then \eqref{UT01} still holds with
        $\ell_{b,\kappa}(r):=\|\kappa\|_\infty \left(r^{1-\frac{\alpha}{2}}+r^{{1 /2}}\right) + r^{\frac{\alpha-1}{2}}\bar K^b_{\alpha}(1)$.
  \end{enumerate}
  Moreover, for the above two classes of Kato's functions $b$, we have
  \begin{align}
    p_{n+1}(t,x;s,y)=\int^s_t\!\!\!\int_{\mR^d}Z(t,x;r,z)\sL^{b,\kappa}_r p_{n}(r,\cdot;s,y)(z)\dif z\dif r. \label{equ}
  \end{align}
\el
\begin{proof}
  As before, we set for $r>0$
  \begin{equation*}
    \widetilde{\ell}_{b,\kappa}(r) := \|\kappa\|_\infty r^{1-\frac{\alpha}{2}} + K^{\tilde{b}}_2(r).
  \end{equation*}
 By (\ref{LLL}) and (\ref{kbkb}), we only need to prove \eqref{UT01} with $\widetilde{\ell}_{b,\kappa}$ instead of $\ell_{b,\kappa}$.
 By \eqref{ET31}, one sees that \eqref{UT01} and \eqref{equ} hold for $n=0$. Now suppose that \eqref{UT01} and \eqref{equ} hold for $\widetilde{\ell}_{b,\kappa}$ and for some $n\in\mN_0$. By Fubini's theorem,  \eqref{ET31}, \eqref{ep4} and (\ref{3P}), we have
  \begin{align*}
    |\widetilde\sL^J_tp_{n+1}(t,\cdot;s,y)(x)|&=\left|\int^s_t\!\!\!\int_{\mR^d}\widetilde\sL^J_t p_n(t,\cdot;r,z)(x) \sL^{b,\kappa}_rZ(r,\cdot;s,y)(z)\dif z\dif r\right|\\
    &\preceq \int^s_t\!\!\!\int_{\mR^d}(c\widetilde{\ell}_{b,\kappa}(r-t))^n \|J\|_\infty\eta_{\alpha,0}(r-t,z-x)\\
    &\quad\times \Big(|\tilde{b}(r,z)|\cdot\xi_{2\lambda,-1}(s-r,y-z)+ \|\kappa\|_\infty \eta_{\alpha,0}(s-r,y-z)\Big)\dif z\dif r\\
    &= (c\widetilde{\ell}_{b,\kappa}(s-t))^n\|J\|_\infty(I_1+I_2),
  \end{align*}
  where
  \begin{align*}
    I_1&:=\int^s_t\!\!\!\int_{\mR^d}\eta_{\alpha,0}(r-t,z-x)\cdot|\tilde{b}(r,z)|\cdot\xi_{2\lambda,-1}(s-r,y-z)\dif z\dif r,\\
    I_2&:= \|\kappa\|_\infty \int^s_t\!\!\!\int_{\mR^d}\eta_{\alpha,0}(r-t,z-x)\cdot\eta_{\alpha,0}(s-r,y-z)\dif z\dif r.
  \end{align*}
  For $I_1$, by \eqref{eqcom} and \eqref{eq3p}, we have
  \begin{align*}
    I_1&\preceq (s-t)^{(\alpha\vee 1 -1)/2} \int^s_t\!\!\!\int_{\mR^d}\eta_{\alpha,0}(r-t,z-x)\cdot|\tilde{b}(r,z)|\cdot\eta_{\alpha\vee 1,0}(s-r,y-z)\dif z\dif r\\
    &\preceq \eta_{\alpha,\alpha\vee 1-1}(s-t,y-x)\int^s_t\!\!\!\int_{\mR^d}|\tilde{b}(r,z)|\Big(\eta_{\alpha\vee 1,0}(r-t,z-x)+\eta_{\alpha\vee 1,0}(s-r,y-z)\Big)\dif z\dif r.
  \end{align*}
  If $\alpha\in(0,1]$, then
  $$
    I_1\preceq\eta_{\alpha,0}(s-t,y-x)K^{\tilde{b}}_{1}(s-t) .
  $$
  If $\alpha\in(1,2)$, then
  $$
    I_1\preceq\eta_{\alpha,\alpha-1}(s-t,y-x)\bar K^{\tilde{b}}_{\alpha}(s-t) =  (s-t)^{\frac{\alpha-1}{2}} \bar K^{\tilde{b}}_{\alpha}(s-t) \eta_{\alpha,0}(s-t,y-x).
  $$
  For $I_2$, by \eqref{3P} we have
  $$
    I_2\preceq \|\kappa\|_\infty \eta_{\alpha, 2-\alpha}(s-t,y-x) = \|\kappa\|_\infty (s-t)^{1-\frac{\alpha}{2}} \eta_{\alpha,0}(s-t,y-x).
  $$
  Combining the above calculations, we obtain \eqref{UT01}.

  Moreover, by Fubini's theorem again and the induction hypothesis, we have
  \begin{align*}
    p_{n+2}(t,x;s,y)&=\int^s_t\!\!\!\int_{\mR^d}p_{n+1}(t,x;r,z)\sL^{b,\kappa}_rZ(r,\cdot;s,y)(z)\dif z\dif r\\ &=\int^s_t\!\!\!\int_{\mR^d}\!\int^r_t\!\!\!\int_{\mR^d}Z(t,x;r',z')\sL^{b,\kappa}_{r'}p_n(r',\cdot;r,z)(z')\dif z'\dif r'\\
    &\qquad\qquad\qquad\times\sL^{b,\kappa}_rZ(r,\cdot;s,y)(z)\dif z\dif r\\ &=\int^s_t\!\!\!\int_{\mR^d}Z(t,x;r',z') \int^s_{r'}\!\!\!\int_{\mR^d}\sL^{b,\kappa}_{r'}p_n(r',\cdot;r,z)(z')\\
    &\qquad\qquad\qquad\times\sL^{b,\kappa}_rZ(r,\cdot;s,y)(z)\dif z\dif r\dif z'\dif r'\\
    &=\int^s_t\!\!\!\int_{\mR^d}Z(t,x;r',z')\sL^{b,\kappa}_rp_{n+1}(r',\cdot;s,y)(z')\dif z'\dif r'.
  \end{align*}
  The proof is complete.
\end{proof}

Under additional regularity assumptions on $b$ and $\kappa$, we can show further regularity of $p_n(t,\!x;s,\!y)$ as given in the following lemma.

\bl
  If $b$ and $\kappa$ are bounded measurable and for some $C_0>0$ and $\gamma\in(0,1)$,
  \begin{align}\label{RT9}
    |b(t,y)-b(t,x)|+|\kappa(t,y,z)-\kappa(t,x,z)|\leq C_0|y-x|^\gamma,\quad t\in \mathbb{R}_+,~x,y,z\in \mathbb{R}^d,
  \end{align}
  then $\nabla^2_x p_n$ is continuous on $\mD^\infty_0$, and for any $T>0$, there are two constants $\lambda > 0$ and $C_1>0$ such that for all $n \in \mathbb{N}$ and $(t,x;s,y)\in\mD^T_0$,
  \begin{align}
    |\nabla^2_x p_n(t,x;s,y)|\leq C_1\Big(C_1(s-t)^{\frac{(2-\alpha)\wedge 1}{2}}\Big)^{n-1}(\eta_{\alpha,0}+\xi_{\lambda/2,-1})(s-t,y-x).\label{UT1}
  \end{align}
\el

\begin{proof}
  Below, we always assume that $0<s-t\leq T$ and $x_1,x_2,x,y \in \mathbb{R}^d$.\\
  (1) By \eqref{eq21}, \eqref{225} and \eqref{RT9}, we have
  \begin{align*}
    |{\tilde b}_t\cdot \nabla Z(t,\cdot;s,y)(x_1)-{\tilde b}_t \cdot \nabla Z(t,\cdot;s,y)(x_2)| \preceq |x_1-x_2|^\gamma\Big(\sum_i\xi_{\lambda/2,-\gamma-1}(s-t, y-x_i)\Big),
  \end{align*}
  and by \eqref{ET32}, \eqref{ET31} and \eqref{RT9},
  \begin{align*}
   |\widetilde\sL^{\kappa(t,\cdot)}Z(t,\cdot;s,y)(x_1)-\widetilde\sL^{\kappa(t,\cdot)} Z(t,\cdot;s,y)(x_2)| \preceq |x_1-x_2|^\gamma\Big(\sum_i\eta_{\alpha,-\gamma}(s-t, y-x_i)\Big).
  \end{align*}
  Let $H_{s,y}(t,x):=\sL^{b,\kappa}_t Z(t,\cdot;s,y)(x)$. By the above two estimates, we have
  \begin{align}\label{NJ1}
  |H_{s,y}(t,x_1)-H_{s,y}(t,x_2)|\preceq |x_1-x_2|^\gamma \Big(\sum_i(\eta_{\alpha,-\gamma}+\xi_{\lambda/2,-\gamma-1})(s-t, y-x_i)\Big).
  \end{align}
  Moreover, by \eqref{eq21}    and \eqref{ET31}, we also have
  \begin{align}\label{NJ2}
    |H_{s,y}(t,x)|\preceq (\eta_{\alpha,0}+\xi_{\lambda,-1})(s-t,y-x).
  \end{align}
  (2) We use induction to prove \eqref{UT1}. First of all, for $n=1$, by \eqref{VV} below, we have
  \begin{align*}
    \nabla^2_xp_1(t,x;s,y)&=\int^s_t\!\!\!\int_{\mR^d}\nabla^2_x p_{0}(t,x;r,z)H_{s,y}(r,z)\dif z\dif r=I_1+I_2+I_3,
    \end{align*}
  where
  \begin{align*}
    I_1&:=\int^s_{\frac{s+t}{2}}\!\int_{\mR^d}\nabla^2_x Z(t,x;r,z)H_{s,y}(r,z)\dif z\dif r,\\
    I_2&:=\int^{\frac{s+t}{2}}_t\!\!\!\int_{\mR^d}\nabla^2_x Z(t,x;r,z)(H_{s,y}(r,z)-H_{s,y}(r,x))\dif z\dif r,\\
    I_3&:=\int^{\frac{s+t}{2}}_t\left(\int_{\mR^d}\nabla^2_x Z(t,x;r,z)\dif z\right)H_{s,y}(r,x)\dif r.
  \end{align*}
  For $I_1$, by \eqref{eq21}, \eqref{NJ2}, \eqref{UY3} and \eqref{UY30}, we have
  \begin{align*}
    |I_1|&\preceq\int^s_{\frac{s+t}{2}}\!\int_{\mR^d}\xi_{\lambda,-2}(r-t,z-x)(\eta_{\alpha,0}+\xi_{\lambda,-1})(s-r,y-z)\dif z\dif r\\ &\preceq(s-t)^{-1}\int^s_t\!\!\!\int_{\mR^d}\xi_{\lambda,0}(r-t,z-x)(\eta_{\alpha,0}+\xi_{\lambda,-1})(s-r,y-z)\dif z\dif r\\
    &\preceq\eta_{\alpha,0}(s-t,y-x)+\xi_{\lambda,-1}(s-t,y-x).
  \end{align*}
  For $I_2$, by \eqref{eq21}, \eqref{NJ1}, \eqref{UY3} and \eqref{UY30}, we similarly have
  \begin{align*}
    |I_2|&\preceq\int^{\frac{s+t}{2}}_t\!\!\!\int_{\mR^d}\xi_{\lambda,-2}(r-t,z-x)|x-z|^\gamma \cdot\Big((\eta_{\alpha,-\gamma}+\xi_{\lambda,-\gamma-1})(s-r, y-z)\\
    &\qquad\qquad\qquad+(\eta_{\alpha,-\gamma}+\xi_{\lambda,-\gamma-1})(s-r, y-x)\Big)\dif z\dif r\\ &\preceq \int^{\frac{s+t}{2}}_t\!\!\!\int_{\mR^d} \xi_{\lambda/2,\gamma-2}(r-t,z-x)\cdot\Big((\eta_{\alpha,-\gamma}+\xi_{\lambda,-\gamma-1})(s-r, y-z)\\
    &\qquad\qquad\qquad+(\eta_{\alpha,-\gamma}+\xi_{\lambda,-\gamma-1})(s-r,y-x)\Big)\dif z\dif r\\
    &\preceq\eta_{\alpha,0}(s-t,y-x)+\xi_{\lambda/2,-1}(s-t,y-x).
  \end{align*}
  For $I_3$, by \eqref{00} below and \eqref{NJ2}, we have
  \begin{align*}
    |I_3|\preceq\left(\int^{\frac{s+t}{2}}_t(r-t)^{\frac{\beta}{2}-1}\dif r\right)(\eta_{\alpha,0}+\xi_{\lambda,-1})(s-t,y-x) \preceq(\eta_{\alpha,0}+\xi_{\lambda,-1})(s-t,y-x).
  \end{align*}
  Combining the above calculations, we obtain \eqref{UT1} for $n=1$.
  \\\vbox{~}
  (3) Suppose \eqref{UT1} holds for some $n\in\mN$. By the induction hypothesis, \eqref{NJ2} and Lemma \ref{Le21},
  \begin{align*}
    &|\nabla^2_xp_{n+1}(t,x;s,y)|\leq C_1(C_1(s-t)^{\frac{(2-\alpha)\wedge 1}{2}})^{n-1}\\ &\quad\times\int^s_t\!\!\!\int_{\mR^d}(\eta_{\alpha,0}+\xi_{\lambda/2,-1})(r-t,z-x) (\eta_{\alpha,0}+\xi_{\lambda,-1})(s-r,y-z)\dif z\dif r\\ &\quad\leq C_1(C_1(s-t)^{\frac{(2-\alpha)\wedge 1}{2}})^{n}(\eta_{\alpha,0}+\xi_{\lambda/2,-1})(s-t,y-x).
  \end{align*}
  Thus  we obtain \eqref{UT1}.
  \\
  \vbox{~}
  (4) The joint continuity of $\nabla^2_xp_n$ follows by the same argument as in Lemma \ref{Le32}. The proof is complete.
\end{proof}

Now we can prove the solvability of the integral equation \eqref{eqdu}.

\bt\label{Th1}
Under {\bf (H$^a$)}, {\bf (H$^\kappa$)} and $b\in\mK_2$, there exists a $\delta>0$ so that \eqref{eqdu} has a unique continuous solution $p(t,x;s,y)$ on $\mD^\delta_0$
such that
  \begin{align}
    |p(t,x;s,y)|\leq C_1(\xi_{\lambda,0}+\|\kappa\|_\infty\eta_{\alpha,2})(s-t,y-x)
    \quad \hbox{ on } \  \mD^{\delta}_0 \label{RT1}
  \end{align}
  for some constant $C_1>0$.
  Moreover,  the following hold.
  \begin{enumerate} [\rm (i)]
    \item (Gradient estimate) $\nabla_xp$ is continuous on $\mD^{\delta}_0$ and for some $C_2>0$ and all $(t,x;s,y)\in\mD^{\delta}_0$,
        \begin{align}
          |\nabla_x p(t,x;s,y)|\leq C_2(\xi_{\lambda,-1}+\|\kappa\|_\infty\eta_{\alpha,1})(s-t,y-x).  \label{RT2}
        \end{align}
    \item (On-diagonal lower bound estimate) There is a constant $C_3>0$ such that for all $|y-x|\leq\sqrt{s-t}<\delta$,
        \begin{align}\label{On}
          p(t,x;s,y)\geq C_3 (s-t)^{-d/2}.
        \end{align}
    \item (C-K equation) For all $(t,x;s,y)\in\mD^{\delta}_0$ and $r\in(t,s)$, the following Chapman-Kolmogorov equation holds:
        \begin{align}\label{CK}
          \int_{\mR^d}p(t,x;r,z)p(r,z;s,y)\dif z=p(t,x;s,y).
        \end{align}
    \item (Fractional derivative estimate) Let $J:\mR^d\to\mR$ be a bounded measurable function satisfying \eqref{Sym1}. If in addition for $\alpha\in(0,1]$, $b\in \mK_1$ and for $\alpha\in(1,2)$, $b\in\bar\mK_\alpha$, then for some $C_4>0$ and all $(t,x;s,y)\in\mD^{\delta}_0$,
        \begin{align}\label{RT3}
          |\widetilde\sL^Jp(t,\cdot;s,y)(x)|\leq C_4\|J\|_\infty \eta_{\alpha,0}(s-t,y-x)
        \end{align}
        and
        \begin{align}
          p(t,x;s,y)=Z(t,x;s,y)+\int^s_t\!\!\!\int_{\mR^d}Z(t,x;r,z)\sL^{b,\kappa}_r p(r,\cdot;s,y)(z)\dif z\dif r. \label{eq}
        \end{align}
    \item (Second order derivative estimate) If $b$ and $\kappa$ are bounded measurable and satisfy \eqref{RT9}, then $\nabla^2_xp$ is continuous on $\mD^{\delta}_0$ and for some $C_5>0$,
        \begin{align}
         |\nabla^2_xp(t,x;s,y)|\leq C_5(\xi_{\lambda,-2} + \|\kappa\|_\infty \eta_{\alpha,0})(s-t,y-x).\label{SEC}
        \end{align}
    \end{enumerate}
\et

\begin{proof}
  (i) Let $\ell_{b,\kappa}(r)$ and the constant $c$ be as in Lemma  \ref{Le32} with $T = 1$. In view of $\lim_{\delta\to 0} \ell_{b,\kappa}(\delta)$ $=$ $0$, one can choose a $\delta_1>0$ such that $c\ell_{b,\kappa}(\delta_1)\leq 1/2$. Thus  by \eqref{ep2}, the series $p(t,x;s,y):=\sum_{n=0}^{\infty} p_{n}(t,x;s,y)$ and $G(t,x;s,y):=\sum_{n=0}^{\infty}\nabla_xp_{n}(t,x;s,y)$ are locally uniformly absolutely convergent on $\mD^{\delta_1}_0$.
  In particular, $p, G$ are continuous on $\mD^{\delta_1}_0$ and
  $$
    \nabla_x p(t,x;s,y)=G(t,x;s,y).
  $$
  On the other hand, due to \eqref{ET1} we have
  $$
    \sum_{n=0}^{m} p_{n}(t,x;s,y)= p_0(t,x;s,y)+\int^s_t\!\!\!\int_{\mR^d} \sum_{n=0}^{m-1}p_n(t,x;r,z)\sL^{b,\kappa}_r Z(r,\cdot;s,y)(z)\dif z\dif r.
  $$
  By taking limits and the dominated convergence theorem, we obtain (\ref{eqdu}). Moreover, by \eqref{ep2}, we have for $j=0,1$,
  \begin{align}
  \begin{split}
    &|\nabla^j_xp(t,x;s,y)-\nabla^j_xp_0(t,x;s,y)|\leq \sum_{n=1}^{\infty}|\nabla^j_xp_{n}(t,x;s,y)|\\
    &\qquad\leq 2c (\|\kappa\|_\infty\eta_{\alpha,2-j} +\ell_{b,\kappa}\xi_{\lambda,-j})(s-t,y-x),
    \end{split}\label{ET5}
  \end{align}
  which in turn implies \eqref{RT1} and \eqref{RT2}.

  Now let $\tilde p(t,x;s,y)$ be another solution to (\ref{eqdu}) satisfying (\ref{RT1}). As in the proof of \eqref{ep2},
  we can show that for all $n\in\mN$,
  \begin{align*}
   |p(t,x;s,y)-\tilde p(t,x;s,y)|\leq&~ C_1\left( c(c\ell_{b,\kappa}(s-t))^{n-1} + (c\ell_{b,\kappa}(s-t))^{n}\right)\|\kappa\|_\infty \eta_{\alpha,2}(s-t,y-x)\\
   &~+ C_1 (c\ell_{b,\kappa}(s-t))^{n}\xi_{\lambda,0}(s-t,y-x).
  \end{align*}
  Since $c\ell_b(s-t)\leq{1 /2}$, letting $n\rightarrow \infty$, we obtain the uniqueness.
  \medskip\\
  (ii) By \eqref{ET5}, if $|x-y|\leq\sqrt{s-t}$, then we have
  \begin{align*}
    p(t,x;s,y)&\geq p_0(t,x;s,y)-2c (\|\kappa\|_\infty\eta_{\alpha,2} + \ell_{b,\kappa}\xi_{\lambda,0})(s-t,y-x)\\
    &\geq (c_1-c_2 \ell_{b,\kappa}(s-t))(s-t)^{-d/2}\geq c_1(s-t)^{-d/2}/2,
  \end{align*}
  provided $s-t\leq\delta_2$ with $\delta_2$ being small enough so that $\ell_{b,\kappa}(\delta_2)\leq \frac{c_1}{2c_2}$.
    \medskip\\
  (iii) By Fubini's theorem, we have
  \begin{align*}
    \int_{\mR^d}p(t,x;r,z)p(r,z;s,y)\dif z=\sum_{n=0}^\infty\sum_{m=0}^n\int_{\mR^d}p_m(t,x;r,z)p_{n-m}(r,z;s,y)\dif z.
  \end{align*}
  For proving \eqref{CK}, it suffices to prove that for each $n\in\mN_0$,
  \begin{align}\label{UT2}
    \sum_{m=0}^n\int_{\mR^d}p_m(t,x;r,z)p_{n-m}(r,z;s,y)\dif z=p_n(t,x;s,y).
  \end{align}
  For $n=0$, it is clearly true by \eqref{CK1}. Now suppose \eqref{UT2} holds for some $n\in\mN$. Write
  $$
    \sum_{m=0}^{n+1}\int_{\mR^d}p_m(t,x;r,z)p_{n+1-m}(r,z;s,y)\dif z=I_1+I_2,
  $$
  where
  \begin{align*}
    I_1&:=\int_{\mR^d}p_{n+1}(t,x;r,z)p_0(r,z;s,y)\dif z,\\
    I_2&:=\sum_{m=0}^{n}\int_{\mR^d}p_m(t,x;r,z)p_{n+1-m}(r,z;s,y)\dif z.
  \end{align*}
  Observing that
  \begin{align}
    \int_{\mR^d}\sL^{b,\kappa}_tp_0(t,\cdot;r,z)(x)p_0(r,z;s,y)\dif z=\sL^{b,\kappa}_t p_0(t,\cdot;s,y)(x),
  \end{align}
  by \eqref{ET1} and Fubini's theorem, we have
  \begin{align*}
    I_1&=\int_{\mR^d}\left(\int^r_t\!\!\int_{\mR^d}p_{n}(t,x;r',z')\sL^{b,\kappa}_{r'}p_0(r',\cdot; r,z)(z')\dif z'\dif r'\right) p_0(r,z;s,y)\dif z\\ &=\int^r_t\!\!\int_{\mR^d}p_{n}(t,x;r',z')\left(\int_{\mR^d}\sL^{b,\kappa}_{r'}p_0(r',\cdot; r,z)(z') p_0(r,z;s,y)\dif z\right)\dif z'\dif r'\\
    &=\int^r_t\!\!\int_{\mR^d}p_{n}(t,x;r',z')\sL^{b,\kappa}_{r'}p_0(r',\cdot; s,y)(z')\dif z'\dif r'.
  \end{align*}
  Similarly, by \eqref{ET1} and the induction hypothesis, we have
  $$
    I_2=\int^s_r\!\!\int_{\mR^d}p_n(t,x;r',z')\sL^{b,\kappa}_{r'}p_0(r',\cdot;s,y)(z')\dif z'\dif r'.
  $$
  Hence,
  $$
    I_1+I_2=\int^s_t\!\!\int_{\mR^d}p_{n}(t,x;r',z')\sL^{b,\kappa}_{r'}p_0(r',\cdot; s,y)(z')\dif z'\dif r'=p_{n+1}(t,x;s,y),
  $$
  which gives \eqref{UT2}.
    \medskip\\
  (iv) If in addition for $\alpha\in(0,1]$, $b\in \mK_1$ and for $\alpha\in(1,2)$, $b\in\bar\mK_\alpha$, then by \eqref{UT01},
  and since $\lim_{\delta\to}\ell_{b,\kappa}(\delta)=0$, where $\ell_{b,\kappa}(r)$ is the same as in Lemma \ref{Th2} with $T = 1$, as above there is a $\delta_3>0$ such that the series $\sum_{n=0}^\infty|\widetilde\sL^{\kappa}p_n(t,\cdot;s,y)(x)|$ is locally uniformly convergent on $\mD^{\delta_3}_0$. In particular, we have on $\mD^{\delta_1\wedge\delta_3}_0$,
  $$
    \widetilde\sL^{\kappa}p(t,\cdot;s,y)(x)=\sum_{n=0}^\infty\widetilde\sL^{\kappa}p_n(t,\cdot;s,y)(x),
  $$
  and so \eqref{RT3} holds. Moreover, by \eqref{equ}, we also have \eqref{eq}.
    \medskip\\
  (v) Let $C_5$ be the constant $C_1$ in \eqref{UT1} with $T=1$. As above, it follows from \eqref{UT1} with $T=1$ that there is a $\delta_4>0$ such that $C_5\delta^{\frac{(2-\alpha)\wedge 1}{2}}_4={1 /2}$.
    \medskip\\
  Finally, we just need to set $\delta:=\delta_1\wedge\delta_2\wedge\delta_3\wedge\delta_4$.
\end{proof}

\medskip

Using \eqref{CK}, we can extend the definition of $p(t, x; s, y)$ to $\mD^\infty_0$.

\medskip

\begin{proof}[Proof of Theorem \ref{main1}]
  We shall show that $p(t,x;s,y)$ in Theorem \ref{Th1} has all the properties in Theorem \ref{main1}. First of all,
  we need to extend the definition of $p(t,x;s,y)$ from $\mD^{\delta}_0$ to $\mD^\infty_0$ by C-K equation as follows: If $\delta<s-t\leq 2\delta$, we define
  \begin{align}\label{Ext}
    p(t,x;s,y)=\int_{\mR^d}p\big(t,x;\tfrac{t+s}{2},z\big)p\big(\tfrac{t+s}{2},z;s,y\big)\dif z.
  \end{align}
  Proceeding this procedure, we can extend $p$ to $\mD^\infty_0$.
  \\
  \vbox{~}
  (1), (2), (3) and (4) follow from \eqref{RT1}, \eqref{RT2},  \eqref{CK}, \eqref{RT3}, \eqref{Ext} and Lemma \ref{Le21}.
  As for \eqref{eqdu} and \eqref{eqdu0} on $\mD^\infty_0$, it follows by \eqref{eqdu} and \eqref{eqdu0} on $\mD^{\delta}_0$ and C-K equation.
  \\
  \vbox{~}
  (5) By the construction of $p(t,x;s,y)$ (see the proof of Theorem \ref{Th1}(i)), there is a $\delta_1 > 0$ such that $p(t,x;s,y) = \sum_{n=0}^\infty p_n(t,x;s,y)$
  on $\mathbb{D}^{\delta_1}_0$. Then, by the dominated convergence theorem and Fubini's theorem, for any $s,t\geq 0$ with $0< s-t < \delta_1$ and $x \in \mathbb{R}^d$,
  \begin{align*}
    &\int_{\mathbb{R}^d} p(t,x;s,y) \dif y = \sum_{n=0}^\infty \int_{\mathbb{R}^d} p_n(t,x;s,y) \dif y= \int_{\mathbb{R}^d} Z(t,x;s,y) \dif y \\
    &+ \sum_{n=1}^\infty \int_t^s\!\!\! \int_{\mathbb{R}^d} p_{n-1}(t,x;r,z) \left(\int_{\mathbb{R}^d} \mathscr{L}^{b,\kappa}_r Z(r,\cdot;s,y)(z) \dif y\right) \dif z \dif r
    = 1 + 0 = 1.
  \end{align*}
  Hence, the conservativeness (\ref{eq:conser}) follows from the above equality and (\ref{Ext}).
\medskip \\
  (6) Let $P_{t,s}f(x):=\int_{\mR^d}p(t,x;s,y)f(y)\dif y.$ By (\ref{eqdu}), we have for any bounded measurable $f$,
  \begin{align}\label{RT5}
    P_{t,s}f(x)=P^{(Z)}_{t,s}f(x)+\int^s_t\! P_{t,r}\sL^{b,\kappa}_rP^{(Z)}_{r,s}f(x)\dif r.
  \end{align}
  Hence, by (\ref{ET01}),  for $f\in C^2_b(\mR^d)$, we have
  \begin{align}
    P_{t,s}f(x)-f(x)&=P^{(Z)}_{t,s}f(x)-f(x)+\int^s_t\! P_{t,r}\sL^{b,\kappa}_rP^{(Z)}_{r,s}f(x)\dif r\no\\
    &=\int^s_t\!P^{(Z)}_{t,r}\sL^{a}_rf(x)\dif r+\int^s_t\!P_{t,r}\sL^{b,\kappa}_rP^{(Z)}_{r,s}f(x)\dif r,\label{RT4}
  \end{align}
  and, by \eqref{RT5} and Fubini's theorem,
  \begin{align*}
    \int^s_t\! P_{t,r}\sL^{a}_rf(x)\dif r&-\int^s_t\!P^{(Z)}_{t,r}\sL^{a}_rf(x)\dif r =\int^s_t\!\!\!\int^r_t\! P_{t,u}\sL^{b,\kappa}_{u}P^{(Z)}_{u,r}\sL^{a}_rf(x)\dif u\dif r\\
    &=\int^s_t\!P^{}_{t,u}\sL^{b,\kappa}_{u}\left(\int^s_u\! P^{(Z)}_{u,r}\sL^{a}_rf(x)\dif r\right)\dif u\\
    &=\int^s_t\!P_{t,u}\sL^{b,\kappa}_{u}\Big(P^{(Z)}_{u,s}f(x)-f(x)\Big)\dif u.
  \end{align*}
  Combining this with \eqref{RT4}, we obtain
  \begin{align*}
    P_{t,s}f(x)-f(x)&=\int^s_t\!P_{t,r}\Big(\sL^{a}_r+\sL^{b,\kappa}_r\Big) f(x)\dif r=\int^s_t\!P_{t,r}\sL_r f(x)\dif r .
  \end{align*}
  (7) By \eqref{eqdu} and \eqref{cz2}, we only need to show that
  $$
    \lim_{|t-s|\to 0}\sup_{x\in\mR^d}\left|\int^s_t\! P_{t,r}\sL^{b,\kappa}_rP^{(Z)}_{r,s}f(x)\dif r\right|=0.
  $$
  Notice that by \eqref{RT1}, \eqref{ep4} and Lemma \ref{lem:3p},
  \begin{align*}
    &\left|\int^s_t\! P_{t,r}\sL^{b,\kappa}_rP^{(Z)}_{r,s}f(x)\dif r\right| \preceq\|f\|_\infty\int^s_t\!\!\! \int_{\mR^d}
    \!\int_{\mR^d}(\xi_{\lambda,0} + \|\kappa\|_\infty \eta_{\alpha,2})(r-t,z-x)\\
    &\qquad\qquad\times\Big(|\tilde{b}(r,z)|\cdot \xi_{\lambda,-1}(s-r,y-z)+\|\kappa\|_\infty \eta_{\alpha,0}(s-r,y-z)\Big)\dif z\dif y\dif r\\
    &\qquad\preceq\|f\|_\infty\Bigg(\int_{\mR^d}\Big(\left(\|\kappa\|_\infty(s-t)^{1-\frac{\alpha}{2}} + K_2^{\widetilde{b}}(s-t)\right) \left(\xi_{\lambda,0}+ \|\kappa\|_\infty \eta_{\alpha,2}\right)(s-t,y-x)\\
    &\qquad\qquad\quad+ \|\kappa\|_\infty\left(\eta_{\alpha,2}+\|\kappa\|_\infty \eta_{\alpha,4-\alpha}\right)(s-t,y-x)\Big)\dif y\Bigg)\\
    &\qquad\leq C\|f\|_\infty\left(\ell_{b,\kappa}(s-t)+ (\ell_{b,\kappa}(s-t))^2\right),
  \end{align*}
  where $\ell_{b,\kappa}(r)$ is the same as in Lemma \ref{Le32}, and the constant $C$ is independent of $x$ and $s,t$. Since $b\in\mK_2$,
  one derives the desired limit.
\end{proof}

\section{Proof of the lower bound}

\subsection{Positivity}\label{S:4.1}

In this subsection,
we show that if
 $\kappa \geq 0$,
then the continuous kernel $p(t,x;s,y)$ constructed in Theorem \ref{Th1} is non-negative.

\bt\label{Th41}
  Under {\bf (H$^a$)}, {\bf (H$^\kappa$)} and $b\in\mK_2$, if $\kappa\geq 0$, then the heat kernel $p(t,x;s,y)$ constructed in Theorem \ref{Th1} is non-negative.
  \et

\begin{proof} We divide the proof into three steps.
\medskip\\
 (i) Let $b_\eps:=b*\rho_\eps^{(1)}$ and $\kappa_\eps:=\kappa*\rho_\eps^{(2)}$, where $\rho_\eps^{(1)}\in C^\infty_c(\mR^{d+1})$ is supported in $B_\eps\subset \mathbb{R}^{d+1}$ satisfying $\int\rho_\eps^{(1)}=1$ and $\rho_\eps^{(2)}\in C^\infty_c(\mR^{2d+1})$ is supported in $B_\eps\subset \mathbb{R}^{2d+1}$ satisfying
  $\int\rho_\eps^{(2)}=1$. For example,
  $$
  \kappa_\eps(t,x,z):=\int_{\mR^{2d+1}}\kappa(s,y_1,y_2)\rho_\eps^{(2)}(t-s,x-y_1,z-y_2)\dif y_1 \dif y_2 \dif s.
  $$
  Let $p_\eps(t,x;s,y)$ be the corresponding heat kernel constructed in Theorem \ref{Th1}. We claim that
  \begin{align}\label{RT25}
    p_\eps(t,x;s,y)\geq 0   \quad \text{ on } \mathbb{D}_0^\infty.
  \end{align}
  While it is possible to use Hille-Yosida-Ray theorem and Courr\'ege's first theorem to prove the claim,
  as it was done in \cite[Theorem 1.2]{CW}
  (see also \cite[Lemma 4.1]{Wang.2015.MZ521} and \cite[Lemma 4.9]{ChenHu.2015.SPA2603}),
  we present here a self-contained proof based on the maximum principle established in Theorem \ref{Le51} in the Appendix.
  Notice that for any $\theta > 0$ and $\varepsilon \in (0,1)$,
  \begin{align}\label{RT24}
    K^{b_\eps}_\theta(\delta)\leq K^{b}_\theta(\delta), \quad
    \|\kappa_\eps\|_\infty\leq\|\kappa\|_\infty,
  \end{align}
  and by \eqref{eq:KatoLocalNorm}, each pair of $b_\eps$ and $\kappa_\eps$ satisfies \eqref{RT9}. Hence, by \eqref{RT1} and \eqref{RT2}, we have the following uniform estimate:
  \begin{align}
    \sup_{\eps\in(0,1)}|\nabla^j_x p_\eps(t,x;s,y)|\leq C(\xi_{\lambda,-j}+\|\kappa\|_\infty\eta_{\alpha,2-j})(s-t,y-x),
    \quad   j=0,1.  \label{RT22}
  \end{align}
  Let $f\in C^2_b(\mR^d)$ be non-negative. Fix $s>0$ and set
  $$
    u_\eps(t,x):=\int_{\mR^d}p_\eps(t,x;s,y)f(y)\dif y,\quad t<s.
  $$
  In order to show \eqref{RT25}, it suffices to check that the conditions of Theorem \ref{Le51} are satisfied for $u_\eps$. Since $b_\eps, \kappa_\eps\in C^\infty_b(\mR_+\times\mR^d)$, by Theorem \ref{Th1}, one sees that $u_\eps\in C_b([0,s]\times\mR^d)$ and $(t,x)\mapsto\nabla^j u_\eps(t,x)$ is continuous for $j=0,1,2$, and
  $$
    u_\eps(t,x)=P^{(Z)}_{t,s}f(x)+\int^s_tP^{(Z)}_{t,r}\sL^{b_\eps,\kappa_\eps}_r u_\eps(r,x)\dif r,
  $$
  where $P^{(Z)}_{t,s}f(x):=\int_{\mR^d}Z(t,x;s,y)f(y)\dif y$. Moreover, by \eqref{RT1}, \eqref{RT2}, \eqref{SEC} and Lemma \ref{Le22}, as in the proof of \eqref{ET32}, we have for any $\gamma\in(0,(2-\alpha)\wedge1)$,
  $$
    |\sL^{b_\eps,\kappa_\eps}_r u_\eps(r,x)-\sL^{b_\eps,\kappa_\eps}_r u_\eps(r,x')|\leq C_\eps(s-r)^{-\frac{\alpha+\gamma}{2}}|x-x'|^\gamma.
  $$
  Hence, for Lebesgue almost all $t\in[0,s]$ (see Lemma \ref{Le54} below),
  $$
    \p_tu_\eps+\sL^a_t u_\eps+\sL^{b_\eps,\kappa_\eps}_t u_\eps=0.
  $$
  Thus  we can use Theorem \ref{Le51} to conclude \eqref{RT25}.
\medskip\\
   (ii) In this step, we  show that
  \begin{align}\label{Equi}
    \mbox{$p_\eps$ is locally equi-continuous on $\mD^\infty_0$}.
  \end{align}
  Recall that
  \begin{align}\label{Lim1}
    p_\eps(t,x;s,y)=Z(t,x;s,y)+\int^s_t\!\!\!\int_{\mR^d}p_\eps(t,x;r,z)\sL^{b_\eps, \kappa_\eps}_r Z(r,\cdot;s,y)(z)\dif z\dif r.
  \end{align}
  Let $D$ be a compact subset of $\mD^{T_0}_{t_0}\subset\mD^\infty_0$,
  where $0<t_0<T_0<\infty$. We first show that
  \begin{align}\label{Lim4}
    \lim_{|h|\to 0}\sup_{\eps\in(0,1)}\sup_{(t,x,s,y)\in D} |p_\eps(t,x;s,y)-p_\eps(t,x;s,y+h)|=0.
  \end{align}
  Notice that for any $\delta < (s-t)/4$, by \eqref{RT22}, \eqref{ep4},  (\ref{eqcom}) and (\ref{ineq}),
  \begin{align}\label{eq:lim6}
  \begin{split}
    &\quad\left|\int^s_{s-\delta}\!\int_{\mR^d}p_\eps(t,x;r,z)\sL^{b_\eps, \kappa_\eps}_r Z(r,\cdot;s,y)(z)\dif z\dif r\right|\\
    &\preceq \int^s_{s-\delta}\!\int_{\mR^d}(\xi_{\lambda,0}+\|\kappa\|_\infty \eta_{\alpha,2})(r-t,z-x)\\
    &\quad\times\Big((|b_\eps(r,z)|+\|\kappa\|_\infty)\cdot \xi_{\lambda,-1}(s-r,y-z) +\|\kappa\|_\infty\eta_{\alpha,0}(s-r,y-z)\Big)\dif z\dif r\\
    &\preceq \left((s-t-\delta)^{-d/2}+(s-t-\delta)^{1-(d+\alpha)/2}\right)\\
    &\quad\times \int_0^\delta\!\!\!\int_{\mR^d}\Big(|b_\eps(s-r,y-z)|\cdot \eta_{\alpha,\alpha-1}(r,z) +\|\kappa\|_\infty\eta_{\alpha,0}(r,z)\Big)\dif z\dif r\\
    &\preceq (s-t)^{-d/2}\left(K_2^b(\delta)+\|\kappa\|_\infty \delta^{1/2} + \|\kappa\|_\infty \delta^{1-\alpha/2}\right).
    \end{split}
  \end{align}
  By \eqref{RT22}, \eqref{eq202}, \eqref{ET33}, Lemmas \ref{lem:3p} and \ref{Le21}, for any $\beta'\in(0,\beta)$ and $\gamma\in(0,(2-\alpha)\wedge 1)$, we have
  \begin{align*}
    &\left|\int^{s-\delta}_{t}\!\!\!\int_{\mR^d}p_\eps(t,x;r,z)\sL^{b_\eps, \kappa_\eps}_r (Z(r,\cdot;s,y)-Z(r,\cdot;s,y+h))(z)\dif z\dif r\right|\\
    &\preceq |h|^{\beta'}\Bigg[\int^{s-\delta}_{t}\!\!\!\int_{\mR^d} (\xi_{\lambda,0}
     + \|\kappa\|_\infty \eta_{\alpha,2})(r-t,z-x)\left(|b_\eps(r,z)|+\|\kappa_\eps\|_\infty\right)\\
    &\qquad\times\Big(\xi_{\lambda,-\beta'-1}(s-r,y-z)+\xi_{\lambda,-\beta'-1}(s-r,y+h-z)\Big)\dif z\dif r\Bigg]\\
    &\quad+|h|^{\beta'\gamma} \|\kappa\|_\infty\Bigg[\int^s_t\!\!\!\int_{\mR^d} (\xi_{\lambda,0}+ \|\kappa\|_\infty \eta_{\alpha,2})(r-t,z-x)\\
    &\qquad\times\Big(\eta_{\alpha,-\beta'\gamma}(s-r,y-z)+\eta_{\alpha,-\beta'\gamma}(s-r,y+h-z)\Big)\dif z\dif r\Bigg]\\
    &\preceq |h|^{\beta'}\delta^{-\beta'/2}\Bigg[\int^s_t\!\!\!\int_{\mR^d} (\xi_{\lambda,0} + \|\kappa\|_\infty \eta_{\alpha,2})(r-t,z-x)\left(|b_\eps(r,z)|+\|\kappa_\eps\|_\infty\right)\\
    &\quad\times\Big(\xi_{\lambda,-1}(s-r,y-z)+\xi_{\lambda,-1}(s-r,y+h-z)\Big)\dif z\dif r\Bigg]\\
    &\quad+|h|^{\beta'\gamma}\left((s-t)^{2-\beta'\gamma-(d+\alpha)} + \|\kappa\|_\infty(s-t)^{4-\alpha-\beta'\gamma-(d+\alpha)}\right)\\
    &\preceq|h|^{\beta'}\delta^{-\beta'/2} K^b_2(s-t)\left(1 + \|\kappa\|_\infty\right)(s-t)^{-d/2}\\
    &\quad+|h|^{\beta'\gamma} (s-t)^{2-\beta'\gamma-(d+\alpha)}\left(1+ \|\kappa\|_\infty(s-t)^{2-\alpha}\right),
  \end{align*}
  which together with \eqref{eq:lim6} gives \eqref{Lim4}. Similarly, we can show
  \begin{align*}
    \lim_{|h|\to 0}\sup_{\eps\in(0,1)}\sup_{(t,x;s,y)\in D} |p_\eps(t,x;s,y)-p_\eps(t,x+h;s,y)|=0,\\
    \lim_{|\delta|\to 0}\sup_{\eps\in(0,1)}\sup_{(t,x;s,y)\in D} |p_\eps(t,x;s,y)-p_\eps(t+\delta,x;s,y)|=0,\\
    \lim_{|\delta|\to 0}\sup_{\eps\in(0,1)}\sup_{(t,x;s,y)\in D} |p_\eps(t,x;s,y)-p_\eps(t,x;s+\delta,y)|=0.
  \end{align*}
  Thus  we obtain \eqref{Equi}.
  \medskip\\

  (iii) By \eqref{Equi},  Ascoli-Arzela's lemma and a diagonalization argument,
  there exist a subsequence $\eps_k$ (still denoted by $\eps$ for simplicity) and a continuous function $\bar p$ such that
  \begin{align}\label{RT23}
    p_{\eps}(t,x;s,y)\to \bar p(t,x;s,y) \mbox{ for all $(t,x;s,y)\in\mD^\infty_0$.}
  \end{align}
  Now we want to take limits  on
  both sides of \eqref{Lim1}. First, by \eqref{RT24}, \eqref{RT22}, \eqref{RT23} and the dominated convergence theorem, we have
  $$
    \int^s_t\!\!\!\int_{\mR^d}p_\eps(t,x;r,z)\widetilde{\sL^{\kappa_\eps}_r} Z(r,\cdot;s,y)(z)\dif z\dif r\stackrel{\eps\to 0}{\to}
    \int^s_t\!\!\!\int_{\mR^d}\bar p(t,x;r,z)\widetilde{\sL^{\kappa}_r} Z(r,\cdot;s,y)(z)\dif z\dif r.
  $$
  Next, for the term containing $\widetilde{b_\eps}(r,\cdot)\cdot\nabla$, by \eqref{RT22} and \eqref{eq21}, we have
  \begin{align*}
    &\left|\int^s_t\!\!\!\int_{\mR^d}p_\eps(t,x;r,z)\widetilde{b_\eps}(r,z)\nabla_z Z(r,z;s,y)\dif z\dif r
    -\int^s_t\!\!\!\int_{\mR^d}\bar p(t,x;r,z)\tilde{b}(r,z)\nabla_zZ(r,z;s,y)\dif z\dif r\right|\\
    &\quad\preceq \int^s_t\!\!\!\int_{\mR^d} (\xi_{\lambda,0}+\|\kappa_\varepsilon\|_\infty\eta_{\alpha,2})(r-t,z-x) \cdot |\widetilde{b_\eps}-\tilde{b}|(r,z)\cdot\xi_{\lambda_1,-1}(s-r,y-z)\dif z\dif r\\
    &\quad+\int^s_t\!\!\!\int_{\mR^d}|p_\eps-\bar p|(t,x;r,z)\cdot|\tilde{b}(r,z)|\cdot \xi_{\lambda_1,-1}(s-r,y-z)\dif z\dif r=:I_1(\eps)+I_2(\eps),
  \end{align*}
  where $\tilde{b}$ is defined in \eqref{eq:tildebDef} and $\widetilde{b_\eps}$ is defined by
  \begin{equation*}
    \widetilde{b_\eps}(t,x):=b_\eps(t,x)+1_{\alpha\in(1,2)}\int_{|z|>1}z\kappa_\eps(t,x,z)|z|^{-d-\alpha}\dif z -1_{\alpha\in(0,1)} \int_{|z|\leq 1}z\kappa_\eps(t,x,z)|z|^{-d-\alpha}\dif z.
  \end{equation*}
  For $I_1(\eps)$, by \eqref{eqcom}, \eqref{eq3p}, Lemma \ref{Le28} and the dominated convergence theorem, we have
  \begin{align*}
    I_1(\eps)&\preceq \int^s_t\!\!\!\int_{\mR^d} (\eta_{2,2}+\|\kappa_\varepsilon\|_\infty\eta_{\alpha,2})(r-t,z-x) \cdot |\widetilde{b_\eps}-\tilde{b}|(r,z)\cdot\eta_{2,1}(s-r,y-z)\dif z\dif r\\
    &\preceq \left(\eta_{2,0}+\|\kappa\|_\infty\eta_{\alpha,0}\right)(s-t,y-x) \int^s_t\!\!\!\int_{\mR^d}|\widetilde{b_\eps}-\tilde{b}|(r,z)\cdot(r-t)(s-r)^{1/2}\\
    &\quad\times\Big(\eta_{2,0}(r-t,z-x)+\eta_{2,0}(s-r,y-z)\Big)\dif z\dif r\\
    &\preceq\left(\eta_{2,1}+\|\kappa\|_\infty\eta_{\alpha,1}\right)(s-t,y-x) \Bigg[\int^s_t\!\!\!\int_{\mR^d}|\widetilde{b_\eps}-\tilde{b}|(r,z)\cdot\eta_{2,1}(r-t,z-x)\dif z\dif r\\
    &\quad+\int^s_t\!\!\!\int_{\mR^d}|\widetilde{b_\eps}-\tilde{b}|(r,z)\cdot\eta_{2,1}(s-r,y-z)\dif z\dif r\Bigg]\stackrel{\eps\to 0}{\to} 0.
  \end{align*}
  For $I_2(\eps)$, by \eqref{RT22}, \eqref{RT23} and the dominated convergence theorem again, we have
  $$
    I_2(\eps)\stackrel{\eps\to 0}{\to} 0.
  $$
  Combining the above limits, one sees that $\bar p(t,x;s,y)$ satisfies \eqref{eqdu}. Similarly, we can show $\bar p$ also satisfies
  \eqref{eqlpe}.
  Thus by the uniqueness, we obtain $\bar p=p$ and so, $p\geq 0$.
\end{proof}

\subsection{Lower bound estimate}\label{S:4.2}

Throughout this subsection, we assume $\kappa \geq 0$.
By Theorem \ref{Th41} and \eqref{e:1.13}, $\{p(t,x;s,y): (t,x;s,y)\in\mD^\infty_0\}$ is
a family of transition probability density functions.
It determines a Feller process
$$
  \Big(\Omega,\sF, (\mP_{t,x})_{(t,x)\in\mR_+\times\mR^d}; (X_s)_{s\geq 0}\Big),
$$
with the property that
$$
  \mP_{t,x}\big(X_s=x,\,0\leq s\leq t\big)=1,
$$
and for $r\in[t,s]$ and $E\in\cB(\mR^d)$,
\begin{align}\label{Tr}
  \mE_{t,x}\big(X_{s}\in E \,|\, X_{r}\big)=\int_{E}p(r,X_{r};s,y)\dif y.
\end{align}
Moreover, for any $f\in C^2_b(\mR^d)$, it follows from (\ref{eqge}) and the Markov property of $X$ that under $\mP_{t,x}$, with respect to the filtration $\sF_s:=\sigma\{X_r, r\leq s\}$,
\begin{align}
  M^f_s:=f(X_s)-f(X_t)-\int^s_t\sL_r f(X_r)\dif r\ \mbox{ is a martingale}.\label{ERY1}
\end{align}
In other words, $\mP_{t,x}$ solves the martingale problem for $(\sL_t, C^2_b (\mR^d))$.

For any Borel set $E$, let
$$
  \sigma_E:=\inf\{s\geq 0:X_s\in E\},\qquad \tau_E:=\inf\{s\geq 0:X_s\notin E\},
$$
be the first hitting and exit time, respectively, of $E$.

 Below for simplicity, we write
$$
  \cJ_\alpha(t,x,z):=\kappa(t,x,z-x)|z-x|^{-d-\alpha}.
$$
We now determine the L\'evy system of
the Feller process $X$, which in particular is a Hunt process.
The proof of the following result is similar to that of \cite{ChenKimSong.2012.AP2483}. For completeness, we give a detailed proof here.
\bl\label{Lemar}
  Suppose that $E$ and $F$ are two disjoint open sets in $\mR^d$. Then
  $$
    \sum_{t<r\leq s}1_{\{X_{r-}\in E, X_{r}\in F\}}-\int^s_t\!1_{E}(X_r)\!\!\int_{F}\cJ_\alpha(r,X_r,z)\dif z\dif r
  $$
  is a $\mP_{t,x}$-martingale for every $t\geq 0$ and $x\in\mR^d$.
\el

\begin{proof}
First of all, by (\ref{ERY1}), $\{X_s, s\geq 0\}$ is a semi-martingale under $\mP_{t,x}$. Let $f\in C^2_b(\mR^d)$ with $f=0$ on $E$ and $f=1$ on $F$. By It\^o's formula, we have
\begin{align*}
  f(X_s)-f(X_t)&=\sum_{i=1}^d\int^s_t\!\p_if(X_{r-})\dif X_r+\sum_{t<r\leq s}\beta_r(f)+\frac{1}{2}\sum_{i,j=1}^d\int^s_t\!\p_i\p_jf(X_{r-})\dif\<X^c,X^c\>_r,
\end{align*}
where
$$
  \beta_r(f):=f(X_r)-f(X_{r-})-\sum_{i=1}^d\p_if(X_{r-})(X_r-X_{r-}).
$$
Hence,
\begin{align*}
  N_s:=&\int^s_t\! 1_{E}(X_{r-})\dif M^f_r =\sum_{i=1}^d\int^s_t\! 1_{E}(X_{r-})\p_if(X_{r-})\dif X_r+\sum_{t<r\leq s}1_{E}(X_{r-})\beta_r(f)\\
  &+\frac{1}{2}\sum_{i,j=1}^d\int^s_t\!1_{E}(X_{r-})\p^2_{ij}f(X_{r-})\dif \<X^c,X^c\>_r-\int^s_t\! 1_{E}(X_r)\sL_rf(X_r)\dif r
\end{align*}
is a martingale. Since $f(x)=\p_if(x)=\p^2_{ij}f(x)=0$ for $x\in E$, we further have
\begin{align*}
  N_s&=\sum_{t<r\leq s}1_{E}(X_{r-})f(X_r)-\int^s_t\! 1_E(X_r)\sL^{\kappa}_rf(X_r)\dif r\\
  &=\sum_{t<r\leq s}1_{E}(X_{r-})f(X_r)-\int^s_t\! 1_E(X_r)\int_{\mR^d}f(z)\cJ_\alpha(r,X_r,z)\dif z\dif r.
\end{align*}
By choosing $f_n\in C^2_b(\mR^d)$ with $f_n|_E=0,~ f_n|_F=1$ and $f_n\to 1_F$, then taking limits, we obtain the desired result.
\end{proof}

In particular, Lemma \ref{Lemar} implies that
$$
  \mE_{t,x}\left[\sum_{t<r\leq s}1_{E}(X_{r-})1_{F}(X_r)\right]=\mE_{t,x}\left[\int^s_t\!\!\!\int_{\mR^d}1_{E}(X_r)1_{F}(z)\cJ_\alpha(r,X_r,z)\dif z\dif r\right].
$$
Let $f$ be a non-negative measurable function on $\mR_+\times\mR^d\times\mR^d$ that vanishes along the diagonal. By a routine measure theoretic argument, we get
$$
  \mE_{t,x}\left[\sum_{t<r\leq s}f(r,X_{r-},X_r)\right]=\mE_{t,x}\left[\int^s_t\!\!\!\int_{\mR^d}f(r,X_r,z)\cJ_\alpha(r,X_r,z)\dif z\dif r\right].
$$
Finally, we can follow the same method as in \cite{ChenKumagai.2010.RMI551} to get the following  L\'evy system.

\bl\label{lem:XLS}
  Let $f$ be a non-negative measurable function on $\mR_+\times\mR^d\times\mR^d$ that vanishes along the diagonal. Then for every stopping time $T$ (with respect to the filtration of $X$), we have
  \begin{align}
    \mE_{t,x}\left[\sum_{t<r\leq T}f(r,X_{r-},X_r)\right]=\mE_{t,x}\left[\int^T_t\!\!\!\int_{\mR^d}f(r,X_r,z)\cJ_\alpha(r,X_r,z)\dif z\dif r\right].\label{sys}
  \end{align}
\el

We need the following two lemmas.
\bl\label{Le34}
  For any $M>0$, there is a constant $\gamma_0\in(0,1)$ depending only on $M$ and the constants in \eqref{eqlpe} with $T = 1$ such that for all $\delta\in(0,M)$,
  \begin{align}
    \sup_{(t,x)\in\mR_+\times\mR^d}\mP_{t,x}\Big(\tau_{B(x,\delta)}\leq t+\gamma_0\delta^2\Big)\leq \tfrac{1}{2}.  \label{eqtaue}
  \end{align}
\el

\begin{proof}
  For simplicity, write $\tau:=\tau_{B(x,\delta)}$. By the strong Markov property of $X$, we have
  \begin{align}
    \mP_{t,x}\Big(\tau\leq t+r\Big)&\leq\mP_{t,x}\Big(\tau\leq t+r; X_{t+r}\in B(x,\tfrac{\delta}{2})\Big) +\mP_{t,x}\Big(X_{t+r}\notin B(x,\tfrac{\delta}{2})\Big)\no\\
    &=\mP_{t,x}\left(\mP_{\tau,X_\tau}\Big(X_{t+r}\in B(x,\tfrac{\delta}{2})\Big); \tau\leq t+r\right)+\mP_{t,x}\Big(X_{t+r}\notin B(x,\tfrac{\delta}{2})\Big)\no\\ &\leq\mP_{t,x}\left(\mP_{\tau,X_\tau}\Big(|X_{t+r}-X_\tau|\geq\tfrac{\delta}{2}\Big); \tau\leq t+r\right) +\mP_{t,x}\Big(X_{t+r}\notin B(x,\tfrac{\delta}{2})\Big)\no\\
    &\leq 2\sup_{t\leq s\leq t+r}\sup_{x\in\mR^d}\mP_{s,x}\Big(|X_{t+r}-x|\geq\tfrac{\delta}{2}\Big).\label{AS}
  \end{align}
  On the other hand, by \eqref{Tr} and \eqref{eqlpe}, there is a constant $C>0$ depending only the constants in \eqref{eqlpe} with $T = 1$ such that for all $r\in(0,1)$, $t\leq s\leq t+r$ and $x\in\mR^d$,
  \begin{align*}
    \mP_{s,x}\Big(|X_{t+r}-x|\geq\tfrac{\delta}{2}\Big)
    &=\int_{{|y-x|\geq\frac{\delta}{2}}}p(s,x; t+r, y)\dif y\leq C\int_{{|y-x|\geq\frac{\delta}{2}}}(\eta_{\alpha,2}+\xi_{\lambda,0})(t+r-s,y-x)\dif y\\ &=C\int_{{|y|\geq\frac{\delta}{2}}}(\eta_{\alpha,2}+\xi_{\lambda,0})(t+r-s,y)\dif y\leq C(t+r-s)(\delta^{-\alpha}+\delta^{-2}),
  \end{align*}
  which together with \eqref{AS} yields
  \begin{equation}\label{eq:firstExitUp2}
    \mP_{t,x}\Big(\tau_{B(x,\delta)}\leq t+r\Big)\leq Cr(\delta^{-\alpha}+\delta^{-2}).
  \end{equation}
  By letting $r=\gamma_0\delta^2$ with $\gamma_0$ being small enough, we obtain \eqref{eqtaue}.
\end{proof}

For a number $\theta > 0$, define
\begin{equation*}
  m^{(\theta)}_\kappa = \inf_{(t,x) \in\mR_+\times \mathbb{R}^d} \essinf_{|z|<\theta} \kappa(t,x,z).
\end{equation*}

\bl\label{Lemmahe}\label{lem:firstExitLow}
 Let $M>0$ and $\gamma_0$ be the same as in Lemma \ref{Le34}. For all $\gamma\in(0,\gamma_0]$,
 there exists a constant $c_1>0$ such that for all $\delta\in(0,M)$, $t>0$, $\theta>4\delta$ and $x,y\in\mR^d$ with $\theta/2 \geq |x-y|\geq 2\delta$,
  \begin{align}\label{GF1}
    \mP_{t,x}\Big(\sigma_{B(y,\delta)}<t+\gamma\delta^2\Big)\geq c_1\frac{\delta^{d+2}\cdot m^{(\theta)}_\kappa}{|y-x|^{d+\alpha}}.
  \end{align}
\el

\begin{proof}
  For $\delta\in(0,\theta/4)$ and $\gamma\in(0,\gamma_0]$, by (\ref{eqtaue}) we have
  \begin{align}\label{CC2}
    \mE_{t,x}\Bigg(\int_t^{(t+\gamma\delta^2)\wedge \tau_{B(x,\delta)}}\dif r\Bigg)\geq \gamma\delta^2\mP_{t,x}\Big(\tau_{B(x,\delta)}\geq t+\gamma\delta^2\Big) \geq \gamma\delta^2/2.
  \end{align}
  Noticing that
  $$
    X_r\notin B(y,\delta)\ \ \text{when}\ \ t<r<(t+\gamma\delta^2)\wedge \tau_{B(x,\delta)},
  $$
  we have
  $$
    {\bf 1}_{X_{(t+\gamma\delta^2)\wedge \tau_{B(x,\delta)}}\in B(y,\delta)}=\sum_{t<r\leq (t+\gamma\delta^2)\wedge \tau_{B(x,\delta)}}{\bf 1}_{X_{r}\in B(y,\delta)}.
  $$
  By (\ref{sys}) and the definition of $\cJ_\alpha$, we have
  \begin{align}
    \mP_{t,x}\Big(\sigma_{B(y,\delta)}<t+\gamma\delta^2\Big)&\geq\mP_{t,x}\Big(X_{(t+\gamma\delta^2)\wedge \tau_{B(x,\delta)}}\in B(y,\delta)\Big)\no\\
    &=\mE_{t,x}\int_t^{(t+\gamma\delta^2)\wedge \tau_{B(x,\delta)}}\!\!\int_{B(y,\delta)}\frac{\kappa(t,X_r,z-X_r)}{|z-X_r|^{d+\alpha}}\dif z\dif r.\label{CC3}
  \end{align}
  Since $\theta/2 \geq |y-x|\geq 2\delta$, we have for all $z\in B(y,\delta)$ and $X_r\in B(x,\delta)$,
  $$
    |z-X_r|\leq|y-z|+|y-x|+|X_r-x| < 2|y-x| \leq \theta.
  $$
  Thus  by \eqref{CC3} and \eqref{CC2}, we have
  \begin{align*}
    \mP_{t,x}\Big(\sigma_{B(y,\delta)}<t+\gamma\delta^2\Big)&\geq \frac{\gamma\delta^2}{2}\int_{B(y,\delta)}\frac{m^{(\theta)}_\kappa}{(2|y-x|)^{d+\alpha}}\dif z
    \geq c_2\frac{\delta^{d+2}\cdot m^{(\theta)}_\kappa}{|y-x|^{d+\alpha}}.
  \end{align*}
  The proof is complete.
\end{proof}

Now we can give

\begin{proof}[Proof of lower bound \eqref{Low}]
  First of all, by the on-diagonal estimate \eqref{On} and \eqref{CK}, for any $T>0$, there is a constant $C>0$ such that
  \begin{align}\label{ER2}
    p(t,x;s,y)\geq C(s-t)^{-d/2},\ \ |y-x|\leq\sqrt{s-t}\leq \sqrt{T}.
  \end{align}
  In fact, let $\delta$ be as in Theorem \ref{Th1}. If $s-t\leq 2\delta$ and $|y-x|\leq \sqrt{s-t}$, then by \eqref{Ext}, we have {\setlength{\arraycolsep}{0.2pt}
  \begin{eqnarray*}
    p(t,x;s,y)&=&\int_{\mR^d}p(t,x;\tfrac{t+s}{2},z)p(\tfrac{t+s}{2},z; s,y)\dif z\\
    &\geq&\int_{B(\frac{x+y}{2},\frac{\sqrt{s-t}}{2})}p(t,x;\tfrac{t+s}{2},z)p(\tfrac{t+s}{2},z; s,y)\dif z\\
    &\stackrel{(\ref{On})}{\geq}& C^2_3(s-t)^{-d}\Vol\Big(B(\tfrac{x+y}{2},\tfrac{\sqrt{s-t}}{2})\Big)\succeq (s-t)^{-d/2}.
  \end{eqnarray*}}
  Using the above estimate repeatedly, we obtain (\ref{ER2}).

  Now by \eqref{ER2} and a standard chain argument (see \cite{FabesStroock.1986.ARMA327}), for any $T>0$, there are positive constants $C,\lambda_2>0$ such that
  \begin{equation}\label{eq:hkLowExp}
    p(t,x;s,y)\geq C\xi_{\lambda_2,0}(s-t,y-x) \ \mbox{ on $\mD^T_0$}.
  \end{equation}
  Thus  to prove \eqref{Low}, it remains to show that there is a $C'>0$ such that
  $$
    p(t,x;s,y)\geq C'm_k\eta_{\alpha,2}(s-t,y-x) \ \mbox{ on $\mD^T_0$},
  $$
  where $m_\kappa := \inf_{(t,x)}\essinf_{z\in \mathbb{R}^d} \kappa(t,x,z)$.
  If $|y-x|\leq \sqrt{s-t}$, due to \eqref{eq:hkLowExp}, there is nothing to prove. Below we assume
  $$
    |y-x|>\sqrt{s-t}=:3\delta.
  $$
  Let $\gamma_0\in(0,1)$ be the same as in Lemma \ref{Le34} and $\theta \in (0,\infty]$ be an arbitrary fixed number. By the strong Markov property of $X$ and Lemma \ref{lem:firstExitLow}, we have for any
  $\theta/2 >|y-x|\geq 3\delta$,
  {\setlength{\arraycolsep}{0pt}
  \begin{eqnarray*}
    \mP_{t,x}\Big(X_{t+2\gamma_0\delta^2}\in B\big(y,2\delta\big)\Big)&\geq& \mP_{t,x}\left(\sigma:=\sigma_{B(y,\delta)}<t+\gamma_0\delta^2;
    \sup_{s\in[\sigma,\sigma+\gamma_0\delta^2]}|X_s-X_\sigma|< \delta\right)\no\\
    &=&\mE_{t,x}\left(\mP_{\sigma,X_\sigma}\left(\sup_{s\in[\sigma,\sigma+\gamma_0\delta^2]}|X_s-X_\sigma| < \delta\right);
    \sigma_{B(y,\delta)} < t+\gamma_0\delta^2\right)\no\\
    &\geq& \inf_{r,z}\mP_{r,z}\left(\tau_{B(z,\delta)}> r+\gamma_0\delta^2\right) \mP_{t,x}\left(\sigma_{B(y,\delta)}< t+\gamma_0\delta^2\right)\no\\
    &\stackrel{(\ref{eqtaue})}{\geq}&\frac{1}{2} \mP_{t,x}\left(\sigma_{B(y,\delta)}<t+\gamma_0\delta^2\right)\stackrel{(\ref{GF1})}{\geq} \frac{c_1\delta^{d+2}\cdot m^{(\theta)}_\kappa}{2|y-x|^{d+\alpha}}.
  \end{eqnarray*}}
  Hence, by (\ref{ER2}), we have for any
  $\theta/2>|y-x|\geq 3\delta$,
  \begin{align}\label{eq:hkLowPol}
    p(t,x;s,y)&\geq \int_{B(y,2\delta)}p\big(t,x;t+2\gamma_0\delta^2,z\big) p\big(t+2\gamma_0\delta^2,z;s,y\big)\dif z\notag\\
    &\geq \inf_{z\in B(y,2\delta)}p\big(t+2\gamma_0\delta^2,z; s,y\big)\mP_{t,x}\Big(X_{t+2\gamma_0\delta^2}\in B\big(y,2\delta\big)\Big)\notag\\
    &\geq C(s-t)^{-d/2}\cdot\frac{c_1\delta^{d+2}\cdot m^{(\theta)}_\kappa}{2|y-x|^{d+\alpha}}\geq C'm^{(\theta)}_\kappa\eta_{\alpha,2}(s-t,y-x).
  \end{align}
  The proof is complete by setting $\theta = \infty$ in the above inequality.
\end{proof}

\section{The truncated case}

Unlike  the upper and lower bound in Corollary \ref{cor:2sided} (see also (\ref{eq:paEst})),
in the truncated  case, the heat kernel $p(t, x; s, y)$ decays exponentially as $|x-y|\to 0$,
In this section, we prove Theorem \ref{thm:trunc} by establishing the following two lemmas.

\begin{lemma} \label{L:5.1}
  Under {\bf (H$^a$)}, {\bf (H$^\kappa$)}, {\bf (HU$^\kappa$)} and $b \in \mathbb{K}_2$, for any $T>0$,
  there are constants $C_1, \lambda_1>0$ such that on $\mathbb{D}^T_0$,
  \begin{equation*}
    p(t,x;s,y) \leq C_1\left(\xi_{\lambda_1,0}+ \bar\eta_{\alpha,1/8}\right)(s-t,y-x),
  \end{equation*}
  where $\bar\eta_{\alpha,\lambda}$ is defined by \eqref{Eta}.
\end{lemma}
\begin{proof}
  By (\ref{eqlpe}), we already know that
   \begin{equation}\label{eq:hkUpLoc}
    p(t,x;s,y) \leq C_0\left(\xi_{\lambda_0,0}+ \|\kappa\|_\infty\eta_{\alpha,2}\right)(s-t,y-x)\ \mbox{ on }\mathbb{D}^T_0.
  \end{equation}
  However, the term $\eta_{\alpha,2}(s-t,y-x)$ is too large when $|y-x|$ is large.
  So, we need to establish a proper upper bound for this case.
  We use induction method to show that there is a constant $c_1\geq 1$ such that for all $n \geq 1$,
  \begin{equation}\label{eq:hkUpTemp}
    p(t,x;s,y) \leq \left(\frac{c_1(s-t)}{n}\right)^{n},\quad 0<s-t\leq T,~|y-x|\geq 2n,
  \end{equation}
  First of all, by (\ref{eq:hkUpLoc}), we have for all $s-t\in(0,T]$ and $|y-x|\geq 2$,
  $$
  p(t,x;s,y) \leq C_0\left((s-t)^{-d/2}\e^{-2\lambda_0/(s-t)}+ \|\kappa\|_\infty 2^{-d-\alpha}(s-t)\right)  \leq c_2(s-t).
  $$
   Hence, (\ref{eq:hkUpTemp}) is true for $n=1$ as long as $c_1\geq c_2$.

   \medskip

  Next we assume that \eqref{eq:hkUpTemp} holds for all $n\leq N$. We want to show \eqref{eq:hkUpTemp} for
  $$
  s-t\in(0,T],\ \ |y-x|\geq 2(N+1).
  $$
  Fix such $x,y$ and let $\tau\coloneqq\tau_{B(x,1)}$ be the first exit time of $X$ from ball $B(x,1)$.
  By the L\'{e}vy system of $X$ (see Lemma \ref{lem:XLS}) and {\bf (HU$^\kappa$)}, we have
  $$
    \mathbb{P}_{t,x}\left(X_\tau \in B(x,2)^c\right) = \mathbb{E}_{t,x}\left[\int_t^\tau\!\!\! \int_{B(x,2)^c} \kappa(r,X_r,z-X_r)\dif z\dif r\right] = 0,
  $$
  where we have used the fact that for $r<\tau$ and $z\in B(x,2)^c$,
   $$
   |z-X_r|\geq |z-x|-|X_r-x|> 1.
   $$
  Let $t_n:=t+\frac{n(s-t)}{N+1}$ for $n=0,1,\cdots, N+1$. By the strong Markov property of $X$,
  \begin{align}
    p(t,x;s,y) =&\, \mathbb{E}_{t,x}\left[p(\tau,X_\tau;s,y);  \tau< s\right]\no\\
    =&\, \mathbb{E}_{t,x}\left[p(\tau,X_\tau;s,y);  \tau < s,X_\tau \in B(x,2)\right]\no\\
    =&\, \sum_{n=0}^N\mathbb{E}_{t,x}\left[p(\tau,X_\tau;s,y);  \tau \in[t_n,t_{n+1}),X_\tau \in B(x,2)\right]\no\\
    \leq&\, \sum_{n=0}^N\sup_{(r,z)\in[t_n,t_{n+1})\times B(x,2)}p(r,z;s,y) \cdot \mathbb{P}_{t,x}\left(\tau\leq t_{n+1}\right).\label{KA1}
  \end{align}
  Noting that for all $z \in B(x,2)$ and $|y-x|\geq 2(N+1)$,
  $$
  |y-z|\geq |y-x| - |z-x| \geq 2N,
  $$
  by the induction hypothesis, we have
  \begin{equation}\label{KA2}
    \sup_{(r,z)\in[t_n,t_{n+1})\times B(x,2)}p(r,z;s,y) \leq \left(\frac{c_1(s-t_n)}{N}\right)^N=c^N_1\left(\frac{(N+1-n)(s-t)}{N(N+1)}\right)^N.
  \end{equation}
  On the other hand, by (\ref{eq:firstExitUp2}) with $r=(n+1)(s-t)/(N+1)$ and $\delta = 1$,
  \begin{equation}\label{KA3}
    \mathbb{P}_{t,x}\left(\tau\leq t_{n+1}\right)\leq C \frac{(n+1)(s-t)}{N+1}, \quad n\leq N .
  \end{equation}
Combining \eqref{KA1}-\eqref{KA3}, we get for all $s-t\in (0,T]$ and $|y-x|\geq 2(N+1)$,
  \begin{align*}
    p(t,x;s,y) \leq&\, c^N_1\cdot C\sum_{n=0}^N \left(\frac{(N+1-n)(s-t)}{N(N+1)}\right)^N \frac{(n+1)(s-t)}{N+1}\\
    =&\, c^N_1\cdot C\left(\frac{s-t}{N+1}\right)^{N+1} \sum_{n=1}^{N+1} (N+2-n) \left(\frac{n}{N}\right)^N,
  \end{align*}
  where $C\geq 1$. Since $s \mapsto (N+2-s)(s/N)^N$ is increasing on $[0,N]$, we have
  \begin{align*}
    \sum_{n=1}^{N+1} (N+2-n) \left(\frac{n}{N}\right)^N &= \left(\frac{N+1}{N}\right)^N + 2 + \sum_{n=1}^{N-1} (N+2-n) \left(\frac{n}{N}\right)^N\\
    &\leq \e+2+\int_0^N (N+2-s)(s/N)^N\dif s\\
    &= \e+2+\frac{1}{N^N}\left(\frac{(N+2)N^{N+1}}{N+1} - \frac{N^{N+2}}{N+2}\right)\\
    &= \e+2+\frac{N}{N+1}\cdot\frac{3N+4}{N+2} \leq 10.
  \end{align*}
  Therefore,
  $$
  p(t,x;s,y) \leq c^N_1\cdot 10 C \left(\frac{s-t}{N+1}\right)^{N+1}.
  $$
  Thus   \eqref{eq:hkUpTemp}  is proven for $c_1=c_2\vee (10 C)$.

  \medskip

  Finally, for $|y-x|\geq 2$, choosing $n\in\mN$ so that $2n\leq |y-x| < 2(n+1)$, by (\ref{eq:hkUpTemp}), we have
  \begin{align*}
    p(t,x;s,y) &\leq \left(\frac{c_1(s-t)}{n}\right)^n \preceq \left(\frac{2(n+1)}{n}\cdot\frac{c_1(s-t)}{|y-x|}\right)^{\frac{n}{2(n+1)}\cdot |y-x|}\\
    &\leq \left(\frac{4c_1(s-t)}{|y-x|}\right)^{\frac{|y-x|}{4}}\preceq\left(\frac{s-t}{|y-x|}\right)^{\frac{|y-x|}{8}},
  \end{align*}
  which together with \eqref{eq:hkUpLoc} gives the desired estimate.
\end{proof}

\begin{lemma}\label{thmhkFiniteLow}
  Under {\bf (H$^a$)}, {\bf (H$^\kappa$)}, {\bf (HL$^\kappa$)} and $b \in \mathbb{K}_2$, for any $T>0$,
  there are constants $C_2,\lambda_2>0$ such that on $\mathbb{D}^T_0$,
  \begin{equation*}
    p(t,x;s,y) \geq C_2\left(\xi_{\lambda_2,0}+ \bar\eta_{\alpha,8}\right)(s-t,y-x).
  \end{equation*}
\end{lemma}
\begin{proof}
  First of all, by (\ref{eq:hkLowExp}), we have
  \begin{align}\label{HJ9}
    p(t,x;s,y) \geq C\xi_{\lambda_2,0}(s-t,y-x)\mbox{ on $\mD^T_0$,}
  \end{align}
  and, by {\bf (HL$^\kappa$)} and  (\ref{eq:hkLowPol}), for $s-t\in(0,T]$ and $|y-x|\leq 1/2$,
  \begin{align}\label{eq:hkLowLoc}
  p(t,x;s,y) \geq c_1\eta_{\alpha,2}(s-t,y-x).
    \end{align}
  Thus  it remains to prove this lemma for $s-t\in (0,T]$ and $|y-x| > 1/2$.
  Let $n$ be the least integer greater than $4|y-x|$, that is,
  \begin{align}\label{HJ1}
  2\leq n-1\leq 4|y-x|<n.
  \end{align}
  For $i=0,1,\cdots,n$, let us define
  \begin{equation*}
    x_i = x+(y-x)i/n, \quad B_i:=B(x_i,1/8)\quad\text{and}\quad  t_i = t + (s-t)i/n.
  \end{equation*}
 Noticing that for all $i=0,1,\cdots,n-1$ and $z_i\in B_i$,
  $$
    |z_i-z_{i+1}|\leq |z_i-x_i|+|x_i-x_{i+1}|+|x_{i+1}-z_{i+1}| \leq \tfrac{1}{8}+\tfrac{|y-x|}{n}+\tfrac{1}{8}<\tfrac{1}{2},
  $$
  by (\ref{eq:hkLowLoc}), we have
  \begin{equation*}
    p(t_i,z_i;t_{i+1},z_{i+1}) \geq c_1\eta_{\alpha,2}(t_{i+1}-t_i,z_{i+1}-z_i) \geq  c_1\left(\frac{1}{2}+\sqrt{\frac{T}{3}}\right)^{-(d+\alpha)} \left(\frac{s-t}{n}\right) =:  c_2\left(\frac{s-t}{n}\right).
  \end{equation*}
  Hence, by C-K equation (\ref{CK}) and \eqref{HJ1},
  \begin{align*}
    p(t,x;s,y) \geq& \int_{B_{n-1}}\cdots\int_{B_1} p(t_0,x;t_1,z_1) \cdots p(t_{n-1},z_{n-1};s,y) \dif z_1\cdots \dif z_{n-1}\\
    \geq& ({\rm {Vol}}(B_1))^{n-1} \left(\frac{c_2(s-t)}{n}\right)^{n}
    \geq c_3 \left(\frac{s-t}{|y-x|}\right)^{8|y-x|},
  \end{align*}
 which together with \eqref{eq:hkLowLoc} and \eqref{HJ9} yields the desired estimate.
\end{proof}

Theorem \ref{thm:trunc} follows directly from the above two lemmas.

\section{Appendix}

\subsection{A maximum principle} In this subsection we show a maximum principle for operator $\sL$,
which has been used to show the uniqueness and positivity of heat kernels in this paper.
\bt\label{Le51}
  Let $a(t,x), b(t,x)$ and $\kappa(t,x,z)$ be bounded measurable with matrix $a(t,x)\geq 0$ and $\kappa(t,x,z)\geq 0$.
  For $T>0$, let $u(t,x)\in C_b([0,T)\times\mR^d)$ satisfy the following equation: for Lebesgue almost all $t\in[0,T)$,
  $$
    \p_t u+\sL_tu\leq 0,\ \ \varliminf_{t\uparrow T}u(t,x)\geq 0.
  $$
  Assume that for each $t\in[0,T)$, $x\in\mR^d$ and $j=0,1,2$, the mappings $x\mapsto\nabla^j_x u(t,x)$ and $t\mapsto\nabla^j_x u(t,x)$ are continuous. Then we have
  \begin{align}\label{NM4}
    u(t,x)\geq 0,\ \ (t,x)\in[0,T)\times\mR^d.
  \end{align}
\et
\begin{proof}
  First of all, we assume that for each $(t,x)\in[0,T)\times\mR^d$,
  \begin{align}\label{HG1}
    \p_t u(t,x)+\sL_tu(t,x)\leq \delta<0\ \mbox{ and }\ \lim_{|x|\to\infty} u(t,x)=\infty.
  \end{align}
  Suppose that \eqref{NM4} is not true.
  Since $\lim_{|x|\to\infty} u(t,x)=\infty$ and $\varliminf_{t\uparrow T}u(t,x)\geq 0$,
  there must be a point $(t_0,x_0)\in [0,T)\times\mR^d$ such that
  $$
    u(t_0,x_0)=\inf_{(t,x)\in[0,T)\times\mR^d}u(t,x)<0.
  $$
  Since $x\mapsto \nabla^2 u(t_0,x)$ is continuous and $x_0$ is an infimum point of $x\mapsto u(t_0,x)$, we have
  $$
    \nabla_x u(t_0,x_0)=0 \mbox{ and }(\p_i\p_j u(t_0,x_0))_{ij} \mbox{ is positive definite and symmetric}.
  $$
  Therefore,
  $$
    \sL_s u(t_0,\cdot)(x_0)\geq 0,
  $$
  and by \eqref{HG1},
  \begin{align}
    &u(t,x_0)-u(t_0,x_0)\leq(t-t_0)\delta-\int^t_{t_0}\sL_s u(s,\cdot)(x_0)\dif s\no\\
    &\quad\leq(t-t_0)\delta-\int^t_{t_0}\sL_su(s,\cdot)(x_0)-\sL_su(t_0,\cdot)(x_0)\dif s.\label{UT4}
  \end{align}
  Notice that by definition and the assumptions,
  \begin{align*}
    |\sL_su&(s,\cdot)(x_0)-\sL_su(t_0,\cdot)(x_0)|\\
    &\leq \|a\|_\infty\|\nabla^2_xu(s,x_0)-\nabla^2_xu(t_0,x_0)\|+\|b\|_\infty|\nabla_xu(s,x_0)-\nabla_x u(t_0,x_0)|\\
    &+\|\kappa\|_\infty\int_{|z|\leq 1}\!\left(\int^1_0\|\nabla^2_xu(s,x_0+\theta z)-\nabla^2_xu(t_0,x_0+\theta z)\|\dif\theta\right)|z|^{2-d-\alpha}\dif z\\
    &+\|\kappa\|_\infty\int_{|z|>1}|u(s,x_0+z)-u(s,x_0)-u(t_0,x_0+z)+u(t_0,x_0)|\cdot|z|^{-d-\alpha}\dif z.
  \end{align*}
  Since $s\mapsto \nabla^j_x u(s,x), j=0,1,2$ are continuous, by dividing both sides of \eqref{UT4} by $t-t_0$ and letting $t\downarrow t_0$, we obtain
  \begin{align*}
    0&\leq \delta+\lim_{t\to t_0}\frac{1}{t-t_0}\int^t_{t_0}|\sL_su(s,\cdot)(x_0)-\sL_su(t_0,\cdot)(x_0)|\dif s=\delta<0,
    \end{align*}
  which is impossible. In other words, the infimum is achieved at terminal time $T$, and \eqref{NM4} holds.

  Next, we need to drop the restriction \eqref{HG1}. Let
  $$
    f(x):=(1+|x|^2)^\beta,\ \ \beta\in(0,\alpha/2).
  $$
  For $\eps,\delta>0$, define
  $$
    u_{\eps,\delta}(t,x):=u(t,x)+\delta(T-t)+\eps\e^{-t} f(x).
  $$
  By easy calculations, one sees that for some $C>0$,
  $$
    |\sL_t f(x)|\leq C(1+|x|^\beta),
  $$
  and
  \begin{align*}
    \p_tu_{\delta,\eps}(t,x)+\sL_t u_{\delta,\eps}(t,x)\leq-\delta+\eps\e^{-t}(\sL_t f(x)-f(x))\leq-\delta/2<0,
  \end{align*}
  provided $\eps$ being small enough so that $\eps\e^{-t}(\sL_tf(x)-f(x))<\delta/2$. Moreover, clearly
  $$
    \lim_{x\to\infty}|u_{\eps,\delta}(t,x)|=\infty.
  $$
  Hence, by what we have proved,
  $$
    u_{\eps,\delta}(t,x)\geq 0.
  $$
  By letting $\eps\to 0$ and then $\delta\to 0$, we obtain \eqref{NM4}.
\end{proof}

\subsection{Proof of Theorem \ref{T23}}
$\!$Let $A$ be a $d\times d$ positive definite matrix and $Z_A(x)$ the $d$-dimensional Gaussian density function with covariance matrix $A$, i.e.,
\begin{align}
  Z_A(x):=\frac{\e^{-\<A^{-1}x,x\>/2}}{\sqrt{(2\pi )^d\det(A)}},\label{eqg0}
\end{align}
where $\det(A)$ denotes the determinant of $A$. For $x,y,z\in\mR^d$ and $t<s$, define
$$
  A_{t,s}(y):=\int^s_t a_r(y)\dif r\ \ \mbox{ and }\ \ Z_y(t,x;s,z):=Z_{A_{t,s}(y)}(z-x).
$$
Clearly, for each fixed $y$, $Z_y(t,x;s,z)$ is smooth in $(x, z)$ and Lipschitz continuous in $t$ and $s$.
By definition and \eqref{eqa1}, for each $j\in\mN_0$, there are constants $C_j,\lambda_j>0$ only depending on $d$ and $c_2$ such that for all $x,y,z\in\mR^d$ and $t<s$,
\begin{align}\label{Es44}
 |\nabla^j_xZ_y(t,x;s,z)|\leq C_j\xi_{\lambda_j,-j}(s-t,z-x).
\end{align}
Moreover, it is easy to see that $Z_y(t,x;s,z)$ satisfies the following equation: for fixed $s>0$ and Lebesgue almost all $t\in[0,s]$,
\begin{align}\label{NM1}
  \p_tZ_y(t,x;s,z)+\sL^{a_\cdot(y)}_tZ_y(t,\cdot;s,z)(x)=0,\ \ x,y,z\in\mR^d.
\end{align}
Now, let us define
\begin{align}
  Z_0(t,x;s,y):=Z_y(t,x;s,y).\label{z0}
\end{align}

The following lemma establishes the H\"older continuity of $\nabla^j_xZ_0(t,x;s,y)$ in $y$ for every  $j\geq 0$.

\bl\label{Le52}
  Under {\bf (\bf H$^a$)}, for each $T>0, j\in\mN_0$ and $\beta'\in(0,\beta)$, there exist constants $C_j,\lambda_j>0$ depending on $T, d$, $c_1$ and $c_2$ such that for all $x,y\in\mR^d$ and $0\leq t<s\leq T$,
  \begin{align}
  \begin{split}
    &|\nabla^j_xZ_0(t,x;s,y)-\nabla^j_xZ_0(t,x;s,y')|\\
    &\quad\leq C_j|y-y'|^{\beta'}\Big(\xi_{\lambda_j,-\beta'-j}(s-t,y-x)+\xi_{\lambda_j,-\beta'-j}(s-t,y'-x)\Big).
   \end{split}\label{yy}
  \end{align}
\el
\begin{proof}
  Observe that by the chain rule,
  $$
    \nabla^j Z_A(x)=H_j(A^{-1}, x)Z_A(x),
  $$
  where $H_j(A, x)$ is a vector valued polynomial of $A,x$, and has the following property
  \begin{align}\label{JH3}
    H_j(\ell^2 A, \ell^{-1} x)=\ell^jH_j(A, x),\ \ \ell>0.
  \end{align}
Thus  by the definition of $Z_0$, we have
  \begin{align}\label{KH1}
    \nabla^j_x Z_0(t,x;s,y)=H_j(A^{-1}_{t,s}(y), y-x)Z_{A_{t,s}(y)}(y-x).
  \end{align}
  Due to \eqref{eqa1}, there is a constant $C>0$ such that for all $0<t<s$,
  \begin{align}\label{KH2}
    C^{-1}(s-t)^d\leq \det(A_{t,s}(y))\leq C(s-t)^d,
  \end{align}
  and
  \begin{align}\label{KH3}
    \<A^{-1}_{t,s}(y)z,z\>\geq \lambda|z|^2/(s-t),\ \ |A^{-1}_{t,s}(y)z|\leq C|z|/(s-t).
  \end{align}
  Let us denote the left hand side of \eqref{yy} by $\sI$. Let $\delta>0$ be a small number, whose value will be determined below. We consider two cases:
  \medskip
  \\
  (Case $|y-y'|>\delta\wedge\sqrt{s-t}$).  In this case, by \eqref{Es44} we have
  \begin{align*}
    \sI&\preceq \xi_{\lambda_j,-j}(s-t,y-x)+\xi_{\lambda_j,-j}(s-t,y'-x)\\
    &\preceq|y-y'|^{\beta'}\Big(\xi_{\lambda_j,-\beta'-j}(s-t,y-x)+\xi_{\lambda_j,-\beta'-j}(s-t,y'-x)\Big).
  \end{align*}
  (Case $|y-y'|\leq\delta\wedge\sqrt{s-t}$).  In this case,  by \eqref{KH1} we have
  \begin{align*}
    \sI&\leq\frac{|H_j(A^{-1}_{t,s}(y), y-x)|}{\sqrt{\det(A_{t,s}(y))}}\Big|\e^{-\<A^{-1}_{t,s}(y)(y-x),y-x\>}-\e^{-\<A^{-1}_{t,s}(y')(y'-x),y'-x\>}\Big|\\
    &+\Bigg|\frac{H_j(A^{-1}_{t,s}(y), y-x)}{\sqrt{\det(A_{t,s}(y))}}-\frac{H_j(A^{-1}_{t,s}(y'), y'-x)}{\sqrt{\det(A_{t,s}(y'))}}\Bigg|\e^{-\<A^{-1}_{t,s}(y')(y'-x),y'-x\>}=:\sI_1+\sI_2.
  \end{align*}
  For $\sI_1$,
  notice that
  \begin{align*}
    R&:=|\<A^{-1}_{t,s}(y)(y-x),y-x\>-\<A^{-1}_{t,s}(y')(y'-x),y'-x\>|\\
    &\leq C_0(s-t)^{-1}(|y-y'|^\beta |y-x|^2+|y-y'| (|y-x|+|y'-x|))\\
    &\leq (C_0\delta^\beta+\tfrac{\lambda}{4})|y-x|^2/(s-t)+C,
  \end{align*}
  where $\lambda$ is the same as in \eqref{KH3}, and by \eqref{JH3} and \eqref{KH3},
  $$
    |H_j(A^{-1}_{t,s}(y), z)|\leq C(s-t)^{-j/2} h_j(|z|/\sqrt{s-t}),
  $$
  where $h_j$ is a $j$-order real polynomial. Thus  in view of $\e^R-1\leq R\e^R$ for all $R>0$,
  by \eqref{KH2} and choosing $\delta$ small enough so that $C_0\delta^\beta=\frac{\lambda}{4}$, we have
  \begin{align*}
    \sI_1&\preceq (s-t)^{-j/2}h_j(|y-x|/\sqrt{s-t})\cdot\xi_{\lambda,0}(s-t,y-x)\cdot (\e^R-1)\\
    &\preceq (s-t)^{-j/2}h_j(|y-x|/\sqrt{s-t})\cdot\xi_{\lambda/2,0}(s-t,y-x)\\
    &\times (s-t)^{-1}\Big(|y-y'|^\beta |y-x|^2+|y-y'|(|y-x|+|y'-x|)\Big)\\
    &\preceq |y-y'|^{\beta'}\xi_{\lambda/3,-\beta'-j}(s-t,y-x),
  \end{align*}
  where we have used the fact that $|y-y'|\leq\sqrt{s-t}$. Similarly, one can show
  \begin{align*}
    \sI_2\preceq|y-y'|^{\beta'}\xi_{\lambda/3,-\beta'-j}(s-t,y'-x).
  \end{align*}
  Combining the above calculations, we get \eqref{yy}.
\end{proof}

The classical Levi's freezing coefficients method suggests that the heat kernel $Z$ of $\sL^a_t$ takes the following form:
\begin{align}
  Z(t,x;s,y)=Z_0(t,x;s,y)+\int^s_t\!\!\!\int_{\mR^d}Z_0(t,x;r,z)Q(r,z;s,y)\dif z\dif r,\label{eqG2}
\end{align}
where $Q$ satisfies the following integro-equation:
\begin{align} \label{eqG3}
  Q(t,x;s,y)=Q_0(t,x;s,y)+\int^s_t\!\!\!\int_{\mR^d}Q_0(t,x;r,z)Q(r,z;s,y)\dif z\dif r
\end{align}
with
\begin{align}
   Q_0(t,x;s,y):=\big(\sL^a_t- \sL^{a_\cdot(y)}_t \big)Z_0(t,\cdot;s,y)(x)
  =\sum_{i,j=1}^d \Big( a^{ij}_t(x)-a^{ij}_t(y)\Big)\p^2_{x_ix_j}Z_0(t,x;s,y).
\end{align}

Let us first solve the integral equation \eqref{eqG3}.
\bl\label{Le53}
  For each $n\in\mN$, define $Q_n(t,x;s,y)$ recursively by
  \begin{align}
    Q_n(t,x;s,y):=\int^s_t\!\!\!\int_{\mR^d}Q_0(t,x;r,z)Q_{n-1}(r,z;s,y)\dif z\dif r.\label{HH3}
  \end{align}
  Under {\bf (H$^a$)}, the series $Q(t,x;s,y):=\sum_{n=0}^\infty Q_n(t,x;s,y)$ is locally uniformly and absolutely convergent, and solves the integral equation \eqref{eqG3}. Moreover,
  \begin{align}
    Q(t,x;s,y)=Q_0(t,x;s,y)+\int^s_t\!\!\!\int_{\mR^d}Q(t,x;r,z)Q_0(r,z;s,y)\dif z\dif r,\label{HH32}
  \end{align}
  and for any $T>0$, on $\mD^T_0$, we have
  \begin{align}\label{Es45}
    |Q(t,x;s,y)|\leq C\xi_{\lambda,\beta-2}(s-t;y-x),
  \end{align}
  and, for any $\beta'\in(0,\beta)$,
  \begin{align}
    &|Q(t,x_1;s,y)-Q(t,x_2;s,y)|\leq C|x_1-x_2|^{\beta'}\sum_{i=1,2}\xi_{\lambda,\beta-\beta'-2}(s-t;y-x_i),\label{RT29}\\
    &|Q(t,x;s,y_1)-Q(t,x;s,y_2)|\leq C|y_1-y_2|^{\beta'}\sum_{i=1,2}\xi_{\lambda,\beta-\beta'-2}(s-t,y_i-x).\label{Qy}
    \end{align}
\el
\begin{proof}
  (i) First of all, by \eqref{Es44} and {\bf (H$^a$)}, there exist $C_0, \lambda>0$ such that
  \begin{align}
    |Q_0(t,x;s,y)|\leq C_0 \xi_{\lambda,\beta-2}(s-t,y-x).\label{HH2}
  \end{align}
We use induction to prove that for all $n\in\mN$, $0\leq t<s$ and $x,y\in\mR^d$,
  \begin{align}
    |Q_{n-1}(t,x;s,y)|\leq \frac{(C_0\Gamma(\beta/2))^{n}}{(\lambda/\pi)^{d(n-1)/2}\Gamma(n\beta/2)}\xi_{\lambda,n\beta-2}(s-t,y-x),\label{HH1}
  \end{align}
 where $\Gamma$ is the usual Gamma function, and
  \begin{align}\label{HH11}
    Q_n(t,x;s,y)=\int^s_t\!\!\!\int_{\mR^d}Q_{n-1}(t,x;r,z)Q_0(r,z;s,y)\dif z\dif r.
  \end{align}
  Suppose that \eqref{HH1} and \eqref{HH11} hold for some $n\in\mN$. Let
  $\gamma_n:=\frac{(C_0\Gamma(\beta/2))^{n}}{(\lambda/\pi)^{d(n-1)/2}\Gamma(n\beta/2)}$. By \eqref{HH2}, \eqref{HH1} and \eqref{CKE}, we have
  \begin{align*}
    |Q_n(t,x;s,y)|&\leq C_0\gamma_n\int^s_t(r-t)^{\frac{\beta}{2}-1}(s-r)^{\frac{n\beta}{2}-1} \left(\int_{\mR^d}\xi_{\lambda,0}(r-t,z-x)\xi_{\lambda,0}(s-r,y-z)\dif z\right)\dif r\\
    &= C_0(\pi\lambda^{-1})^{d/2}\gamma_n\xi_{\lambda,0}(s-t,y-x)\int^s_t(r-t)^{\frac{\beta}{2}-1}(s-r)^{\frac{n\beta}{2}-1}\dif r\\
    &=C_0(\pi\lambda^{-1})^{d/2}\gamma_n\xi_{\lambda,(n+1)\beta-2}(s-t,y-x)\cB(\tfrac{n\beta}{2},\tfrac{\beta}{2}) =\gamma_{n+1}\xi_{\lambda,(n+1)\beta-2}(s-t,y-x),
  \end{align*}
  where in the last step we have used $\cB(\tfrac{n\beta}{2},\tfrac{\beta}{2}) = \Gamma(\frac{\beta}{2}) \Gamma(\frac{n\beta}{2})/\Gamma(\frac{(n+1)\beta}{2})$. Moreover, by the induction hypothesis and Fubini's theorem, we have
  \begin{align*}
    Q_{n+1}(t,x;s,y)&=\int^s_t\!\!\!\int_{\mR^d}Q_0(t,x;r,z)\int^s_r\!\!\!\int_{\mR^d}Q_{n-1}(r,z;r',z')Q_0(r',z';s,y)\dif z'\dif r'\dif z\dif r\\
    &=\int^s_t\!\!\!\int_{\mR^d}\!\!\int^{r'}_t\!\!\!\!\int_{\mR^d}Q_0(t,x;r,z)Q_{n-1}(r,z;r',z')\dif z\dif rQ_0(r',z';s,y)\dif z'\dif r'\\
    &=\int^s_t\!\!\!\int_{\mR^d}Q_n(t,x;r',z')Q_0(r',z';s,y)\dif z'\dif r'.
  \end{align*}
  Thus  by \eqref{HH1}, the series $Q=\sum_{n=0}^\infty Q_n$ is locally uniformly and absolutely convergent, and solves (\ref{eqG3}) and \eqref{HH32}. Moreover, we also have the estimate (\ref{Es45}).
  \\
  \vbox{~}
  (ii) Next, we prove (\ref{RT29}). We first show that
  \begin{align}\label{rt11}
    |Q_0(t,x_1;s,y)-Q_0(t,x_2;s,y)|\preceq |x_1-x_2|^{\beta'}\sum_{i=1,2}\xi_{\lambda,\beta-\beta'-2}(s-t;y-x_i).
  \end{align}
  If $|x_1-x_2|>\sqrt{s-t}$, then by \eqref{HH2}, one sees that \eqref{rt11} holds. If $|x_1-x_2|\leq\sqrt{s-t}$,
  then by the definition of $Q_0$ and \eqref{Es44}, we have for some $\theta\in[0,1]$,
  \begin{align*}
    &|Q_0(t,x_1;s,y)-Q_0(t,x_2;s,y)|\leq |a_t(x_1)-a_t(x_2)|\cdot|\nabla^2_xZ_0(t,x_1;s,y)|\no\\
    &\qquad\qquad+|a_t(x_2)-a_t(y)|\cdot|\nabla^2_xZ_0(t,x_1;s,y)-\nabla^2_xZ_0(t,x_2;s,y)|\no\\
    &\quad\preceq |x_1-x_2|^\beta\xi_{\lambda_2,-2}(s-t,y-x_1)+|y-x_2|^\beta|x_1-x_2|\xi_{\lambda_3,-3}(s-t,y-x_2-\theta(x_1-x_2))\no\\
    &\quad\preceq |x_1-x_2|^{\beta'}\xi_{\lambda,\beta-\beta'-2}(s-t,y-x_1)+|y-x_2|^\beta|x_1-x_2|^{\beta'}\xi_{\lambda,-\beta'-2}(s-t,y-x_2) \no\\
    &\quad\preceq|x_1-x_2|^{\beta'} \sum_{i=1,2}\xi_{\lambda,\beta-\beta'-2}(s-t,y-x_i).
  \end{align*}
  Write
  $$
    G(t,x;s,y):=\int^s_t\!\!\!\int_{\mR^d}Q_0(t,x;r,z)Q(r,z;s,y)\dif z\dif r.
  $$
  By (\ref{Es45}) and \eqref{rt11}, we have
  \begin{align}
    &\quad|G(t,x_1;s,y)-G(t,x_2;s,y)|\no\\
    &\preceq |x_1-x_2|^{\beta'}\int^s_t\!\!\!\int_{\mR^d}  \xi_{\lambda,\beta-2}(s-r,y-z)\sum_{i=1,2}\xi_{\lambda,\beta-\beta'-2}(r-t,z-x_i)\dif z\dif r\no\\
    &\preceq|x_1-x_2|^{\beta'}\sum_{i=1,2}\xi_{\lambda,\beta-\beta'-2}(s-t,y-x_i).\label{qq}
  \end{align}
  Combining this with (\ref{rt11}), we obtain \eqref{RT29}.
  \\
  \vbox{~}
  (iii) We now show that on $\mD^T_0$,
  \begin{align}
    |Q_0(t,x;s,y_1)-Q_0(t,x;s,y_2)|\preceq |y_1-y_2|^{\beta'}\sum_{i=1,2}\xi_{\lambda,\beta-\beta'-2}(s-t,y_i-x).\label{qyy}
  \end{align}
  If $|y_1-y_2|>\sqrt{s-t}$, then by \eqref{HH2} again, one sees that \eqref{qyy} holds. If $|y_1-y_2|\leq\sqrt{s-t}$, we write
  \begin{align*}
    &|Q_0(t,x;s,y_1)-Q_0(t,x;s,y_2)|\leq\big|a_t(y_1)-a_t(y_2)\big|\Big(|\nabla^2_xZ_0(t,x;s,y_1)|+|\nabla^2_xZ_0(t,x;s,y_2)|\Big)\\
    &\quad+\Big(|a_t(x)-a_t(y_1)|\wedge|a_t(x)-a_t(y_2)|\Big)\big|\nabla^2_xZ_0(t,x;s,y_1)-\nabla^2_xZ_0(t,x;s,y_2)\big|=:I_1+I_2.
  \end{align*}
  For $I_1$, we have
  \begin{align*}
    I_1\preceq |y_1-y_2|^{\beta}\sum_{i=1,2}\xi_{\lambda,-2}(s-t,y_i-x)\preceq|y_1-y_2|^{\beta'}\sum_{i=1,2}\xi_{\lambda,\beta-\beta'-2}(s-t,y_i-x).
  \end{align*}
  For $I_2$, by (\ref{yy}) we have
  \begin{align*}
    I_2&\preceq |y_1-y_2|^{\beta'} \big(|y_1-x|^{\beta}\wedge|y_2-x|^{\beta}\big)\sum_{i=1,2}\xi_{\lambda,-\beta'-2}(s-t,y_i-x)\\
    &\preceq|y_1-y_2|^{\beta'}\sum_{i=1,2}\xi_{\lambda,\beta-\beta'-2}(s-t,y_i-x).
  \end{align*}
  Thus  \eqref{qyy} is proven. Using \eqref{qyy} and as in (ii), we have \eqref{Qy}.
\end{proof}

The following result can be derived in the same way as that in the proof of \cite[Theorem 6, p.13]{Friedman.1964.346} and so its proof is omitted here.

\bl \label{Le54}
  Let $f:\mR_+\times\mR^d\to\mR$ be a measurable function and satisfy that for some $\gamma\in(0,1)$
  $$
    |f(t,x)-f(t,x')|\leq C|x-x'|^\gamma.
  $$
  Fix $s>0$ and define $V(t,x):=\int^s_t\!\int_{\mR^d}Z_0(t,x;r,z)f(r,z)\dif z\dif r.$ Then we have the following conclusions:
  \begin{enumerate}[(i)]
    \item The mapping $(t,x)\mapsto\nabla^2_xV(t,x)$ is continuous on $[0,s)\times\mR^d$ and
        \begin{align}\label{VV}
          \nabla^2_xV(t,x)=\int^s_t\!\!\!\int_{\mR^d}\nabla^2_x Z_0(t,x;r,z)f(r,z)\dif z\dif r,
        \end{align}
     where the integral in the right hand side is understood in the sense of double integral.

    \item For Lebesgue-almost all $t\in[0,s]$ and $x\in\mR^d$,
        \begin{align}\label{NM2}
          \p_tV(t,x)+\sL^{a_\cdot(y)}_t V(t,x)+f(t,x)=0.
        \end{align}
  \end{enumerate}
\el

Now we are ready to give

\begin{proof}[Proof of Theorem \ref{T23}]
  We need to check $Z(t,x;s,y)$ defined by \eqref{eqG2} has all the stated properties.
   Let
  $$
    \Phi(t,x;s,y):=\int^s_t\!\!\!\int_{\mR^d}Z_0(t,x;r,z)Q(r,z;s,y)\dif z\dif r.
  $$
  (1) By \eqref{Es44} and \eqref{Es45}, we have
  \begin{align}
    |\Phi(t,x;s,y)|&\preceq \int^s_t\!\!\!\int_{\mR^d}\xi_{\lambda,0}(r-t,z-x)\xi_{\lambda,\beta-2}(s-r,y-z)\dif z\dif r\no\\
    &\preceq \left(\int^s_t(s-r)^{\frac{\beta-1}{2}}\dif r\right)\xi_{\lambda,0}(s-t,y-x)\preceq \xi_{\lambda,\beta+1}(s-t,y-x),\label{11}
  \end{align}
  which together with \eqref{Es44} gives the upper bound of \eqref{ET41}.

  In view of \eqref{RT29}, by \eqref{VV}, for $j=1,2$, we may write
  \begin{align}
    \nabla^j_x\Phi(t,x;s,y)= & \int^s_{\frac{s+t}{2}}\!\int_{\mR^d}\nabla^j_xZ_0(t,x;r,z)Q(r,z;s,y)\dif z\dif r\no\\
    &+\int^{\frac{s+t}{2}}_t\!\!\!\int_{\mR^d}\nabla^j_xZ_0(t,x;r,z)(Q(r,z;s,y)-Q(r,x;s,y))\dif z\dif r\no\\
    &+\int^{\frac{s+t}{2}}_t\!\!\!\left(\int_{\mR^d}\nabla^j_xZ_0(t,x;r,z)\dif z\right)Q(r,x;s,y)\dif r=:I_1+I_2+I_3.\label{2}
  \end{align}
  For $I_1$, by (\ref{Es44}) and (\ref{Es45}), we have
  \begin{align*}
    |I_1|&\preceq \int^s_{\frac{s+t}{2}}\!\int_{\mR^d}|\nabla^j_xZ_0(t,x;r,z)|\cdot|Q(r,z;s,y)|\dif z\dif r\\
     &\preceq \int^s_{\frac{s+t}{2}}\!\int_{\mR^d}\xi_{\lambda_j,-j}(r-t,z-x)\xi_{\lambda_j,\beta-2}(s-r,y-z)\dif z\dif r\\
     &\preceq \left(\int^s_{\frac{s+t}{2}}(r-t)^{-\frac{j}{2}}(s-r)^{\frac{\beta}{2}-1}\dif r\right) \xi_{\lambda_j,0}(s-t,y-x) \preceq \xi_{\lambda_j,\beta-j}(s-t,y-x).
  \end{align*}
  For $I_2$, by (\ref{Es44}) and (\ref{RT29}), we have
  \begin{align*}
    |I_2|&\preceq \int^{\frac{s+t}{2}}_t\!\!\!\int_{\mR^d}\xi_{\lambda_j,-j}(r-t,z-x)|z-x|^{\beta'} \Big(\xi_{\lambda_j,\beta-\beta'-2}(s-r,y-z)+\xi_{\lambda_j,\beta-\beta'-2}(s-r,y-x)\Big)\dif z\dif r\\
    &\preceq \int^{\frac{s+t}{2}}_t\!\!\!\int_{\mR^d}\xi_{\lambda_j/2,\beta'-j}(r-t,z-x) \Big(\xi_{\lambda_j/2,\beta-\beta'-2}(s-r,y-z)+\xi_{\lambda_j/2,\beta-\beta'-2}(s-r,y-x)\Big)\dif z\dif r\\
    &\preceq \left(\int^{\frac{s+t}{2}}_t(r-t)^{\frac{\beta'-j}{2}}\dif r\right)\xi_{\lambda_j/2,\beta-\beta'-2}(s-t,y-x)\preceq \xi_{\lambda_j/2,\beta-j}(s-t,y-x).
  \end{align*}
  For $I_3$, noticing that for each $y\in\mR^d$,
  $$
    \int_{\mR^d}\nabla^j_x Z_y(t,x;r,z)\dif z=\nabla^j_x \int_{\mR^d}Z_{A_{t,r}(y)}(r-t,z-x)\dif z=0,
  $$
  by calculations as in Lemma \ref{Le52}, we have
  \begin{align}
    \left|\int_{\mR^d}\nabla^j_x Z_z(t,x;r,z)\dif z\right|&=\left|\int_{\mR^d}\nabla^j_x Z_{A_{t,r}(z)}(r-t,z-x)-\nabla^j_x Z_{A_{t,r}(x)}(r-t,z-x)\dif z\right|\no\\
    &\preceq\int_{\mR^d}|z-x|^\beta\xi_{\lambda_j,-j}(r-t,z-x)\dif z\preceq (r-t)^{\frac{\beta-j}{2}}.\label{00}
  \end{align}
  Therefore,
  \begin{align*}
    |I_3|\preceq\int^{\frac{s+t}{2}}_t(r-t)^{\frac{\beta-j}{2}}\xi_{\lambda_j,\beta-2}(r-s,y-x)\dif r\preceq\xi_{\lambda_j,\beta-j}(s-t;y-x).
  \end{align*}
  Combining the above calculations, we obtain
  $$
    |\nabla^j_x\Phi(t,x;s,y)|\preceq \xi_{\lambda_j,\beta-j}(s-t;y-x),
  $$
  which together with \eqref{Es44} yields \eqref{eq21}.
  \medskip
  \\
  (2) By \eqref{Es44}, \eqref{Qy} and \eqref{UY30}, we have for $j=0,1$,
  \begin{align*}
    &|\nabla^j_x\Phi(t,x;s,y)-\nabla^j_x\Phi(t,x;s,y')|\\
    &\preceq|y-y'|^{\beta'}\int^s_t\!\!\!\int_{\mR^d}\xi_{\lambda,-j}(r-t,z-x)\Big(\xi_{\lambda,\beta-\beta'-2}(s-r,y'-z)+\xi_{\lambda,\beta-\beta'-2}(s-r,y-z)\Big)\dif z\dif r\\
    &\preceq|y-y'|^{\beta'}\Big(\xi_{\lambda,\beta-\beta'-j}(s-t,y-x)+\xi_{\lambda,\beta-\beta'-j}(s-t,y'-x)\Big),
  \end{align*}
  which together with \eqref{eqG2} and (\ref{yy}) implies (\ref{eq202}).
  \medskip  \\
  (3) In view of (\ref{11}), it suffices to show that
  \begin{align}\label{Lim9}
    \lim_{|t-s|\to 0}\sup_{x\in\mR^d}\left|\int_{\mR^d}Z_0(t,x;s,y)f(y)\dif y-f(x)\right|=0.
  \end{align}
  Notice that (see \eqref{00})
  \begin{align*}
    \left|\int_{\mR^d}Z_0(t,x;s,y)\dif y-1\right|
    &=\left|\int_{\mR^d}Z_{A_{t,s}(y)}(t,x;s,y)-Z_{A_{t,s}(x)}(t,x;s,y)\dif y\right|\\
    &\preceq \int_{\mR^d}|y-x|^{\beta}\xi_{\lambda,0}(s-t,y-x)\dif y\preceq (s-t)^{\frac{\beta}{2}}.
  \end{align*}
  Thus  to prove \eqref{Lim9}, it reduces to show
  \begin{align*}
    \lim_{|t-s|\to 0}\sup_{x\in\mR^d}\left|\int_{\mR^d}Z_0(t,x;s,y)(f(y)-f(x))\dif y\right|=0.
  \end{align*}
  Since $f$ is uniformly continuous, for any $\eps>0$, there exists a $\delta>0$ such that for all $|y-x|\leq\delta$,
  $$
    |f(y)-f(x)|\leq \eps.
  $$
  Therefore, by \eqref{Es44}, we have
  \begin{align*}
    &\left|\int_{\mR^d}Z_0(t,x;s,y)(f(y)-f(x))\dif y\right| \leq\left(\int_{|y-x|\leq\delta}+\int_{|y-x|>\delta}\right)Z_0(t,x;s,y)|f(y)-f(x)|\dif y\\
    &\qquad\preceq\eps\int_{|y-x|\leq \delta}Z_0(t,x;s,y)\dif y+2\|f\|_\infty\int_{|y-x|>\delta}Z_0(t,x;s,y)\dif y\\
    &\qquad\preceq\eps\int_{|y-x|\leq \delta}\xi_{\lambda,0}(s-t,y-x)\dif y+2\|f\|_\infty\int_{|y-x|>\delta}\xi_{\lambda,0}(s-t,y-x)\dif y\\
    &\qquad\preceq\eps\int_{\mR^d}\xi_{\lambda,0}(s-t,y-x)\dif y+2\|f\|_\infty(s-t)^{\alpha/2}\int_{|y-x|>\delta}|y-x|^{-d-\alpha}\dif y\\
    &\qquad\preceq\eps+2\|f\|_\infty(s-t)^{\alpha/2}\delta^{-\alpha}.
  \end{align*}
  Letting $|t-s|\to 0$ and then $\eps\to 0$, we get the desired limit.
  \medskip\\
  (4) and (5). By \eqref{NM1} and \eqref{NM2}, it is easy to see that for fixed $s>0$ and Lebesgue almost all $t\in[0,s]$,
  \begin{align*}
    \p_tZ(t,x;s,y)+\sL^{a}_tZ(t,\cdot;s,y)(x)=0,\ \ x,y\in\mR^d,
  \end{align*}
  that is, equation \eqref{e:2.15} holds.
  In particular, if we let $P^{(Z)}_{t,s}f(x):=\int_{\mR^d}Z(t,x;s,y)f(y)\dif y$, then for
  bounded continuous function $f$ on $\mR^d$,
  $(t,x)\mapsto\nabla^2_xP^{(Z)}_{t,s}f(x)$ is continuous and
  $$
    \p_tP^{(Z)}_{t,s}f(x)+\sL^{a}_tP^{(Z)}_{t,s}f(x)=0,\ \
    \lim_{t\uparrow r}P^{(Z)}_{t,s}f(x)=P^{(Z)}_{r,s}f(x) \ \hbox{for } r\in (0, s].
  $$
  On the other hand, since $t\mapsto P^{(Z)}_{t,r}P^{(Z)}_{r,s}f(x)$ satisfies the same equation with the same final value, by Theorem \ref{Le51},
  we get (\ref{CK1}). Moreover, if we take $f\equiv1$, then we get (\ref{z11}). The same reason yields the non-negativity of $Z(t,x;s,y)$.
  \medskip\\
  (6) Fix $s>0$. For $f\in C^2_b(\mR^d)$, set
  $$
    u(t,x):=f(x)+\int_t^sP^{(Z)}_{t,r}\sL_r^af(x)\dif r,\ \ t\in[0,s].
  $$
  Then we have
  \begin{align*}
    \sL^a_t u(t,x)=\sL^a_tf(x)+\int_t^s\sL^a_tP^{(Z)}_{t,r}\sL_r^af(x)\dif r.
  \end{align*}
  Integrating both sides  from $t_0$ to $s$ with respect to $t$ and by Fubini's theorem, we obtain
  \begin{align*}
    \int^s_{t_0}\sL^a_t u(t,x)\dif t&=\int^s_{t_0}\sL^a_tf(x)\dif t+\int^s_{t_0}\!\!\int_t^s\sL^a_tP^{(Z)}_{t,r}\sL_r^af(x)\dif r\dif t\\
    &=\int^s_{t_0}\sL^a_tf(x)\dif t+\int^s_{t_0}\!\!\int_{t_0}^r \sL^a_tP^{(Z)}_{t,r}\sL_r^af(x)\dif t\dif r\\
    &=\int^s_{t_0}P^{(Z)}_{t_0,r}\sL_r^af(x)\dif r=u(t_0,x)-f(x).
  \end{align*}
  In particular, for almost all $t\in[0,s]$ and $x\in\mR^d$,
  $$
    \p_tu(t,x)+\sL_t^au(t,x)=0,\  \ \lim_{t\uparrow s}u(t,x)=f(x).
  $$
  Using Theorem \ref{Le51} once again (uniqueness), we obtain
  $$
    u(t,x)= P^{(Z)}_{t,s}f(x).
  $$
   (7) The upper bound estimate has been shown in (1). We only need to prove the lower bound estimate. By definition, one sees that for $|x-y|\leq (s-t)^{1/2}$,
  $$
    Z_0(t,x;s,y)\geq C_0(s-t)^{-d/2},
  $$
  and
  \begin{align}
    Z(t,x;s,y)\geq Z_0(t,x;s,y)-|\Phi(t,x;s,y)|\geq C_0(s-t)^{-d/2}-C_1(s-t)^{(\beta+1-d)/2},
  \end{align}
  which has a lower bound $\frac{C_0}{2}(s-t)^{-d/2}$ provided $C_1(s-t)^{(\beta+1)/2}\leq \frac{C_0}{2}$.
  Since $Z(t,x;s,y)$ is non-negative, such an on-diagonal lower bound estimate together with C-K equation and a standard chain argument yields the lower bound estimate \eqref{ET41} (see \cite{FabesStroock.1986.ARMA327} or the proof of Lemma \ref{thmhkFiniteLow}).
  \medskip\\
  (8) Finally, we need to show the uniqueness of $Z(t, x; s, y)$. Suppose that $\widetilde Z(t, x; s, y)$ is another kernel that solves
 \eqref{e:2.15} and has property \eqref{eq21}. Then for
  bounded continuous function $f$ on $\mR^d$ and $s>0$,
  $w(t, x):=\int_{\mR^d} \widetilde Z(t,x;s,y)f(y)\dif y $ is continuous and
  $$
    \p_t w(t, x) +\sL^{a}_t w(t, x)=0,\ \ \lim_{t\uparrow s} w(t, x)=f(x).
  $$
 It follows from Theorem \ref{Le51}, $w(t, x)= P_{t, s}^{(Z)} f(x)$
 and so $\widetilde Z(t, x; s, y)=Z(t, x; s, y)$. The proof of Theorem \ref{T23} is now complete.
\end{proof}

\bigskip

 \def\cprime{$'$}

\end{document}